\documentclass[12pt,oneside,english]{amsart}
\usepackage[T1]{fontenc}
\usepackage[latin9]{inputenc}
\usepackage{geometry}
\usepackage{amsbsy,verbatim}
\usepackage{amstext,txfonts}
\usepackage{amsthm}
\usepackage{amssymb}
\usepackage{scalerel,stackengine}
\usepackage{enumitem}

\DeclareMathOperator{\E}{\mathbb{E}}

\DeclareMathOperator{\LL}{\mathcal{L}}
\DeclareMathOperator{\g}{\gamma}
\DeclareMathOperator{\codim}{\text{codim}}

\usepackage{stmaryrd}
\usepackage{comment}

\stackMath
\newcommand\reallywidehat[1]{%
\savestack{\tmpbox}{\stretchto{%
  \scaleto{%
    \scalerel*[\widthof{\ensuremath{#1}}]{\kern-.6pt\bigwedge\kern-.6pt}%
    {\rule[-\textheight/2]{1ex}{\textheight}}
  }{\textheight}%
}{0.5ex}}%
\stackon[1pt]{#1}{\tmpbox}%
}

\makeatletter
\numberwithin{equation}{section}
\numberwithin{figure}{section}
\theoremstyle{plain}
\newtheorem{thm}{\protect\theoremname}
\theoremstyle{plain}
\newtheorem{cor}[thm]{\protect\corollaryname}
\theoremstyle{plain}
\newtheorem{lem}[thm]{\protect\lemmaname}
\ifx\proof\undefined
\newenvironment{proof}[1][\protect\proofname]{\par
	\normalfont\topsep6\p@\@plus6\p@\relax
	\trivlist
	\itemindent\parindent
	\item[\hskip\labelsep\scshape #1]\ignorespaces
}{%
	\endtrivlist\@endpefalse
}
\providecommand{\proofname}{Proof}
\fi
\theoremstyle{plain}
\newtheorem{prop}[thm]{\protect\propositionname}
\theoremstyle{plain}

\theoremstyle{plain}
\newtheorem*{prop*}{\protect\propositionname}
\theoremstyle{definition}

\theoremstyle{plain}

\theoremstyle{remark}
\newtheorem*{lem*}{\protect\lemmaname}
\theoremstyle{remark}

\theoremstyle{plain}

\newtheorem{definition}[thm]{Definition}

\usepackage{babel}

\makeatother

\usepackage{babel}
\providecommand{\claimname}{Claim}
\providecommand{\corollaryname}{Corollary}
\providecommand{\definitionname}{Definition}
\providecommand{\remarkname}{Remark}
\providecommand{\lemmaname}{Lemma}
\providecommand{\propositionname}{Proposition}
\providecommand{\theoremname}{Theorem}
\providecommand{\hypothesisname}{Hypothesis}
\newcommand{\remove}[1]{}

\usepackage[hidelinks]{hyperref}

\begin{document}

\global\long\def\connected{\text{highly connected}}%
 
\global\long\def\f{\mathcal{F}}%
 
\global\long\def\a{\mathcal{A}}%
 
\global\long\def\pn{\mathcal{P}\left(\left[n\right]\right)}%
 

\global\long\def\Hom{\mathrm{Hom}}%
 
\global\long\def\l{\mathcal{L}}%
 
\global\long\def\s{\mathcal{S}}%
 
\global\long\def\j{\mathcal{J}}%
 
\global\long\def\d{\mathcal{D}}%
 
\global\long\def\Cay{\mathrm{Cay}}%

\global\long\def\OPT{\mathrm{OPT}}
 
\global\long\def\Image{\mathrm{Im}}%

\global\long\def\supp{\mathrm{supp}}
 
\global\long\def\GL{\mathrm{GL}}%
 
\global\long\def\SL{\mathrm{SL}}%
 
\global\long\def\Inf{}%
 
\global\long\def\Id{\textrm{Id}}%
 
\global\long\def\Tr{\mathrm{Tr}}%
 
\global\long\def\sgn{\textrm{sgn}}%
 
\global\long\def\p{\mathcal{P}}%
 
\global\long\def\h{\mathcal{H}}%
 
\global\long\def\N{\mathbb{N}}%
 
\global\long\def\a{\mathcal{A}}%
 
\global\long\def\b{\mathcal{B}}%
 
\global\long\def\c{\mathcal{C}}%
 
\global\long\def\E{\mathbb{E}}%
 
\global\long\def\x{\mathbf{x}}%
 
\global\long\def\y{\mathbf{y}}%
 
\global\long\def\z{\mathbf{z}}%
 
\global\long\def\c{\mathcal{C}}%
 
\global\long\def\av{\mathsf{A}}%
 
\global\long\def\chop{\mathrm{Chop}}%
 
\global\long\def\stab{\mathrm{Stab}}%
 
\global\long\def\Span{\mathrm{Span}}%
 
\global\long\def\Domain{\mathrm{Domain}}%
 
\global\long\def\codim{\mathrm{codim}}%
 
\global\long\def\Var{\mathrm{Var}}%
 
\global\long\def\rank{\mathrm{rank}}%
 
\global\long\def\t{\mathsf{T}}%

\newcommand{\bE}{\mathbb{E}}

\providecommand{\remarkname}{Remark}
\providecommand{\lemmaname}{Lemma}
\providecommand{\theoremname}{Theorem}
\providecommand{\theoremname}{Corollary}
\providecommand{\theoremname}{Claim}
\providecommand{\theoremname}{Proposition}

\newcommand{\spn}[1]{\langle \! \langle #1 \rangle \! \rangle}
\newcommand{\card}[1]{\left| #1 \right|}

\newcommand{\SO}{\textup{SO}}
\newcommand{\Spin}{\textup{Spin}}
\newcommand{\Sp}{\textup{Sp}}
\newcommand{\SU}{\textup{SU}}
\newcommand{\U}{\textup{U}}
\newcommand{\HS}{\textrm{HS}}
\newcommand{\sign}{\textrm{sign}}
\newcommand{\parenth}[1]{{\left( #1 \right)}}
\renewcommand{\O}{\textup{O}}

\newcommand*\coef[4]{{#1}_{#2}({#3},{#4})}
\newcommand{\IC}{\mathrm{IC}}
\newcommand{\q}{\tilde{q}}
 
\newcommand{\F}{\mathbb{F}}

\global\long\def\sqbinom#1#2{\left[\begin{array}{c}
#1\\
#2
\end{array}\right]}%

\newcommand{\IB}[4]{k_{{#1}}({#2},{#3},{#4})}

\newcommand{\Mt}{M_t(n,q)}

\title[Forbidden Intersection Theorems for Matrix Spaces]{Forbidden Intersection Theorems for Matrix Spaces}
\author[1]{Esty Kelman}
\address[1]{Massachusetts Institute of Technology, Cambridge, MA., U.S.A. \& Boston University, Boston, MA., U.S.A. } 
\email[1]{ekelman@mit.edu}

\author[2]{Nathan Lindzey}
\address[2]{Department of Mathematics, University of Memphis, Memphis, TN., U.S.A.}
\email[2]{nathan.lindzey@memphis.edu}

\author[3]{Ohad Sheinfeld}
\address[3]{Einstein Institute of Mathematics, Hebrew University, Jerusalem, Israel.}
\email[3]{oshenfeld@gmail.com}

\thanks{This project has received funding from the European Union's Horizon 2020 research and innovation programme under grant agreement No 802020-ERC-HARMONIC. The work of the third author is partially supported by the European Research Council (StG no.~101163794)}

\begin{abstract}
A family of $m \times n$ matrices $\mathcal{F} \subseteq \mathbb{F}_q^{m \times n}$ is \emph{$(t-1)$-intersection-free} if $\dim \ker(A-B) \neq t-1$ for all $A,B \in \mathcal{F}$. A \emph{forbidden $(t-1)$-intersection problem} for a collection of matrices asks for the size and structure of extremal $(t-1)$-intersection-free families within that collection.

We solve this problem in $\GL(n,q)$ for all pairs $(n,t)$ such that $t<c\cdot n$ where $c$ is a universal constant. We show that the $t$-umvirates and their duals, are the only maximal $(t-1)$-intersection-free families $\mathcal{F} \subset \GL(n,q)$. Here, a $t$-umvirate is defined as the family of all matrices that agree on a fixed $t$-dimensional subspace, and its dual as those whose transposes agree on it. The best previously known result, due to Ellis, Kindler, and Lifshitz, established this bound under the assumption $n \geq e^{Ct\log t}$ for some constant $C>0$. We also give Frankl--R\"odl-type constructions showing that this range of $t$ is almost the best possible: we show that for values of $t>n/2$ the extremal behavior changes and no clean analogue is expected.

Our proof builds upon recent global hypercontractivity results for matrix spaces due to Evra, Kindler, and Lifshitz, and broadly applies to any sufficiently dense class of matrices.

\end{abstract}

\maketitle
\section{Introduction}\label{sec:new_intro}

Let $\binom{[n]}{k}$ be the collection of $k$-element subsets of $[n] := \{1,2,\ldots,n\}$. A family $\mathcal{F} \subseteq \binom{[n]}{k}$ is \emph{$t$-intersecting} if $|A \cap B| \geq t$ for any two $A,B \in \mathcal{F}$. A \emph{$t$-umvirate} of $\binom{[n]}{k}$ is a family of the form
\[
    \mathcal{F}_T = \left \{A \in \binom{[n]}{k} : T \subseteq A \right \}
\]
for a fixed subset $T \subseteq [n]$ of size $t$.
In their seminal 1961 paper, Erd\H{o}s, Ko, and Rado  proved that the $t$-intersecting families of the largest size are the $t$-umvirates, provided that $n$ is sufficiently large with respect to $k$ and $t$. 
\noindent This theorem, now known as \emph{the Erd\H{o}s--Ko--Rado theorem}, is a fundamental result in extremal combinatorics, and initiated the research area of intersection theorems. We refer the reader to
\cite{GodsilMeagher,EllisSurvey} for a comprehensive overview of the many
generalizations surrounding the Erd\H{o}s--Ko--Rado theorem. 
Following this work, Frankl~\cite{F78} and Wilson~\cite{Wilson84}, obtained partial results when $t$ grows with respect to $k,n$,until 1997, when Ahlswede and Khachatrian obtained the extremal families for every $(n,k,t)$ in their celebrated \emph{complete intersection theorem}~\cite{AK97}. The theorem states for each $(n,k,t)$, that one of the \emph{Frankl families} $F_{n,k,t,r}=\{S \in \binom{[n]}{k}:|S\cap [t+2r]| \geq t+r\}$ has the maximum size.

In a similar vein, Erd\H{o}s and S\'os asked for the maximum size of a
\emph{$(t-1)$-intersection-free} family $\mathcal{F}$ of $k$-element subsets
of $[n]$, that is, a family $\mathcal{F}$ such that $|A \cap B| \neq t-1$
for all $A,B \in \mathcal{F}$. Problems of this type are commonly referred
to as \emph{forbidden intersection problems}. While superficially similar
to the Erd\H{o}s--Ko--Rado problem, they have proven to be significantly
more difficult to resolve.

Although the problem has been solved in various parameter regimes in the
asymptotic setting (e.g., ~\cite{FranklF85,keller2021junta,EllisKL24,KZ24}), it remains open in full generality for arbitrary
$t$, $k$, and $n$ (see~\cite{EllisKL24} for a detailed discussion).
In all regimes where the problem has been resolved, the extremal family
is actually $t$-intersecting, making the forbidden intersection problem a
natural strengthening of the $t$-intersection problem.

The known solutions employ a diverse range of techniques, including
combinatorial, algebraic, probabilistic, and Fourier-analytic methods,
many of which have found important applications elsewhere. For example,
the work of Frankl and F\"uredi~\cite{FranklF85} contains one of the early
applications of the powerful delta-system method.
Keller and Lifshitz~\cite{keller2021junta} developed a broad extension of the
junta method that has since been widely used, while the work of Kupavskii and Zakharov~\cite{KZ24} introduced the spread approximation method, building on the breakthrough result of Alweiss et al.~\cite{AlweissLWZ19} on the sunflower conjecture. This method has since found numerous further applications (see, e.g.,~\cite{FranklK25,Kupavskii23partitions,Kupavskii23hereditary,KupavskiiN24}).

Like the Erd\H{o}s--Ko--Rado theorem, the Erd\H{o}s--S\'os problem and other forbidden intersection problems have since been studied in other domains, for example permutations~\cite{DN22, KZ24, KLS25} and codes~\cite{iarovikova2025forbiddingjustintersectionshort, keevash23}. In the following sections, we give an  overview of such results for $m \times n$ matrices $\mathbb{F}_q^{m \times n}$ over a finite field $\mathbb{F}_q$. 

\subsection{Intersection Problems for Matrices}
A family $\mathcal{F}\subseteq \mathbb{F}_q^{m\times n}$ is \emph{$t$-intersecting} if for all $A,B\in\mathcal{F}$, we have 
$\dim\ker(A-B)\ge t $. In contrast to the classical subsets setting, there are two natural $t$-umvirate families.

A \emph{$t$-umvirate} of $\mathbb{F}_q^{m \times n}$ is obtained by fixing a list $v$ of linearly independent vectors $v= v_1,\ldots,v_t\in\mathbb{F}_q^{n}$ and a list $w$ of arbitrary vectors $w = w_1,\ldots,w_t\in\mathbb{F}_q^{m}$, and taking
\[
\mathcal{U}(v,w)\;=\;\{A\in\mathbb{F}_q^{m\times n}:\ Av_i=w_i\ \text{for all }1\le i\le t\}.
\]
Dually, fix a list of linearly independent vectors $w = w_1,\ldots,w_t\in\mathbb{F}_q^{m}$ and a list of arbitrary vectors $v = v_1,\ldots,v_t\in\mathbb{F}_q^{n}$, and define the \emph{dual $t$-umvirate}
\[
\mathcal{U}^\ast(w,v)\;=\;\{A\in\mathbb{F}_q^{m\times n}:\ A^\top w_i=v_i\ \text{for all }1\le i\le t\}.
\]
When referring to generic $t$-umvirates, we often suppress the parameters, denoting them by $\mathcal{U}$ or $\mathcal{U}^*$.
Note that when $m=n$, the dual $t$-umvirates are also $t$-intersecting, giving a second type of extremal family (a phenomenon absent for $\binom{[n]}{k}$, permutations, or codes).

Huang~\cite{Huang87} (for $n<m$ and $(m,q)\neq(n+1,2)$) and Tanaka~\cite{Tanaka06} (for all $t\le n\le m$ and $q\ge 2$) proved a complete intersection theorem in this setting: for all $t\le n\le m$, the largest $t$-intersecting families are precisely the $t$-umvirates, and when $m=n$ also the dual $t$-umvirates. Their arguments are geometric and design-theoretic, tailored to $\mathbb{F}_q^{m\times n}$. Interestingly, in the matrix setting, the (dual) $t$-umvirates are always maximum. In particular, no analogue of the aforementioned Frankl families arises.

From here, there are two natural directions. One is to study intersection theorems in structured matrix groups such as $\GL(n,q)$ and $\SL(n,q)$. The other is to consider the harder Erd\H{o}s--S\'os \emph{forbidden intersection} variant, where one forbids intersection of exactly dimension $t-1$.

Ellis, Kindler, and Lifshitz~\cite{EllisKL23} recently did both. They investigated Erd\H{o}s--S\'os problems over $\mathbb{F}_q^{m \times n}$ for $n \ge n_0(t)$ and $m$ such that $|m-n|=O(1)$.\footnote{We note that $n_0(t)$ is unspecified in~\cite{EllisKL23} but inspection of the proof shows $n_0(t)\ge e^{Ct \log(t)}$ for some constant $C>0$~\cite{EllisPC}.} Here a family of $\mathcal{F}\subseteq\mathbb{F}_q^{m \times n}$ matrices is \emph{$(t-1)$-intersection-free} if for every $A,B\in \mathcal{F}$, $\dim \ker(A-B) \neq t-1$. Their main result is an Erd\H{o}s--S\'os theorem for the general linear group $\GL(n,q) \subseteq \mathbb{F}_q^{n \times n}$, namely, that the $t$-umvirates (and their duals) are the extremal $(t-1)$-intersection-free families of $\GL(n,q)$. They also state that the same proof works for $\SL(n,q)$ \emph{mutatis mutandis}. Here, the $t$-umvirates are the families obtained by picking linearly independent $v_1,\ldots,v_t \in \mathbb{F}_q^n$, linearly independent $w_1,\ldots,w_t \in \mathbb{F}_q^n$, and then taking all $A \in \GL(n,q)$ such that $Av_i = w_i$ for all $1 \leq i \leq t$. The dual families are constructed in a similar manner, only instead we insist that $A^\top v_i = w_i$ for all  $1 \leq i \leq t$. It is a routine calculation to see that a (dual) $t$-umvirate of $\GL(n,q)$ has size $ \prod_{i=1}^{n-t} (q^n-q^{i+t-1})$. 
\begin{thm}\cite{EllisKL23}\label{thm:GL(n,q)}
    For any $t \in \mathbb{N}$, there exists $n_0 = n_0(t) \in \mathbb{N}$ such that the following holds. If $n \in \mathbb{N}$ with
$n \geq n_0(t)$, $q$ is a prime power, and $\mathcal{F} \subseteq \GL(n,q)$ is $(t - 1)$-intersection-free, then
\[
|\mathcal{F}| \leq \prod_{i=1}^{n-t} (q^n-q^{i+t-1}).
\]
Moreover, equality holds if and only if $\mathcal{F}$ is a $t$-umvirate or a dual $t$-umvirate.
\end{thm}

Independently, Ernst and Schmidt~\cite{ernst2023intersection} study the Erd\H{o}s--Ko--Rado variant, which is implied by Theorem~\ref{thm:GL(n,q)}. A family $\mathcal{F} \subseteq \GL(n,q)$ is \emph{$t$-intersecting} if $\dim \ker(A-B) \geq t$ for any $A,B \in \mathcal{F}$. When $t=1$, we simply say that the family is \emph{intersecting}. Using the representation theory of $\GL(n,q)$, they gave a tight upper bound of $\prod_{i=1}^{n-t} (q^n-q^{i+t-1})$ on the size of a largest $t$-intersecting family for sufficiently large $n$ with respect to $t$, but they did not characterize the extremal families. Their dependency on $n$ is comparable to Ellis et al., but their proof does not readily carry over to $\SL(n,q)$~\cite{KaiPC}. We refer the reader to Appendix~\ref{sec:app} for a brief comparison of the two approaches.

Both Ellis et al.~and Ernst and Schmidt~also obtained `decoupled', that is, \emph{cross-intersecting} versions of their main results (see~\cite{EllisKL23,ernst2023intersection} for their formal statements). Here, we say that two families $\mathcal{F},\mathcal{G} \subseteq \mathbb{F}_q^{m \times n}$ are \emph{cross-$t$-intersecting} if for all $A \in \mathcal{F}$ and $B \in \mathcal{G}$, we have  $\dim \ker(A-B) \geq t$. Note that setting $\mathcal{F} = \mathcal{G}$ recovers the definition of $t$-intersecting. Similarly, a pair of families $\mathcal{F},\mathcal{G} \subseteq \mathbb{F}_q^{m \times n}$ is \emph{cross-$(t-1)$-intersection-free} if $\dim \ker(A-B) \neq  t-1$ for all $A\in \mathcal{F}$ and $B \in \mathcal{G}$. Note that setting $\mathcal{F} = \mathcal{G}$ recovers the definition of $(t-1)$-intersection-free. We will work with this decoupled notion of $(t-1)$-intersection-free throughout this paper. 

\medskip

\noindent Our main result is a substantial improvement on the results of Ellis et al.

\begin{thm}[Main Result]\label{thm:main2}
There exists a constant $C > 0$ such that for all $t,n \in \mathbb{N}$ such that $n \ge Ct$, if $\mathcal{F} \subseteq \GL(n,q)$ is $(t - 1)$-intersection-free, then
\[
|\mathcal{F}| \leq \prod_{i=1}^{n-t} (q^n-q^{i+t-1}).
\]
Moreover, equality holds if and only if $\mathcal{F}$ is a $t$-umvirate or a dual $t$-umvirate.
\end{thm}

Note that since every $t$-intersecting family is automatically $(t-1)$-intersection-free, Theorem~\ref{thm:main2} immediately yields the sharp Erd\H{o}s--Ko--Rado-type bound for $t$-intersecting families in $\GL(n,q)$ whenever $n\ge Ct$, thereby substantially extending the previously known range of parameters. 

To the best of our knowledge, all existing exact results for forbidden $(t-1)$-intersection problems in various settings (such as permutations and codes, see Section~\ref{sec:recent_results}) do not address the regime where $t$ can be large as $c\cdot n$ when $n$ is the ambient dimension (which is $n$ for $S_n,[m]^n,{[n] \choose k}, \GL(n,q)$), making our work the first to obtain exact results in this harder regime. For larger values of $t$, the problem has a different character, as discussed in Section~\ref{sec:recent_results}.

We shall deduce Theorem~\ref{thm:main2} among others from a technical stability result (Theorem~\ref{thm:main}) for cross-$(t-1)$-intersection-free families of $\mathbb{F}_q^{m \times n}$. Our stability result leverages many well-known facts about the \emph{bilinear form scheme} $\LL(V,W)$, a commutative association scheme defined over $\mathbb{F}_q^{m \times n}$, which we review in the next sections. Here, $\LL(V,W)$ denotes the set of all linear transformations from $V = \mathbb{F}_q^n$ to $W = \mathbb{F}_q^m$. At a high level, our methods exploit the fact that several natural sets and groups of matrices with non-Abelian structure (e.g., full-rank matrices and the special linear group) have high density within a commutative structure, allowing us to transfer stability results in the friendlier abelian domain $\LL(V,W)$ to the non-Abelian space without incurring too much error. This philosophy was already observed in Ellis et al.~\cite{EllisKL23}, but our proof is different -- the new crucial ingredient being the new \emph{global hypercontractivity} for matrix spaces, of Evra, Kindler and Lifshitz~\cite{EKL23} (see Section~\ref{sec:global}). 

We note that global hypercontractivity for the symmetric group~\cite{keevash2023sharp} has proven to be quite useful in a range of applications, including extremal combinatorics~\cite{keevash2024largest,KLS25}, character bounds~\cite{lifshitz2023bounds}, and product-growth/covering phenomena in finite groups~\cite{keller2024improved}.
The aforementioned global hypercontractivity result for matrix spaces has only been used beyond the original work for hardness of approximation results~\cite{minzer2024near,minzer2025near}. We hope that our approach will find use in further applications of matrix-space hypercontractivity.

\subsection{The Stability Result and a Proof Sketch}\label{sec:main}

We now present the statement of our stability result and an outline of the proof.

\begin{thm}[Stability]\label{thm:main} 
There exist constants $\alpha, c_0>0$ and $\beta> 1$ such that the following holds.
For every $m,n,t \in \mathbb{N}$ that satisfy $t \le \alpha n$ and $n \le m \le \beta n$, if $\mathcal{F}\subseteq \mathbb{F}_q^{m\times n}$ is a $(t-1)$-intersection-free family with
\[
|\mathcal{F}| \ge q^{-100}q^{m(n-t)},
\]
then $\mathcal{F}$ is almost contained in a $t$-umvirate $\mathcal{U}$ (or in a dual $t$-umvirate when $m=n$), in the sense that
\[
|\mathcal{F}\cap \mathcal{U}| \ge \bigl(1-q^{-c_0 n}\bigr)\,|\mathcal{F}|.
\]

\end{thm}
In the proof, we actually establish a slightly stronger cross-stability statement 
(Lemma~\ref{lem:iterative_argument}), but we state the above consequence in the introduction for simplicity.

The proof of Theorem~\ref{thm:main} follows the `structure versus pseudorandomness' paradigm. The main ingredient of the pseudorandomness component is a sharp global hypercontractivity result of Evra, Kindler, and Lifshitz~\cite{EKL23} for $\LL(V,W)$ stated in Section~\ref{sec:step1}. It is useful to think of `globalness' as our strong pseudorandomness property, where global means no significant density increase occurs after fixing a small number of rows/columns in some basis (see Section~\ref{sec:global} for a formal definition). 

\medskip

\noindent Let $\mathcal{F},\mathcal{G} \subseteq \LL(V,W) = \mathbb{F}_q^{m \times n}$ be  cross-$(t-1)$-intersection-free families. We define $\E[f]$ to be the \emph{measure} of $\mathcal{F}$ where $f\in L^2(\mathcal{L}(V,W))$ are their respective characteristic functions (see Section~\ref{sec:bilinear_forms} for more details). The proof consists of three steps, defined below:
\begin{enumerate}[label=Step \arabic*]
    \item \emph{(Global Hypercontractivity):} if either $\mathcal{F}$ or $\mathcal{G}$ is global, then we show that $|\mathcal{F}| \cdot |\mathcal{G}|$ has to be `small' by looking at the eigenvalues and eigenspaces of the `$t-1$ intersection free' Cayley graph, and using global hypercontractivity results of~\cite{EKL23}. 
    \item \emph{(No Density Bump):} assuming both $\mathcal{F}$ and $\mathcal{G}$ are not global, but one satisfies the weaker pseudorandomness condition of having no `small' density bump within any dictator, i.e., no $q^{\xi t}$-density increase (for some appropriately chosen constant $\xi > 1$) inside $\mathbb{F}_q^{(m-1) \times n}$ or $\mathbb{F}_q^{m \times (n-1)}$ after fixing one row or one column. In this case, we show there exist restrictions $\mathcal{F}' \subseteq \mathcal{F}$ and $\mathcal{G}' \subseteq \mathcal{G}$ such that 
    \begin{enumerate}
        \item  the cross-$(t-1)$-intersection-free pair $\mathcal{F}',\mathcal{G}'$ is isomorphic to a pair of cross-$(t-1)$-intersection-free families $\widetilde{\mathcal{F}}',\widetilde{\mathcal{G}}' \subseteq \LL(\widetilde{V},\widetilde{W})$ in a smaller space of bilinear forms $\LL(\widetilde{V},\widetilde{W})$,
        \item the restrictions do not lose too much measure as families of $\LL(V,W)$, and  
        \item the family $\widetilde{\mathcal{F}}' \subseteq \LL(\widetilde{V},\widetilde{W})$ is global. 
    \end{enumerate}
    If these conditions are met, then the pair $\widetilde{\mathcal{F}}',\widetilde{\mathcal{G}}'$ satisfies the hypotheses of Step 1 and $|\widetilde{\mathcal{F}}'| \cdot |\widetilde{\mathcal{G}}'|$ is roughly $|\mathcal{F}| \cdot |\mathcal{G}|$; therefore, their measure has to be `small' by Step~1.
    
    However, a key technical difficulty is that, unlike in product spaces, naive restrictions or projections of matrices
need not preserve $\dim\ker(A-B)$ or globalness, and thus may destroy the forbidden-intersection and globalness properties.
We address this by passing to suitable quotient spaces and, instead of working with $\{0,1\}$-valued indicator functions, we pass to a class of $[0,1]$-valued functions that we call \emph{fiber-average functions}, so that cross-$(t-1)$-intersection-freeness and `globalness' are preserved (see Section~\ref{sec:step2} for definitions).
    \item \emph{(Density Bump Within a Dictatorship):} If Steps 1 and 2 do not hold, i.e., one family admits a `small' density bump within some dictator, then we upgrade this density bump to almost containment in this dictator via an inductive argument. Iterating the density increment identifies $t$ independent constraints $Av_i=w_i$, forcing the family to be almost contained in a $t$-umvirate. 
\end{enumerate}
With this stability theorem in hand, we are then able to prove our main result via a bootstrapping argument.

\section{Related Work}

\subsection{Recent Results on Forbidden Intersection Problems}\label{sec:recent_results}

In recent years, there has been substantial progress in determining the maximum size and structure of families with forbidden $(t-1)$-intersection for growing values of $t$, in a variety of settings beyond linear maps.
Notable examples include forbidden intersection results for the slice~\cite{KZ24}, codes~\cite{iarovikova2025forbiddingjustintersectionshort, keevash23}, and permutations~\cite{KLS25, KZ24}.

At the same time, all known \emph{exact} results for forbidden $(t-1)$-intersection problems, including results for the original Erd\H{o}s-S\'os forbidden intersection problem such as~\cite{EllisKL24}, show a recurring pattern: a maximum family free of $(t-1)$-intersection is in fact $t$-intersecting.
The forbidden intersection problem has also been studied in parameter regimes where the maximum family is not $t$-intersecting, most notably in the Frankl-R{\"o}dl theorem~\cite{frankl1987forbidden}.
However, in these regimes the known results provide only upper bounds on the size of the extremal family and do not yield structural characterizations.

A representative illustration appears in the Erd\H{o}s-S\'os problem for permutations.
When $t$ grows linearly with $n$, the only known result is due to Keevash and Long~\cite{KL17}, who showed that for $\epsilon n < t \leq (1-\epsilon)n$, any $(t-1)$-intersection free family $F \subset S_n$ satisfies $|F| \leq (n!)^{1-\delta}$, where $\delta=\delta(\epsilon)$.

This qualitative shift has a natural conceptual explanation.
When $(t-1)$-intersection is forbidden and $t$ is very large, the constraint becomes much weaker: two random elements satisfy it with high probability.
Consequently, probabilistic constructions can produce very large families, and many non-isomorphic structures may attain comparable sizes (as demonstrated in Section~\ref{sec:FR}).
This makes the extremal problem fundamentally different from the small-$t$ regime, where extremal families are small juntas, in particular, variants of the Frankl families.

Against this backdrop, to the best of our knowledge, all existing exact results for forbidden $(t-1)$-intersection problems remain outside the Frankl-R{\"o}dl regime.
The closest result in this direction is due to~\cite{KLS25}, which for permutations establishes an exact theorem up to $t \le cn/\log(n)$, whereas the Frankl-R{\"o}dl regime corresponds to $t=\Theta(n)$.

In this sense, our result is the first to establish the implication that "forbidden $(t-1)$-intersection forces $t$-intersection" for $t$ linear in $n$, thereby matching the maximum parameter range (up to the linear constant) suggested by existing obstructions.

Regarding the proof techniques, a common ingredient of the recent works mentioned above is a reduction step. 
In this step, a forbidden $(t-1)$-intersection condition is refined, after suitable restrictions, to a $1$-intersection condition for families that exhibit some form of pseudorandomness, formalized via notions of globalness or spreadness.

This reduction typically takes place inside a product space (or the slice $\binom{[n]}{k}$), or alternatively the families are embedded into such a space after the reduction.
The subsequent argument shows that a $1$-intersecting pseudorandom family cannot be large.
It relies crucially on the product structure and applies either a sharp threshold argument or a random gluing argument, which is conceptually similar but technically distinct.

At a high level, these arguments show that a monotone pseudorandom family on $\{0,1\}^n$, with a $p$-biased distribution for small $p$, must have measure close to one when going up to density $p=1/2$, and thus cannot satisfy a nontrivial intersection constraint (see~\cite[Lemma~6]{KellerLMS24} or~\cite[Theorem 4]{KZ24}, for example).
The proofs of these sharp threshold results rely either on  hypercontractivity or on the spreadness lemma introduced in \cite{AlweissLWZ19}.

In the setting of matrix spaces, however, this approach does not apply.
Natural embeddings of matrices into product spaces do not, in general, preserve or increase the intersection parameter
\[
I(A,B)=\dim\ker(A-B),
\]
and consequently the forbidden intersection property is not compatible with the sharp threshold framework in product spaces or slices.

The general strategy of our proof is similar to the proof of ~\cite{KLS25} for permutations, but the details are fundamentally different, since our space is not embedded in a product space. 
As a result, our approach diverges from the above works in two fundamental respects:
\begin{itemize}
    \item Rather than reducing the forbidden $(t-1)$-intersection condition to a $1$-intersection condition, we work directly with the Cayley graph whose edges correspond to pairs of matrices whose difference has kernel dimension exactly $t-1$.
    \item Our arguments do not rely on embedding the family into a product space and applying sharp threshold results for pseudorandom families in that setting.
\end{itemize}

Regarding intersection problems in matrix spaces, the proof of Ernst and Schmidt~\cite{ernst2023intersection} uses radically  different techniques, relying heavily on the representation theory of $\GL(n,q)$. The proof of Ellis et al.~\cite{EllisKL23} is more closely related to our approach, and we compare the two in the following subsection.

\subsection{Comparison with~\cite{EllisKL23}}
The proof of Ellis et al.\ begins by establishing a regularity lemma, which yields a junta $\mathcal{J}$ covering most of the $(t-1)$-intersection-free family $\mathcal{F}$. Moreover, it also has the property that each intersection of the constituent subfamilies of $\mathcal{J}$ with $\mathcal{F}$ is an \emph{uncapturable family}, a weak form of pseudorandomness. They then prove that $\mathcal{J}$ must in fact be $t$-intersecting by strengthening this weak notion of uncapturability to a much stronger notion of \emph{quasiregularity}. At this point, they apply a hypercontractive inequality they developed, together with an argument involving the eigenvalues of the $(t-1)$-intersection Cayley graph.

Our overall approach diverges structurally by avoiding the machinery of the regularity lemma entirely. Instead of relying on junta approximation, we work directly with the $(t-1)$-intersection-free families. This direct approach (together with other technical ingredients) allows us to extend the range of $t$ up to $t=\Theta(n)$, whereas the regularity-lemma approach seems to be limited to substantially smaller $t$. We also obtain the extremal families in Step~3 of our proof (see Section~\ref{sec:main}) via a different argument, and in Step~1 we use a new global hypercontractivity for matrix spaces from~\cite{EKL23}.

However, our proof does share several technical ingredients with Ellis et al. In Step~1, as in their work, we make use of the $(t-1)$-intersection graph, showing that a global family cannot be cross-$(t-1)$-intersection free with another family, using a similar strategy to their proof that two quasiregular families are not free from cross-$(t-1)$-intersection. We also combine a hypercontractive inequality with bounds on the eigenvalues of the adjacency matrix of the $(t-1)$-intersection Cayley graph. However, the bounds on the magnitudes of the eigenvalues obtained by Ellis et al.\ are too crude for our purposes. Accordingly, in Proposition~\ref{prop:eigs} we derive asymptotic bounds on these eigenvalues that are essentially optimal.

Step~2 of our proof is also in the same spirit as Ellis et al. In particular, we upgrade the `no density bump to any dictator' property to globalness, in a manner similar, though not identical, to their upgrade from uncapturable families to quasiregular families.

\section{Preliminaries}

\subsection{The Bilinear Forms Scheme}\label{sec:bilinear_forms}

Because the notion of matrix intersection is defined in terms of rank, it turns out to be more natural to work with a well-known quotient of the group $(\mathbb{F}_q^{m \times n},+)$ that carries the structure of a symmetric association scheme. Henceforth, we let $V$ and $W$ be subspaces of dimension $n$ and $m$ of a vector space over $\mathbb{F}_q$ such that $n \leq m$ and $q$ is a prime power. We let $\spn{ \cdot }$ denote the $\mathbb{F}_q$-span of a collection of vectors. Define $\mathcal{L}(V,W)$ to be the set of linear transformations from $V$ to $W$, which is isomorphic to the space $\mathbb{F}^{m \times n}_q$ of $m \times n$ matrices over the field $\mathbb{F}_q$ upon picking a basis. We will frequently make use of this bijection without mention. Let $\mathcal{B}_{m,n} = \{\mathbf{A}_0,\mathbf{A}_1,\cdots,\mathbf{A}_n\}$ be the set of $q^{mn} \times q^{mn}$ binary matrices indexed by $\LL(V,W)$, defined such that $(\mathbf{A}_i)_{A,B} = 1$ if $\text{rank}(A-B)=i$; otherwise, 0. Each $\mathbf{A}_i$ can be identified with a $v_i$-regular graph, where $v_i$ is the number of rank $i$ matrices of $\LL(V,W)$. It is well-known that $\mathcal{B}_{m,n}$ forms a symmetric association scheme, namely, the \emph{bilinear forms scheme} (see~\cite{Delsarte78,GodsilAssoc}, for example). Its name is due to the fact that any matrix $A \in \mathbb{F}_q^{m \times n}$ can be seen as a bilinear form  $u^\top A v$, $u \in \mathbb{F}_q^m$, $v \in \mathbb{F}_q^n$. The non-identity $\mathbf{A}_i$'s are collectively referred to as the \emph{associates} of the association scheme. We now recall some well-known facts about $\mathcal{B}_{m,n}$, establishing some notation along the way that will be used throughout this work.

Any real-valued function $f$ of the inner-product space $L^2(\mathcal{L}(V,W))$ where 
$$\langle f,g \rangle := \frac{1}{q^{mn}} \sum_{A \in \LL(V,W)} f(A){g(A)}$$ 
admits a \emph{degree decomposition} 
\[
    f = f^{=0} + f^{=1} + \cdots + f^{=n}
\]
where $f^{=d} \in V^{=d} \leq L^2(\mathcal{L}(V,W))$ is the orthogonal projection onto the space of so-called pure degree-$d$ real-valued functions on $\LL(V,W)$, so that $\langle f^{=i}, f^{=j} \rangle = 0$ for all $0 \leq i , j \leq n$ such that $i \neq j$. Formally, we define $f^{=d} := E_df$ where $E_d$ is the \emph{$d$th primitive idempotent} of $\mathcal{B}_{m,n}$, i.e., the orthogonal projection onto the space $V^{=d}$ of pure degree-$d$ functions. For any matrix $A$ that lies in the real span of $\mathcal{B}_{m,n}$, we have $AE_i = \eta_i E_i$ for some $\eta_i \in \mathbb{R}$. It follows that each $V^{=d}$ is contained in an eigenspace of $A$, and in particular, that $V^{=d}$ is the $\eta_d$-eigenspace of $A$ provided that $\eta_d \neq \eta_i$ for all $i \neq d$. Moreover, it is well-known that the dimension of $V^{=d}$ equals the number of rank $d$ matrices of $\mathcal{L}(V,W)$. 

Define $\E[f] := \langle 1, f \rangle = \frac{1}{q^{mn}} \sum_{A \in \LL(V,W)} f(A)$. 
For a family $\mathcal{F} \subseteq \LL(V,W)$, we call $\E[f]$ the \emph{measure} of $\mathcal{F}$ where $f$ is the characteristic function (indicator function) of $\mathcal{F}$.

The eigenvalues of the associates $\mathbf{A}_i \in \mathcal{B}_{m,n}$ are well-known and play an important role in this work.  In Proposition~\ref{prop:eigs}, we give upper bounds on the magnitudes of the eigenvalues of each normalized associate $\tilde{\mathbf{A}}_i := v_i^{-1}\mathbf{A}_i$. For more details about the eigenvalues of $\mathcal{B}_{m,n}$, we refer the reader to~\cite{CioabaH22,Delsarte78}. 

To see why $\mathcal{B}_{m,n}$ is relevant to the study of intersecting families, simply observe that the independent sets of the graph $\mathbf{A}_n$ are precisely the intersecting families of $\LL(V,W)$. 
More generally, the maximum independent sets of the graph $\mathbf{A}_n + \mathbf{A}_{n-1} + \cdots + \mathbf{A}_{n-t+1}$ are the largest $t$-intersecting families of $\LL(V,W)$, and the largest $(t-1)$-intersection-free families of $\LL(V,W)$ are the maximum independent sets of $\mathbf{A}_{n-t+1}$.

\subsection{Global Restrictions}\label{sec:restriction}\label{sec:global}

In this section we give formal definitions for restriction, globalness, and global restrictions of $\LL(V,W)$. Along the way, we recall some standard results concerning $L_p$ norms of inner-product spaces and we state several notational conventions that will be used throughout this work.

Let us first recollect some examples of restriction that are undoubtedly more familiar to the reader. Recall that a $d$-restriction of the hypercube $\mathbb{Z}_2^n$ is one of $\binom{n}{d}$ subcubes $\mathbb{Z}_2^{n-d}$ obtained by fixing $d$ of its $n$ coordinates. Similarly, a $d$-restriction in the symmetric group $S_n$ is a \emph{double-translate} of $S_{n-d}$, i.e., $\sigma S_{n-d} \pi \subset S_n$ for some $\sigma,\pi \in S_n$. Intuitively, the notion of restriction for bilinear forms is similar to these combinatorial examples; however, a rigorous definition requires  some linear-algebraic formalities.

\begin{definition} [Restriction~\cite{EKL23}]\label{def:restriction}
Let $T \in \LL(V, W)$, $V' \leq V$, $W' \leq W$.  For any $f : \LL(V,W) \rightarrow \mathbb{C}$, we define its \emph{$(V',W',T)$-restriction} 
$f_{(V',W')\rightarrow T} : \LL(V/V', W') \rightarrow \mathbb{C}$ such that 
$$f_{(V',W') \rightarrow T}(A)= f(A'+T) \quad \text{ for all } A \in \LL(V/V', W')$$ 
where $A' \in \LL(V,W)$ is the unique linear map with kernel containing $V'$ satisfying $A':= A \circ Q_{V'}$, 
 and $Q_{V'}$ is the natural quotient map, i.e., $Q_{V'} : V \rightarrow V/V'$ such that $v \mapsto v + V'$. 
\end{definition}
\noindent In this way, we identify $\LL(V/V' ,W')$ with the space of linear maps of $\LL(V,W)$ that annihilate $V'$ and whose image lies within $W'$. The linear maps of the form $A' + T \in \LL(V,W)$ in the definition above are precisely the
linear maps $B \in \LL(V,W)$ such that $B$ agrees with $T$ on $V'$ and $B^\top$
agrees with $T^\top$ on the annihilator of $W'$.

The \emph{size} of a $(V',W',T)$-restriction is  $\dim(V') + \text{codim}(W')$.
If the size of a $(V',W',T)$-restriction equals $d$, then we call $f_{(V' ,W') \rightarrow T}$ a \emph{$d$-restriction} of $f$. We say a $(V',W',T)$-restriction is \emph{strong} if $\dim(V') = \mathrm{codim}(W')$. In terms of matrices, the $d$-restriction of a function $f \in L^2(\LL(V,W))$ corresponds to
restricting $f$ to those matrices where $r$ specific rows and $c$ specific columns take fixed values specified by $T$ such that $r + c = d$
(and translating the domain by a fixed matrix if the matrix of $T$ has non-zero entries outside the $r$ fixed
rows and the $c$ fixed columns).

The definition of restriction for set families follows from the case $f = 1_{\mathcal{F}}$ where $1_{\mathcal{F}} : \LL(V,W) \rightarrow \{0,1\}$ is the \emph{characteristic (indicator) function} of $\mathcal{F}$, i.e., $1_{\mathcal{F}}(A) = 1$ if $A \in \mathcal{F}$; otherwise, $1_{\mathcal{F}}(A) = 0$. In particular, for any $\mathcal{F} \subseteq \LL(V,W)$, $T \in \LL(V, W)$, $V' \leq V$, $W' \leq W$, the $(V',W',T)$-restriction $\mathcal{F}_{(V',W') \rightarrow T} \subseteq \mathcal{F} \subseteq \LL(V,W)$ of $\mathcal{F}$ is the set of matrices $A' + T \in \LL(V,W)$ as in Definition~\ref{def:restriction} such that $A' + T \in \mathcal{F}$. 

For any family $\mathcal{F} \subseteq \LL(V,W)$, $v \in V$, and $w \in W$, define the \emph{one-restrictions} or \emph{dictators}
\[
   \mathcal{F}_{v \mapsto w} := \left\{ B \in \mathcal{F} : Bv = w   \right\},
\]
dually,
\[
\mathcal{F}_{w \uparrow v} := \left\{ B \in \mathcal{F} : B^\top w = v   \right\}.
\]
We call the first one a column restriction, and the second one a row restriction.

For all $\mathcal{F} \subseteq \LL(V,W)$ with characteristic function $f \in L^2(\LL(V,W))$, define $\mu(\mathcal{F}) := \E[f] = \langle 1,f \rangle = \langle f, f \rangle$ to be the \emph{measure} of $\mathcal{F}$. Since we identify $\LL(V/V' ,W')$ with the set of linear maps of $\LL(V,W)$ that annihilate $V'$ and whose image lies within $W'$, for a family $\mathcal{F} \subseteq \LL(V/V' ,W')$, we use the notation $\mu^{(V',W')}(\mathcal{F})$ to indicate the measure is taken with respect to $\LL(V/V' ,W')$ in order to avoid confusion.

\begin{definition}[Global]
Fix $\g >0$. We say that $f: \LL(V,W) \rightarrow [0,1]$ is \emph{$\g$-global} if for all $d$, the measure of every $d$-restriction $(V',W',T)$ jumps by at most $\g^d$, formally,
$$\E[f_{(V',W')\rightarrow T}] \le  \g^d \E[f]$$
for every $d$-restriction $(V',W')\rightarrow T$. We call $\g$ the \emph{global constant}.
We say $\mathcal{F} \subseteq \LL(V,W)$ is \emph{$\g$-global} if its indicator function is \emph{$\g$-global}.
\end{definition}

We also require a more refined notion of globalness that is well-defined for arbitrary functions $f \in L^2(\LL(V,W))$ and $L_p$ norms. For all $p \in (0,\infty)$, define the \emph{p-norm} of $f \in L^2(\LL(V,W))$ to be
\[
\| f \|_p := \left(\frac{1}{q^{mn}} \sum_{A \in \LL(V,W)} |f(A)|^p \right)^{1/p}.
\]
Note that $\|f\|_p \leq \|f\|_q$ for all $1 \leq p < q$. Let $fg := f(A)g(A)$ for all $A \in \LL(V,W)$ be the pointwise product of two functions $f,g : \LL(V,W) \rightarrow \mathbb{R}$. Recall the following well-known result due to H\"{o}lder.
\begin{thm}[H\"{o}lder's Inequality]\label{thm:holder} Let $f,g : \LL(V,W) \rightarrow \mathbb{R}$ and $1 \leq p < \infty$. Let $q$ be the H\"{o}lder conjugate of $p$, i.e., $1/p + 1/q = 1$ (setting $q = \infty$ if $p=1$). Then 
\[
    \|fg\|_1 \leq \|f\|_p \|g\|_q.
 \]
\end{thm}

\begin{definition}[$(r,\varepsilon)$-global~\cite{EKL23}] A function $f : \LL(V,W) \rightarrow \mathbb{R}$ is \emph{$(r,\varepsilon)$-restriction global} if $$\|f_{(V',W') \mapsto T}\|_2^2 \leq \varepsilon$$ 
for every $r$-restriction $(V',W',T)$. 
\end{definition}

\begin{definition}[$(r,\varepsilon,L_p)$-global~\cite{EKL23}] For all non-negative $p \neq 2$, we say a function $f : \LL(V,W) \rightarrow \mathbb{R}$ is \emph{$(r,\varepsilon,L_p)$-restriction global} if $$\|f_{(V',W') \mapsto T}\|_p \leq \varepsilon$$ 
for every $r$-restriction $(V',W',T)$. 
\end{definition} 

Note that the two definitions are slightly inconsistent, and we keep them in order to be consistent with~\cite{EKL23}.

Evra, Kindler, and Lifshitz~\cite{EKL23} give several level-$d$ inequalities for $\LL(V,W)$ with respect to $(r,\varepsilon)$-globalness and $(r,\varepsilon,L_p)$-globalness that we shall invoke in Section~\ref{sec:step1}. While the last two definitions of globalness are similar in spirit, they differ from the first, as a $\g$-global function insists that every restriction does not give a substantial jump in measure. We use Evra et al.'s level-$d$ inequalities by taking  $\varepsilon = \g^d \|f\|_2^2$, 
because any $\g$-global function family $\mathcal{F}$ with indicator function $f$ satisfies 
$$\mu^{(V',W')}(\mathcal{F}_{(V',W') \mapsto T}) = \|f_{(V',W') \mapsto T} \|_2^2 \leq \g^d \|f\|_2^2 = \gamma^d \mu(\mathcal{F})$$ 
for all $d$-restrictions $(V',W',T)$. 

If a function is $d$-restriction global with respect to both the $L_p$ and $L_q$ norms, then we can interpolate via the 
 \emph{log-convexity of $L_p$ norms}, as demonstrated in Proposition~\ref{prop:1-2-log-convexity}. 

\begin{thm}[log-convexity of $L_p$-norms]\label{thm:log-convex}
For all $0 \le \theta \le 1$ and positive $p_\theta,p_0, p_1 \in \mathbb{R}$ satisfying $\frac{1}{p_\theta} = \frac{1-\theta}{p_0}+\frac{\theta}{p_1}$, we have
    $$\| f \|_{p_\theta} \le \| f \|_{p_0}^{1-\theta}\| f \|_{p_1}^{\theta}.$$
\end{thm}
\begin{prop}\label{prop:1-2-log-convexity}
Let $1 < \ell' < 2$ and let $\ell$ be its H\"older conjugate. If $f  \in L^2(\LL(V,W))$ is both $(d,\varepsilon_1,L_1)$-restriction global and $(d,\varepsilon_2)$-restriction global, then $f$ is $(d,\varepsilon_{\ell'},L_{\ell'})$-restriction global where 
$$\varepsilon_{\ell'} = \varepsilon_{2}^{\frac{1}{\ell}} \varepsilon_{1}^{\frac{\ell -2 }{\ell}}.$$
\end{prop}
\begin{proof}
Setting $p_0 = 2$, $p_1 = 1$, and $p_\theta = \ell'$ in Theorem~\ref{thm:log-convex} gives $\theta = (\ell-2)/\ell$, thus
\begin{align*}
    \| f \|_{\ell'} \le \| f \|_{2}^{\frac{2}{\ell}}\| f \|_{1}^{\frac{\ell-2}{\ell}}  \leq  \varepsilon_{2}^{\frac{1}{\ell}} \varepsilon_{1}^{\frac{\ell -2 }{\ell}},
\end{align*}
as desired.
\end{proof}

We conclude this section with a special type of restriction called a \emph{global restriction} that will ensure that condition (c) of Step 2 is met. 
\begin{definition}[Global Restriction]\label{def:global_restriction}
For any $\{0,1\}$-valued 
function $f \in L^2(\LL(V,W))$ and any $\g > 0$, we construct a \emph{$\g$-global  restriction} of $f$ by picking a restriction $(V',W',T)$ such that 
\[
\E[f_{(V',W') \rightarrow T}] / \g^{d}
\]
is maximum over all the choices of $d \in \mathbb{Z}_{\ge 0}$ and $d$-restrictions $(V',W',T)$. (If there are several maxima, we pick one where $d$ is maximum). It is easy to see that $f_{(V',W') \rightarrow T}$ is indeed $\g$-global, and also satisfies
\[\E[f_{(V',W') \rightarrow T}] \ge \g^{d} \E[f].\]
\end{definition}

\begin{definition}[Strong Restriction]
For any $f : \LL(V,W) \rightarrow \mathbb{C}$, a $2d$-restriction $(V',W',T)$ is \emph{strong} if $\mathrm{dim}(V') = \mathrm{codim}(W') = d$
\end{definition}

For Step 2, we require our global restriction to be strong. In Definition~\ref{def:global_restriction}, one cannot guarantee that a global restriction will be strong; however, the same procedure, ranging over all strong restrictions, gives a strong $2d$-restriction of $f$ bounded from below by $\g^{d} \E[f]$.

\begin{lem}[Strong Global Restriction]\label{lem:global_restriction}
For any $\{0,1\}$-valued function $f \in L^2(\LL(V,W))$  and any $\g > 0$, we construct a \emph{strong $\g$-global restriction} of $f$ by picking a strong $2d$-restriction $(V',W',T)$ such that
\[
\E[f_{(V',W') \rightarrow T}]/ \g^{d}
\]
is maximum over all the choices of $d \in \mathbb{Z}_{\ge 0}$ and strong $2d$-restrictions $(V',W',T)$. (If there are several maxima, we pick one where $d$ is maximum). It is easy to verify that this restriction is indeed $\g$-global and satisfies 
    \[
        \E[f_{(V',W') \rightarrow T}] \ge \g^{d} \E[f].
    \]
\end{lem}

\subsection*{Notation}
For the proofs of Steps 1-3, we list the constants.
\begin{itemize}
    \item $a,\alpha$ are scalars associated with $t$;
    \item $b,\beta, \beta'$ are scalars associated with $m$;
    \item $\gamma > 0$ is the global constant;
    \item $\xi \in \mathbb{N}$ is a constant associated with the `density bump' chosen in Step 3;
     \item $\tilde{c} > 0$ is a constant chosen in Step 2; and
    \item $c_3 > 0$ is a constant chosen in Step 3.
\end{itemize}

\section{Step 1: Global Hypercontractivity}\label{sec:step1}

In this section, we combine the global hypercontractivity results of~\cite{EKL23} with bounds on the eigenvalues of the associates of $\mathcal{B}_{m,n}$ to prove Step~1 (Theorem~\ref{thm:step1}). This theorem says that if $\mathcal{F},\mathcal{G} \subseteq \LL(V,W)$ are cross-$(t-1)$-intersection-free and one of the families is global, then their measure must be small.

We actually prove a more general version (needed later) in which 
$f,g : \LL(V,W) \to [0,1]$ instead of $\{0,1\}$ (here, we think of $\mathcal{F},\mathcal{G} \subseteq \LL(V,W)$ as $\{0,1\}$-valued functions). In this setting, cross-$(t-1)$-intersection-free means $\langle f, \mathbf{A}_{n-t+1} g \rangle = 0$. A key observation from Section~\ref{sec:bilinear_forms} is that the eigenspaces of $\mathbf{A}_{n-t+1}$ are $V^{=d}$, so after normalization we get
\[
\langle f, \tilde{\mathbf{A}}_{n-t+1} g \rangle
= \mathbb{E}[f]\mathbb{E}[g] + \sum_{d=1}^n \eta_d \, \langle f^{=d}, g^{=d}\rangle,
\]
where $\eta_d$ are the normalized eigenvalues. We show that this inner product is positive by proving that $\mathbb{E}[f]\mathbb{E}[g]$ dominates the sum (when the expectations are large enough). This is done by bounding each term $|\eta_d \langle f^{=d}, g^{=d}\rangle|$ separately.

We bound $|\eta_d|$ in Proposition~\ref{prop:eigs}. Then, in the proof of Theorem~\ref{thm:step1}, we apply H\"older's inequality (for large $\ell$) to obtain
\[
|\langle f^{=d}, g^{=d}\rangle|
\le \|f^{=d}\|_\ell \, \|g^{=d}\|_{\ell'}
\approx \|f^{=d}\|_\ell \, \mathbb{E}[g]
\]
where $\ell'$ is the H\"older conjugate of $\ell$.
Since $f$ is global, we can apply the global hypercontractivity results to prove Lemma~\ref{lem:lvld}, which shows that
\[
\|f^{=d}\|_\ell \le \mathbb{E}[f]\cdot o(1/|\eta_d|).
\]
Putting these bounds together completes the argument.

Throughout this section we use several results from~\cite{EKL23}. Some of these results are stated in terms of \emph{generalized influences}, which we do not define here and treat as a black box; see~\cite[\S 2.2]{EKL23} for details.

\begin{thm}\cite[Theorem 5.5]{EKL23}\label{thm:small-influence}
    Let $\ell \geq  4$ be a power of $2$ and let $\ell'$ be its H\"older conjugate. Suppose that $f \in L^2(\LL(V,W))$  is $(d, \varepsilon, L_{\ell'} )$-restriction global, and let $\varepsilon'  = q^{500d^2\ell}\varepsilon^2$. Then $f^{=d}$ has $(d,\varepsilon')$-small generalized influences.
\end{thm}
\begin{lem}\cite[Lemma 3.6]{EKL23}\label{lem:small-influence} 
    If $f \in L^2(\LL(V,W))$ is of degree $d$ and has $(d,\varepsilon')$-small generalized influences, then $f$ is $(r,q^{10dr}\varepsilon')$-restriction global for any $r \geq d$.
\end{lem}
\begin{cor}\label{cor:global}
     Let $\ell \geq  4$ be a power of $2$ and let $\ell'$ be its H\"older conjugate. Suppose that $f \in L^2(\LL(V,W))$ is $(d, \varepsilon_{\ell'}, L_{\ell'} )$-restriction global. Then $f^{=d}$ is $(d,q^{510d^2\ell}\varepsilon_{\ell'}^2)$-restriction global.
\end{cor}
\begin{proof}
Apply Theorem~\ref{thm:small-influence} and then Lemma~\ref{lem:small-influence} with $r = d$.
\end{proof}

\noindent Theorem~\ref{thm:lvld} below is a global \emph{level-$d$ inequality} for the bilinear forms scheme $\LL(V,W)$, which will be instrumental for carrying out the first step of the proof.  
\begin{thm}[Level-$d$ Inequality {\cite[Theorem 1.13]{EKL23}}]\label{thm:lvld}
    Suppose that $f \in L^2(\LL(V,W))$ is a $(d,\varepsilon)$-restriction global function of degree $d$. Let $\ell \geq 4$ be a power of 2. Then
    \[
     \|f\|_\ell^\ell \leq q^{200d^2\ell^2} \|f\|_2^2 ~\varepsilon^{\ell/2 - 1}
    \]
\end{thm}
\noindent The following technical lemma is a simple consequence of Theorem~\ref{thm:lvld}.
\begin{lem}\label{lem:lvld}
    Let $\ell \geq 4$ be a power of 2, let $\ell'$ be its H\"{o}lder conjugate, and $1 \leq d \leq n$. Let $f \in L^2(\LL(V,W))$ be a function with $\|f\|_1 \le 1$, such that  $f^{=d}$ is $(d,q^{510d^2\ell} (\varepsilon_2^{\frac{1}{\ell}} \varepsilon_1^{\frac{(\ell - 2)}{\ell}})^2)$-restriction global where $\varepsilon_1 = \g^d\|f\|_1$ and $\varepsilon_2 = \g^d$.
 Then
\[
 \|f^{=d}\|_\ell \leq q^{710d^2\ell}\g^{d}  \|f\|_2^{{2}/{\ell}} \|f\|_1^{1-4/\ell}. 
\]
\end{lem}
\begin{proof}
    By Theorem~\ref{thm:lvld} we have 
\begin{align*}
    \|f^{=d}\|_\ell^\ell &\leq q^{200d^2\ell^2} \|f^{=d}\|_2^2 ~(q^{510d^2\ell} (\g^{d/\ell'} \|f\|_1^{(\ell - 2)/\ell})^2)^{\ell/2 - 1} \\
    &\leq q^{200d^2\ell^2} \|f^{=d}\|_2^2 ~q^{(510d^2\ell)(\ell/2-1)} (\g^{d} \|f\|_1^{(\ell - 2)/\ell})^{\ell - 2} \\
    &\leq q^{710d^2\ell^2}\g^{d\ell} \|f^{=d}\|_2^2 \cdot \|f\|_1^{(\ell - 2)^2/\ell} \\
    &\leq q^{710d^2\ell^2}\g^{d\ell}\|f^{=d}\|_2^2 \cdot \|f\|_1^{\ell - 4}\\
     &\leq q^{710d^2\ell^2}\g^{d\ell}  \|f\|_2^2\|f\|_1^{\ell-4}. 
\end{align*}
\end{proof}

\noindent To apply Lemma~\ref{lem:lvld}, we require sufficiently strong upper bounds on the magnitudes of the eigenvalues of $\tilde{\mathbf{A}}_{n-t+1}$. Recall from Section~\ref{sec:bilinear_forms} that the eigenspaces of $\tilde{\mathbf{A}}_{n-t+1}$ are precisely the subspaces $V^{=d}$ (or a union of those). The following proposition provides bounds on the corresponding eigenvalues.

\begin{prop}\label{prop:eigs}
Let $m \geq n$ and let $\tilde{\mathbf{A}}_{n-t+1}$ be the matrix $\mathbf{A}_{n-t+1} \in \mathcal{B}_{m,n}$ scaled so that its largest eigenvalue is 1, and let $\eta_d$ be the eigenvalue
corresponding to $V^{=d}$ for all $0 \leq d \leq n$.  

If $q=2$ and $n=m$, then we have 
\[
    |\eta_d| =  
    \begin{cases}
             O\left( 2^{-(d(2n-d)-t^2)/2} \right)   \quad &\text{ if $d \geq t$ }\\
           O\left( 2^{-d(n - t + 1) + \log_2 n}\right) \quad & \text{ if $d < t$.}
    \end{cases}
\]
Otherwise, we have 
\[
    |\eta_d| = 
        \begin{cases}
            O\left( q^{-(d(2m-d)-t(2m-2n+t))/2} \right) \quad &\text{ if $d \geq t$ }\\
            O\left( q^{-d (n - t + 1)} \right) \quad & \text{ if $d < t$.}
        \end{cases}
\]
\end{prop}
\begin{proof}
    Recall that $v_{n-t+1}$ is the number of matrices of $\LL(V,W)$ of rank $n-t+1$. The eigenvalues $\eta_d$ of $\tilde{\mathbf{A}}_{n-t+1}$ are
    \begin{align}\label{eq:eigs}
        \eta_d = \frac{1}{v_{n-t+1}} \sum_{i=0}^{\min (n-t+1,n-d)} (-1)^{n-t+1 - i} q^{mi + \binom{n-t+1-i}{2}} {n-i \brack t-1}{ n - d \brack i }
    \end{align}
    (see~\cite{CioabaH22,Delsarte78}, for example). We have $v_{n-t+1} = \Omega(q^{(m+t-1)(n-t+1)})$ by basic counting and the following well-known bounds on Gaussian binomial coefficients:
    $$q^{k(n-k)} \leq {n \brack k} \leq 4q^{k(n-k)},$$
    (see~\cite[Lemma 3.5]{NeumannP95}, for example).
    Unless $q=2$ and $m=n$, the magnitudes of each of the summands in the eigenvalue expression above are an increasing function in $i$ (see~\cite[\S 4]{CioabaH22}, for example). Because the summation is also alternating, an upper bound on $|\eta_d|$ is obtained by bounding the summand corresponding to 
    $$i = \min(n-(t-1),n-d) = n-\max(t-1,d) =: n-l.$$ 
    We have
    \begin{align*}
        |\eta_d| &= O\left( \frac{q^{m(n-l) + \binom{l-t+1}{2}} {l \brack t-1} {n-d \brack n-l} }{q^{(m+t-1)(n-t+1)}} \right)\\
        &= O\left( \frac{q^{m(n-l) + (l-t+1)(l-t)/2}q^{(t-1)(l-t+1)}q^{(n-l)(l-d)}}{q^{(m+t-1)(n-t+1)}} \right).
    \end{align*}
    Substituting $l = d$ for the case where $d \geq t$ gives 
    \begin{align*}
        |\eta_d| &= O\left( \frac{q^{m(n-d) + (d - t + 1) (d + t - 2)/2 }}{q^{(m+t-1)(n-t+1)}} \right) \\
        &= O\left( q^{(d^2 - d (2 m + 1) + (t - 1) (2 m - 2 n + t))/2} \right)\\
        &\leq  O\left( q^{-(d(2m-d)-t(2m-2n + t))/2} \right)\\
    \end{align*}
    Substituting $l = t-1$ for the case where $d < t$ gives
    \begin{align*}
        |\eta_d| &= O\left( \frac{q^{m(n-t+1) + (n-t+1)(t-1-d)}}{q^{(m+t-1)(n-t+1)}} \right) \\
        &= O\left( q^{-d (n - t + 1)} \right).
    \end{align*}
    
Now assume $q=2$ and $m=n$. The $d \geq t$ case is easily seen from~\cite[Lemma 7]{EllisKL23}. 
For the remaining case, first note for any fixed $d,t \in \mathbb{N}$ such that $d < t < n$, that the function 
\[
    f(i) := 2^{(n^2 + n (2 i - 1) - t^2 + 3 t - i^2 + i - 2)/2 - d i}
\]
is a non-decreasing function of $i$ in the range $0 \leq i \leq n-d$. By the standard Gaussian coefficient estimates above, the absolute value of the $i$th summand in Equation~(\ref{eq:eigs}) is $O(f(i))$. 
Taking $i = n-l=n-t+1$ as above, the largest term of the summation is again
\[
    O\left(\frac{2^{(n^2 + n (2 (n-t+1) - 1) - t^2 + 3 t - (n-t+1)^2 + (n-t+1) - 2)/2 - d (n-t+1)}}{2^{(n+t-1)(n-t+1)}} \right) \! = \! O\left(\frac{2^{(n + t - 1) (n - t+1) - d(n - t+1) } }{2^{(n+t-1)(n-t+1)}}\right) \! = \! O\left( 2^{-d(n - t + 1)} \right).
\]
Crudely, we now have $|\eta_d| = O\left( 2^{-d(n - t + 1) + \log_2 n} \right)$, as desired.
\end{proof}

If $\mathcal{F}$ and $\mathcal{G}$ are the same $t$-umvirate, then $\mathcal{F},\mathcal{G}$ is a `large' cross-$(t-1)$-intersection-free pair of families with measure
\[
    \mu(\mathcal{F})\mu(\mathcal{G}) = \left( \frac{q^{m(n-t)}}{q^{mn}} \right)^2 = q^{-2mt}.
\]
Henceforth, this will be our notion of `large' measure.

We are now in a position to prove Theorem~\ref{thm:step1}, the main result of Step 1. It implies that if $\mathcal{F}$ and $\mathcal{G}$ are cross-$(t-1)$-intersection-free and one of the families is $\gamma$-global, then their measure must be `small', i.e., if one of two functions $f,g : \mathcal{L}(V,W) \rightarrow \{0,1\}$ is global with respect to the $L_1$ norm, and $\langle f, \tilde{\mathbf{A}}_{n-t+1}g\rangle = 0$, then $\E[f]\E[g]$ is `small'. 
\begin{thm}\label{thm:step1}
    For each fixed $\beta \ge 1$ such that $m \le \beta n$. Let $\gamma = q^{\tilde{c}\beta' \xi t}$ for any $\xi \geq 8$ and for any $\tilde{c},\beta ' \geq 1$.  
    For every $c>1$, there exist $n_0 \in \mathbb{N}$ and $\alpha>0$ such that for every $ n>n_0$ and $\alpha n \geq t$, the following holds. 
    
    Let  $f,g : \LL(V, W) \rightarrow [0,1]$ such that 
    $\langle f, \tilde{\mathbf{A}}_{n-t+1} g \rangle = 0$. 
    For all $d$, assume that $f$ is $(d,\epsilon_1,L_1)$-restriction global with $\epsilon_1 = \g^d \|f\|_1$. Then we have
    $$\E[f]\E[g] \le q^{-2cmt}.$$
\end{thm}

\begin{proof}
Suppose the contrary, i.e., there exists some $c > 1$ such that for all $n_0 \in \mathbb{N}$ and for all $\alpha > 0$, there exists some $n \in \mathbb{N}$ and $t \leq \alpha n$ such that $\E[f]\E[g] \geq q^{-2cmt}$.

Upon expanding $f$ and $g$ in the Fourier basis, we have 
    \[
        0 = \langle f, \tilde{\mathbf{A}}_{n-t+1} g \rangle  = \E[f]\E[g] + \sum_{d=1}^n \eta_d~ \langle f^{=d},g^{=d}\rangle 
    \]
    where $\eta_d$ is the eigenvalue of $\tilde{\mathbf{A}}_{n-t+1}$ corresponding to  $V^{=d}$. By H\"older's inequality, we have
    \[
         \left| \langle f^{=d},g\rangle  \right|  \leq \|f^{=d}\|_{\ell} \|g\|_{\ell'} 
    \]
    where we assume that $\ell \geq 4$ is a power of 2. Since $f$ is both $(d,\epsilon_1,L_1)$-restriction global and $(d,\g^d)$-restriction global (here, $L_2$-globalness is trivial since $\|f_{(V',W')\rightarrow T}\|_2 \le 1 \le \gamma^d$), by Proposition~\ref{prop:1-2-log-convexity} we have that $f$ is $(d,\epsilon_{\ell '},L_{\ell'})$-restriction global with 
    \begin{align*}
        \varepsilon_{\ell'} &= \g^{d/\ell'} \|f\|_1^{(\ell-2)/\ell}.
    \end{align*}  
    Corollary~\ref{cor:global} implies for all $d$ that $f^{=d}$ is 
    $(d,q^{510d^2\ell}\varepsilon_{\ell'}^2)$-restriction global.

    For $d \leq \frac{1}{4}\sqrt{mt}$, we claim that 
    $$\left| \langle f^{=d},g^{=d}\rangle  \right| \leq q^{1608cD\sqrt{mt}d}  \E[f]\E[g]$$
for some constant $D>1$. To see this, we may apply  Lemma~\ref{lem:lvld}, which gives 
    \begin{align*}
         \left| \langle f^{=d},g^{=d}\rangle  \right| = \left| \langle f^{=d},g \rangle  \right|   &\leq \|f^{=d}\|_{\ell} \|g\|_{\ell'}\\
         &\leq  q^{710d^2\ell}\g^{d} \|f\|_2^{2/\ell} \|f\|_1^{1-4/\ell} \|g\|_{\ell'}\\
        &\leq  q^{710d^2\ell}\g^{d} 
        \|f\|_1^{1-4/\ell} \|g\|_{\ell'}
        \intertext{where in the last inequality we used the fact that $\|f\|_2 \leq 1$. Setting $p_0 = 2$, $p_1 = 1$, and $p_\theta = \ell'$ in Theorem~\ref{thm:log-convex} applied to $g$ gives}
      &\leq  q^{710d^2\ell}\g^{d} 
        \|f\|_1^{1-4/\ell} \|g\|_2^{2/\ell} \|g\|_1^{1-2/\ell}\\
        &\leq  q^{710d^2\ell}\g^{d} 
        (\|f\|_1\|g_1\|)^{1-4/\ell},
        \intertext{where in the last inequality we used that $0 \le g\le 1$ and which implies $\|g\|_2, \|g\|_1\le 1$. Since $f$ and $g$ are non-negative, we have}
        &=  q^{710d^2\ell}\g^{d} 
        \left( \E[f]\E[g] \right)^{1-4/\ell}\\
        &\leq  q^{710d^2\ell}\g^{d} q^{8cmt/\ell}  \E[f]\E[g].
        \intertext{
        Recall that $\g = q^{\tilde{c}\beta '\xi t}$. Let $D :=\tilde{c}\beta '\xi > 1$, so that $\g = q^{Dt}$. For $d \le \frac{1}{4}\sqrt{mt}$, pick $1 \leq C < 2$ so that $\ell := C\sqrt{mt}/d$ is a power of 2. We have
        }
        &\leq  q^{((710C+\frac{8c}{C})\sqrt{mt} + Dt)d}  \E[f]\E[g]\\
        &\leq  q^{((1600 + 8c)\sqrt{mt} + D t)d}  \E[f]\E[g]\\
         &\leq  q^{1608cD\sqrt{mt}d}  \E[f]\E[g],
    \end{align*}
which proves the claim.

For sufficiently small $\alpha > 0$, we have $d \leq 10cDt \leq \frac{1}{4}\sqrt{mt}$, so applying the claim above to each $d \leq 10cDt$ along with 
Proposition~\ref{prop:eigs} gives us 
    \begin{align} \label{eq:EfEg}
    \E[f]\E[g] &\leq \sum_{d=1}^n |\eta_d|~   \left| \langle f^{=d},g^{=d}\rangle  \right|  \\
  &\leq  \sum_{d=1}^{10cDt}  q^{(1608cD\sqrt{mt})d}
 \cdot |\eta_d| \cdot \E[f]\E[g]  + q^{2cmt} \!\!\! \! \sum_{d = 10cDt+1}^n  O( q^{-(d(2m-d)-t(2m-2n+t))/2}) \E[f]\E[g].\\ 
      \intertext{Here, we have used the trivial bound $\left| \langle f^{=d},g^{=d}\rangle  \right| \le 1 \le q^{2cmt}\E[f]\E[g]$ for $d>10cDt$. Breaking the first summation above into two parts based on $t$, and then using the fact that $d \leq 10cDt$ in the second summation below gives us}
    \label{eq:EfEg1} &\leq  \left( \sum_{d=1}^{t-1} q^{(1608cD\sqrt{mt})d} O\left (q^{-d(n-t) - d + \log_q n} \right)\right) \E[f]\E[g]  \\ 
        \label{eq:EfEg2} &\quad \quad + \quad  \left( \sum_{d=t}^{10cDt} q^{(16080c^2D^2\sqrt{mt})t}  O\left (q^{-(d(2m-d)-t(2m-2n+t))/2} \right)\right) \E[f]\E[g] \\ 
       \label{eq:EfEg3} &\quad \quad + \quad  \left( \sum_{d=10cDt+1}^n q^{2cmt}  O\left( q^{-(d(2m-d)-t(2m-2n+t))/2} \right) \right) \E[f]\E[g].
    \end{align}
Recall that $t \leq \alpha n$. We now show that each of the three terms above are less than $\E[f]\E[g]/3$ for sufficiently small $\alpha > 0$ (depending on the constants $\beta$, $c$, and $D = \tilde{c}\beta'\xi$) and sufficiently large $n_0$.

\begin{itemize}
    \item For the term (\ref{eq:EfEg1}), we have 
    \[
     \sum_{d=1}^{t-1} q^{(1608cD\sqrt{mt})d} O\left (q^{-d(n-t) - d + \log_q n} \right) \leq \sum_{d=1}^{t-1} q^{\log_q n + (1608cD\sqrt{mt})d} O(q^{-(n/2)d - d}) < 1/3
    \]
for sufficiently small $1/2 \geq \alpha > 0$ and sufficiently large $n_0$.

\medskip

\item For the term~(\ref{eq:EfEg2}), we have 

\begin{align*}
\sum_{d=t}^{10cDt} q^{(16080c^2D^2\sqrt{mt})t}  O\left (q^{-(d(2m-d)-t(2m-2n+t))/2} \right) &\leq \sum_{d=t}^{10cDt} q^{(16080c^2D^2\sqrt{\beta t})t \sqrt{n}}  O\left (q^{-(t(2m-d)-t(2m-2n+t))/2} \right) \\ 
&\leq \sum_{d=t}^{10cDt} q^{(16080c^2D^2\sqrt{\beta t})t \sqrt{n}}  O\left (q^{-(-td+2tn-t^2)/2} \right)\\
&\leq \sum_{d=t}^{10cDt} q^{(16080c^2D^2\sqrt{\beta t})t \sqrt{n}}  O\left (q^{-tn + 11cDt^2/2} \right)\\
 & \leq q^{\log_q n + 11cDt^2/2 + (16080c^2D^2\sqrt{\beta})t \sqrt{nt} }  O\left (q^{-tn} \right) < 1/3\\
\end{align*}
for sufficiently small $\alpha > 0$ and sufficiently large $n_0$.

\medskip

\item For the term~(\ref{eq:EfEg3}), we have 

\begin{align*}
\sum_{d=10cDt+1}^n q^{2cmt}  O\left( q^{-(d(2m-d)-t(2m-2n+t))/2} \right)  &\leq q^{3cmt}  O\left( q^{-10cDtm + 5cDtd + tm -tn + t^2/2}\right) \\
        &\leq  q^{3cmt}  O\left( q^{-10cDtm + 5cDtd + tm}\right)\\
        &\leq  q^{3cmt}  O\left( q^{-4cDtm}\right) = O(q^{-cDtm}) < 1/3,
\end{align*}

for sufficiently small $\alpha > 0$ and sufficiently large $n_0$.

\end{itemize}

\medskip

The foregoing shows for sufficiently small $\alpha$ that the right-hand side of (\ref{eq:EfEg}) is less than $\mathbb{E}[f]\mathbb{E}[g]$, a contradiction.
\end{proof}

\section{Step 2: No Density Bump Within a Dictatorship}\label{sec:step2}
Step~2 has a simple conceptual goal: starting from cross-$(t-1)$-intersection-free families
$\mathcal{F},\mathcal{G}\subseteq\mathcal{L}(V,W)$ with large measure, we pass to a $d$-restriction for a \emph{small} $d$
after which $\mathcal{F}$ becomes global, while the cross-$(t-1)$-intersection-free
condition is retained. Once this is achieved, Step~1 (Theorem~\ref{thm:step1}) can be applied to show that $\mathcal{F},\mathcal{G}$ are small.

Throughout we assume a weak pseudorandomness hypothesis on $\mathcal{G}$: it does not exhibit a
substantial density increase under any $1$-restriction (i.e., inside any dictator). Under this assumption
we will show that if both $\mathcal{F}$ and $\mathcal{G}$ have sufficiently large measure, we want to find restrictions,
$(V',W',T'), (V',W',T)$, of size at most $n/2$
such that:
\begin{itemize}
    \item the restricted family $\mathcal{F}_{(V',W') \rightarrow T'}$ is global,
    \item the restricted families $\mathcal{F}_{(V',W')\to T'}$ and $\mathcal{G}_{(V',W')\to T}$ remain cross-$(t-1)$-intersection-free.
\end{itemize}
\noindent In its current form, this idea fails because the families do not remain cross-$(t-1)$-intersection-free. We resolve this issue by projecting onto a suitable subspace.
\noindent The technical work in this step involves finding the restrictions and projections that satisfy these conditions.

\medskip

\subsection*{Preserving forbidden intersections under linear restrictions}
At a high level, our strategy for finding this pair of restrictions is to `force' $\mathcal{F}$ to become global by passing to a suitable restriction via Lemma~\ref{lem:global_restriction}. In product spaces, such an argument is relatively straightforward: taking coordinate restrictions preserves the relevant forbidden-intersection structure (see, e.g.,~\cite{KLS25}). However, in the present linear-algebraic setting, restrictions are not simple coordinate projections. Passing to a restriction corresponds to imposing linear constraints, and although the families  $\mathcal{F}_{(V',W')\to T'}$ and $\mathcal{G}_{(V',W')\to T}$ can be viewed as subfamilies of $\LL(V,W)$ and be cross-$(t-1)$-intersection-free there, it is not true that the cross-$(t-1)$-intersection-free property is preserved when viewed as subfamilies of $\LL(V/V',W')$.

A significant part of the proof is devoted to choosing the restriction and the associated subspaces so that
the families stay cross-$(t-1)$-intersection-free. We analyze how forbidden intersection conditions behave under restrictions, and use new projections and restrictions to appropriate subspaces to ensure that the intersection parameter is preserved. This is one of the main technical sources of complexity in this section.

\medskip

\subsection*{Maintaining globalness after passing to a smaller space.} A second subtlety arises from globalness. Even if $\mathcal{F}_{(V',W') \rightarrow T}$ is global, it may not remain global when viewed as a family of a smaller space $\mathcal{L}(\widetilde{V},\widetilde{W})$ 
obtained after projection (required to preserve the intersection size). For this reason, rather than working directly with characteristic functions of projected families, we pass to their \emph{fiber averages}. That is, instead of considering the projection of the indicator functions of $\mathcal{F}_{(V',W')\to T'}, \mathcal{G}_{(V',W')\to T}$, we define averaged functions
\[
\widetilde f,\widetilde g:\mathcal{L}(\widetilde V,\widetilde W)\to[0,1]
\]
obtained by averaging $f,g$ over the fibers of a projection map $\Gamma$. This operation preserves measure and
(crucially) preserves $L_1$-globalness, which is exactly what Step~1 requires.

With $\widetilde f,\widetilde g$ in hand, we apply Theorem~\ref{thm:step1} inside
$\mathcal{L}(\widetilde V,\widetilde W)$, forcing their product measure
$\E[\widetilde f]\E[\widetilde g]$ to be small, which causes $\mathcal{F},\mathcal{G}$ to have small measure. 

More formally, for a contradiction, suppose that the indicator functions 
$f,g \in L^2(\mathcal{L}(V,W))$ of cross-$(t-1)$-intersection-free families have large measure, meaning:
\begin{align}\label{eq:contra}
\E[f]\E[g] \ge q^{-2\tilde{c}mt}.
\end{align}
We further assume that for all $1$-restrictions $(V',W',S)$
we have 
$$\E[g_{(V,W') \rightarrow S}] \leq q^{\xi t} \E[g].$$
Under this assumption, we construct spaces 
$\widetilde{V}$, $\widetilde{W}$ and restricted fiber-average functions 
$\widetilde{f},\widetilde{g}\in L^2(\mathcal{L}(\widetilde{V},\widetilde{W}))$ satisfying the hypotheses of Theorem~\ref{thm:step1} such that
\[
\E[\widetilde{f}]\,\E[\widetilde{g}]
\;\ge\;
\tfrac12\,\E[f]\,\E[g].
\]
Applying Theorem~\ref{thm:step1} then implies that for all $c>1$,
\[
\E[\widetilde{f}]\,\E[\widetilde{g}]
\;\le\;
2q^{-2c\tilde{m}t},
\]
where $\tilde{m}=\dim\widetilde{W}\ge n/2$. 
Thus, for $c>1$ sufficiently large (and $\alpha$ sufficiently small, $n$ sufficiently large), this contradicts~\eqref{eq:contra}.

Motivated by this discussion, the section is devoted to proving the following proposition.

\begin{prop}\label{prop:main step 2}
Fix $\beta' \ge 1$ such that $m \le \beta' n$. 
For all $\tilde{c}>0$, there exists $\alpha>0$ 
depending on $\xi \geq 8$, $\tilde c$, and $\beta'$ such that if $t \le \alpha n$, then the following holds. 

Let $\mathcal{F},\mathcal{G} \subseteq \mathcal{L}(V,W)$ be cross-$(t- 1)$-intersection-free with corresponding indicator functions $f,g \in L^2(\LL(V,W))$. Furthermore, assume for every $1$-restriction $(V',W',S)$ we have $$\E[g_{(V',W') \rightarrow S}] \leq q^{\xi t} \E[g].$$ 
Then $\E[f]\E[g] \leq q^{-2\tilde{c}mt}$. 
\end{prop}

Before proving the proposition, we first need to state and prove a few lemmas.

\noindent Recall that a strong $2d$-restriction is a restriction $(V',W',T)$ such that $\dim V' = \codim W' = d$, that is, we fix the entries of $d$ rows and $d$ columns in $T$, a total of $d(n + m - d)$ entries. 

\begin{lem}\label{prop:res_size}
Let $\beta'$ and $m$ be as in Proposition~\ref{prop:main step 2}.
Let $f,g$ be the indicator functions of $\mathcal{F},\mathcal{G} \subseteq \LL(V,W)$ and assume that $q^{-2\tilde{c}mt} \leq \E[f]\E[g]$ for some constant $\tilde{c}>0$. Let $\gamma = q^{\tilde{c}\beta' \xi t}$ such that $\xi \geq 8$, and let $\mathcal{F}_{(V',W')\rightarrow T'} \subseteq \LL(V,W)$ be a $\gamma$-global strong restriction of $\mathcal{F}$ of size $2s$. Then $s \leq 2n/\xi$.
\end{lem}
\begin{proof}
Since $\mathcal{F}_{(V',W')\rightarrow T'}$ is a global restriction, by Lemma~\ref{lem:global_restriction} we have 
\[
\gamma^{s}\E[f]  \leq \E[f_{(V',W')\rightarrow T'}] \leq 1.
\]
Since $q^{-2\tilde{c}mt} \leq \E[f]\E[g] \leq \E[f]$, taking logs shows that $s \leq 2m/\beta'\xi \leq 2n/\xi$, as desired. 
\end{proof}
For any two families $\mathcal{A},\mathcal{B} \subseteq \LL(V,W)$, define their \emph{difference} to be
\[
    \mathcal{A} - \mathcal{B} := \{ A-B : A \in \mathcal{A}, B \in \mathcal{B}\}.
\]
Under the proper isomorphism, we may assume the difference $\mathcal{F}_{(V',W')\rightarrow T'} - \mathcal{G}_{(V',W')\rightarrow T}$ has a simple form, as shown in the lemma below.

\begin{lem}\label{prop:change-of-basis}
    Let $\mathcal{F}_{(V',W')\rightarrow T'}$ and $\mathcal{G}_{(V',W')\rightarrow T}$ be strong $2s$-restrictions. Then there exists an isomorphism of $\LL(V,W)$  such that any element of $\mathcal{F}_{(V',W')\rightarrow T'} - \mathcal{G}_{(V',W')\rightarrow T}$ under this isomorphism is of the form
    \[ 
        \begin{bmatrix}
            I_u & 0 & 0  \\
            0 & 0_{s-u} & X' \\
            0 & X & * 
        \end{bmatrix}
    \]
    for some fixed matrices $X,X'$ that depend only on $T,T'$, and $*$ indicates a block matrix with arbitrary entries. Moreover, this isomorphism of $\LL(V,W)$ preserves globalness.
\end{lem}
\begin{proof}
    Each of the strong $2s$-restrictions are with respect to the same subspaces $V',W'$, so we may perform a change of basis so that the first $s$ columns and $s$ rows of every element of $\mathcal{F}_{(V',W')\rightarrow T'} - \mathcal{G}_{(V',W')\rightarrow T}$ is of the form
    \[
            \begin{bmatrix}
            X_1 & X_2 \\
            X_3 & *
        \end{bmatrix}
    \]
    where the entries of the $s \times s$ block $X_1$, $s \times (n-s)$ block $X_2$, and $(m-s) \times s$ block $X_3$ are fixed constants given by $T'-T$.
    We may apply elementary row and column operations to the first $s$ rows and first $s$ columns so that the elements of $\mathcal{F}_{(V',W')\rightarrow T'} - \mathcal{G}_{(V',W')\rightarrow T}$ have the form
    \[
            \begin{bmatrix}
            I_u & 0 & Y'  \\
            0 & 0_{s-u} & X' \\
            Y & X &  *
        \end{bmatrix}
    \]
     where $u$ is the rank of the $s \times s$ block $X_1$, and $X,Y,X',Y'$ are fixed blocks that depend only on $T$ and $T'$. We may apply further elementary row operations on the first $u$ rows and elementary column operations on the first $u$ columns so that every element of the difference can be written as 
    \[
            \begin{bmatrix}
            I_u & 0 & 0  \\
            0 & 0_{s-u} & X' \\
            0 & X &  * 
        \end{bmatrix} \quad + \quad 
        \begin{bmatrix}
            0_u & 0 & 0  \\
            0 & 0_{s-u} & 0 \\
            0 & 0 &  Z
        \end{bmatrix}
    \]
    where the $(m-s) \times (n-s)$ block $Z$ depends only on $Y,Y'$ and the elementary row operations performed on the first $u$ rows and on the elementary column operations performed on the first $u$ columns. 
    The second part now follows from the general fact that, by symmetry, a family $\mathcal{F} \subseteq \LL(V,W)$ is $\gamma$-global if and only if $L\mathcal{F}R+Z' = \{LAR + Z' : A \in \mathcal{F}\}$ is $\gamma$-global for any invertible linear transformations $L \in \mathrm{GL}_m(\mathbb{F}_q)$, $R \in \mathrm{GL}_n(\mathbb{F}_q)$, and arbitrary linear transformation $Z' \in \LL(V,W)$.
\end{proof}

\begin{lem}\label{prop:union bound}
Let $s,\xi \in \mathbb{N}$ such that $\xi \geq 8$, and $\mathcal{F} \subseteq \mathcal{L}(V,W)$. Let $V' \leq V$ and $W' \leq W$ be subspaces such that $\dim V' + \mathrm{codim}~W' = 1$. Let $\mathcal{S}$ be a collection of at most $ 2q^{s}$ 1-restrictions $(V',W',S)$. If 
$$|\mathcal{F}_{(V',W')\rightarrow S}| \leq \frac{1}{q^{n-\xi t}} |\mathcal{F}|$$
for every $(V',W',S) \in \mathcal{S}$, then for sufficiently large $n$, we have
$$\left |~ \mathcal{F} \setminus \bigcup_{(V',W',S) \in S} \mathcal{F}_{(V',W')\rightarrow S} ~\right | \geq | \mathcal{F} |/2$$
provided that $t \leq \alpha n \leq n/\xi^2$ and $s \leq 2n/\xi$.
\end{lem}
\begin{proof}
By the union bound, we have 
\begin{align*}
        \left | ~\mathcal{F} \setminus \bigcup_{(V',W',S) \in \mathcal{S}} \mathcal{F}_{(V',W')\rightarrow S}~\right| 
        \geq |\mathcal{F}| -  \sum_{(V',W',S)  \in \mathcal{S}}   |\mathcal{F}_{(V',W')\rightarrow S}| &\geq \left( 1-\frac{2q^{s}}{q^{n-\xi t}} \right) |\mathcal{F}|\\
        &\geq \left( 1-\frac{q^{1 + 2n/\xi}}{q^{(1-\xi \alpha)n}} \right) |\mathcal{F}|\\
        &\geq |\mathcal{F}|/2,
\end{align*}
as desired.
\end{proof}

\begin{proof}[Proof of Proposition~\ref{prop:main step 2}]
For a contradiction, suppose there exists a $\tilde{c} > 0$ 
such that 
$$q^{-2\tilde{c}mt} \leq \E[f]\E[g].$$ 
Let $\gamma = q^{\tilde{c}\beta' \xi t}$. Lemma~\ref{lem:global_restriction} gives us a $\gamma$-global strong restriction of size $2s$, $(V',W',T')$ of $\mathcal{F}$,
and by Lemma~\ref{prop:res_size} $s \leq 2n/\xi$.  

Now consider any strong restriction $(V',W',T)$ of $\mathcal{G}$.
Henceforth, we work tacitly under the isomorphism given by Lemma~\ref{prop:change-of-basis}, and we let $v_1,\ldots,v_n \in V$ be the ordered basis given by Lemma~\ref{prop:change-of-basis} such that $v_1$ indexes the first column, $v_2$ indexes the second column, and so on.  

Define the subspaces 
$$V_0 := \spn{v_1,\ldots,v_u}, V_1 := \spn{v_{u+1},\ldots,v_s},V_2 := \spn{v_{s+1},\ldots,v_n}.$$ 
For any $v \in V$, we may write $$v = \begin{pmatrix} x\\y\\z \end{pmatrix}$$ 
where $x \in V_0$, $y \in V_1$, and $z \in V_2$. From Lemma~\ref{prop:change-of-basis} it follows that $Bv=Av$ if and only if $v$ is a solution to the following system of linear equations:
\[
            \begin{pmatrix}
            I_u & 0 & 0  \\
            0 & 0_{s-u} & X' \\
            0 & X &  * +Z
        \end{pmatrix} 
        \begin{pmatrix} x\\y\\z \end{pmatrix} = 0.
\]
Writing in terms of equations and slight rearrangement gives us the following:
\begin{align}
    x &= 0\\
    X'z  &= 0\\
    (A'-B')z &= Xy 
\end{align} 
where $A',B'$ are the submatrices of $A,B$ obtained by deleting their first $s$ rows and $s$ columns. The foregoing shows that $A,B$ agree on $v$ if and only if $x=0$, $z$ lies in the kernel of $X'$, and $(A'-B')z$ lies in the column space of $X$.

The linear system of equations above implies that 
\[
    \dim \ker (A-B) = \dim~\spn{ z \in V_2 : z \in \ker X' \text{ and } (A'-B')z \in \mathrm{colspace}~X }.
\]
Moreover, if the matrix $T$ in the strong restriction $(V',W',T)$ of $\mathcal{G}$ is chosen so that $X,X'$ both have full rank, then we deduce that $A,B$ are cross-$(t-1)$-intersection-free if and only if $A',B' \in \mathbb{F}_q^{(m-s)\times (n-s)}$ are cross-$(t-1)$-intersection-free on the subspace of vectors $z \in V_2$ such that $z \in \ker X' \text{ and } (A'-B')z \in \mathrm{colspace}~X$. 

Let $\widetilde{V} := \ker X'$ and $\widetilde{W} := \mathbb{F}_q^{m-s}/\mathrm{colspace}~X$. Let $\tilde{m} := \dim \widetilde{W}$, $\tilde{n} := \dim \widetilde{V}$, and note that 
$$\widetilde{m} = m-2s+u \geq n/2.$$ 
Define the linear transformation
\[
\Gamma : \mathbb{F}_q^{(m-s)\times (n-s)} \rightarrow \mathcal{L}(\widetilde{V},\widetilde{W}) \quad \text{ such that } \quad \Gamma(A') = Q_{\mathrm{colspace}~X}\circ A'|_{\ker X'} 
\]
for all $A' \in \mathbb{F}_q^{(m-s)\times (n-s)}$ where $\circ$ denotes composition and $Q_{\mathrm{colspace}~X} : \mathbb{F}_q^{m-s} \rightarrow \mathbb{F}_q^{m-s} / \mathrm{colspace}~X$ is the natural quotient map $x \mapsto x + \mathrm{colspace}~X$. 
Note that for $A',B'\in \mathbb F_q^{(m-s)\times(n-s)}$ we have
\[
\ker(\Gamma(A')-\Gamma(B'))
=
\{z\in \ker X' : (A'-B')z\in \operatorname{colspace}X\}.
\]
Consequently, for the maps $A,B\in \LL(V,W)$ corresponding to $A',B'$ respectively, we have
\[
\dim\ker(A-B)=\dim\ker(\Gamma(A')-\Gamma(B')).
\]
Thus the pair of restricted families is cross-$(t-1)$-intersection-free
in $\LL(V,W)$ if and only if their images under $\Gamma$ are
cross-$(t-1)$-intersection-free in $\LL(\widetilde V,\widetilde W)$.

Recall that, as \(T\) varies, the strong \(2s\)-restrictions
\((V',W',T)\) partition \(\LL(V,W)\). We now keep only those restrictions
for which the reduction above preserves the intersection parameter.
More precisely, let \(\mathcal T\) be the collection of strong
\(2s\)-restrictions \((V',W',T)\) for which the corresponding matrices
\(X\) and \(X'\) both have full rank, while \(T'\) remains fixed.
For every restriction in \(\mathcal T\), the preceding discussion shows that, for every pair
$A\in \mathcal{F}_{(V',W')\to T'}, B\in \mathcal{G}_{(V',W')\to T},$
the corresponding pair $\Gamma(A'),\Gamma(B')$ in \(\LL(\widetilde V,\widetilde W)\) have the same intersection size. Meaning, 
\[
\dim\ker(A-B)
=
\dim\ker(\Gamma(A')-\Gamma(B')).
\]
Equivalently, no additional
intersection can arise from the row and column directions that were
fixed by the restriction (or in other words the intersection does not arise from intersection in $T,T'$).

It remains to discard the part of \(\mathcal{G}\) lying in the complementary
restrictions. If \((V',W',T)\notin \mathcal T\), then either \(X\) or
\(X'\) is not full rank. In the first case, there is a nonzero vector
\(v\in\ker X\), and every element of this restricted piece of \(\mathcal{G}\)
lies in a column dictator of the form $\mathcal{G}_{v\mapsto T'v}$ (setting $Tv=T'v$),
after the harmless
translation already implicit in the normalization. In the second case,
there is a nonzero vector \(w\in\ker (X')^{\top}\), and the restricted
piece lies in a row dictator of the form $\mathcal{G}_{ w \uparrow (T')^\top w}$ (setting $T^\top w =(T')^\top w $). Let
\(\mathcal S\) denote the collection of all such dictators arising from
bad restrictions. Since there are at most \(2q^s\) of them, we may
remove their union and set
\[
\mathcal{G}' := \mathcal{G} \setminus
\bigcup_{(V',W',S)\in\mathcal S} \mathcal{G}_{(V',W')\to S}.
\]
By Lemma~\ref{prop:union bound}, this removal loses at most half of \(\mathcal{G}\), and hence
\(|\mathcal{G}'|\ge |\mathcal{G}|/2\). Moreover, every element of \(\mathcal{G}'\) belongs to one of
the good restrictions in \(\mathcal T\). 
Since the restrictions of $\mathcal{T}$ partition $\mathcal{G}'$, we have 
$$
|\mathcal{G}|/2 \leq |\mathcal{G}'| = \! \! \! \sum_{(V',W',T) \in \mathcal{T}} \! \! \! |\mathcal{G}_{(V',W') \rightarrow T}| \leq |\mathcal{T}| \max_{(V',W',T) \in \mathcal{T}} |\mathcal{G}_{(V',W') \rightarrow T}| \leq q^{sm+sn - s^2} \! \! \! \max_{(V',W',T) \in \mathcal{T}} |\mathcal{G}_{(V',W') \rightarrow T}|; 
$$
therefore, we may pick a restriction $(V',W',T) \in \mathcal{T}$ that satisfies $$\mu(\mathcal{G})/2 \leq \mu^{(V/V',W')}(\mathcal{G}_{(V',W') \rightarrow T}).$$
Since $\mathcal{F}_{(V',W') \rightarrow T'}$ is a global restriction of $\mathcal{F}$, we have $$\mu(\mathcal{F}) \leq \mu^{(V/V',W')}({\mathcal{F}}_{(V',W') \rightarrow T'}).$$
Together, we have 
\[
   \mu(\mathcal{F}) \mu(\mathcal{G})/2 \leq \mu^{(V/V',W')}({\mathcal{F}}_{(V',W') \rightarrow T'}) \cdot \mu^{(V/V',W')}(\mathcal{G}_{(V',W') \rightarrow T)}.
\]

\medskip

\noindent Unfortunately, the map $\Gamma$ is not measure-preserving, since it is not injective, hence $$\Gamma(\mathcal{F}_{(V',W') \rightarrow T'}) =: \widetilde{\mathcal{F}}_{(V',W') \rightarrow T'}$$
may not be a $\gamma$-global family of $\mathcal{L}(\widetilde{V},\widetilde{W})$. Therefore, we cannot immediately apply Theorem~\ref{thm:step1} to their respective indicator functions. 

To get around this, we instead consider \emph{fiber-average functions} $\tilde{f},\tilde{g} : \LL(\widetilde{V},\widetilde{W}) \rightarrow [0,1]$, defined below. 

For each $\tilde{A} \in \LL(\widetilde{V},\widetilde{W})$, let $\Gamma^{-1}(\tilde{A}) \subseteq \mathbb{F}_q^{(m-s) \times (n-s)}$ denote the preimage of $\tilde{A}$. Define $\tilde{f},\tilde{g} : \LL(\widetilde{V},\widetilde{W}) \rightarrow [0,1]$ such that 
\[
    \tilde{f}(\tilde A) = \frac{|\Gamma^{-1}(\tilde A) \cap {\mathcal{F}}_{(V',W') \rightarrow T'}|}{|\Gamma^{-1}(\tilde A)|} \quad  \text{ and } \quad   \tilde{g}(\tilde A) = \frac{|\Gamma^{-1}(\tilde A) \cap {\mathcal{G}}_{(V',W') \rightarrow T}|}{|\Gamma^{-1}(\tilde A)|}
\]
for all $\tilde A \in \LL(\widetilde{V},\widetilde{W})$ where $\mathcal{F}_{(V',W') \rightarrow T'}, \mathcal{G}_{(V',W') \rightarrow T} \subseteq \mathbb{F}_q^{(m-s) \times (n-s)}$.  By definition, we have $\|\tilde{f}\|_1 = \|f_{(V',W')\rightarrow T'}\|_1$, and similarly for $g$. For any $d \in \mathbb{N}$ and $d$-restriction $(V'',W'',T'')$ of $\LL(\widetilde{V},\widetilde{W})$, we have 
$$\|\tilde{f}_{(V'',W'') \rightarrow T''} \|_1 = \|f'_{(V'',W'') \rightarrow T''}\|_1;$$ 
where $f'$ is the indicator function of $\mathcal{F}_{(V',W') \rightarrow T'}$. Thus $\tilde{f}$ is $(d,\gamma^d \|\tilde{f}\|_1,L_1)$-restriction global since $f'$ is $(d,\gamma^d \|f\|_1,L_1)$-restriction global. Moreover, we have that $\langle \tilde{f}, \tilde{\mathbf{A}}_{\tilde{n}-t+1}~ \tilde{g} \rangle = 0$ if and only if $\langle f, \tilde{\mathbf{A}}_{n-t+1}~ g \rangle = 0$. 

Recall that $\tilde{m} \geq n/2$. Applying Theorem~\ref{thm:step1} (with $\beta = 2\beta'$, $\tilde{c} = \tilde{c}$, $\beta' = \beta'$, $c = \tilde{c}4\beta'$) to $\tilde{f},\tilde{g}$ with $\varepsilon_1 = \gamma^d \|\tilde{f}\|_1$ gives
\[
  q^{-2\tilde{c}\beta'n t-1} \leq q^{-2\tilde{c}\beta'n t}/2  \leq q^{-2\tilde{c}m t}/2  \leq \E[f]\E[g]/2 \leq \E[\widetilde{f}]\E[\widetilde{g}] \leq q^{-2c\tilde{m}t} \leq q^{-cnt} 
\]
 Thus
\[
1 \leq \frac{q^{2\tilde{c}\beta'n t+1}}{q^{cnt}} = q^{(2\tilde{c}\beta'-c)nt +1},
\]
which gives a contradiction, completing the proof. 
\end{proof}


\section{Step 3: Density Bump Within a Dictatorship}
In this section we complete the structural part of the argument. Step~2 (Proposition~\ref{prop:main step 2}) shows that, given a pair of cross $(t-1)$ intersection-free families $\mathcal{F},\mathcal{G}\subseteq\mathbb{F}_q^{m \times n}$, unless they are already small, one obtains a density increment inside a dictatorship. We assume that $|\mathcal{F}|\cdot|\mathcal{G}|$ is `large' and finish via the following stages:

\begin{enumerate}
    \item \emph{From a density bump to a common dictatorship.}
    We first show that a nontrivial density increase inside a dictatorship
    can be upgraded to almost containment of both families in this same
    dictatorship. This uses an induction hypothesis at codomain dimension
    $m-1$. We also use Lemma~\ref{lem:density_bump_on_smaller_side_imply_too_large}
    to rule out row restrictions when $m>n$, and prove that the relevant dictatorship
    is a column restriction.

    \item \emph{Iteration and bootstrapping.}
    We then apply the previous step repeatedly. Lemma~\ref{lem:iterative_argument}
    shows that, after $t$ iterations, both families are almost contained in
    the same $t$-umvirate (or, when $m=n$, in the same dual $t$-umvirate).
    Finally, Lemma~\ref{lem:counting_t-1_free_with_t_umvirate} and Lemma~\ref{lem:number of different t umvirate families with t-j intersection with U} 
    are used together in Section~\ref{subsec:proof of inductive lemma} to bootstrap this
    approximate structure into the exact product bound, showing that the maximum of $|\mathcal{F}|\cdot|\mathcal{G}|$ is attained
precisely when both families are the same $t$-umvirate, with the additional
possibility of the same dual $t$-umvirate in the square case.
\end{enumerate}
After we prove the previous stages for $\mathcal{F},\mathcal{G}\subseteq\mathbb{F}_q^{m \times n}$, we get almost `for free' that for large cross-$(t-1)$-intersection-free families, $\mathcal{F}$ and $\mathcal{G}$ inside other matrix groups (e.g., $\GL(n,q)$, $\SL(n,q)$), both are almost fully contained in the same $t$-umvirate. We then just have to perform a bootstrapping argument inside each matrix group (which is the same heuristic argument with different technical details), to get that the maximum $(t-1)$-intersection-free family of $\GL(n,q)$ is a $t$-umvirate.
\begin{definition}\label{def:maximal_cross_intersection_free_families}
    For parameters $t\leq n\leq m$ and field size $q$, 
    we let $\IB{q}{n}{m}{t}$ be the maximum product sizes of cross-$(t-1)$-intersection-free families over $\mathcal{L}(V,W)$ for $V=\mathbb{F}_q^n$ and $W=\mathbb{F}_q^m$:  
    \[\IB{q}{n}{m}{t}:=\max_{\mathcal{F,G}: \text{ cross-}(t-1)\text{-intersection-free}}{|\mathcal{F}||\mathcal{G}|}.\]
\end{definition}
If $\mathcal{U}\subseteq \LL(V,W)$ is a $t$-umvirate, then $|\mathcal{U}| = q^{m(n-t)}$, and in this situation $\mathcal{U}$ is also $(t-1)$-intersection-free. Consequently,
\[
\IB{q}{n}{m}{t}\geq |\mathcal{U}|^2 = q^{2m(n-t)}\ .
\]
Our goal is to  prove that the above lower bound is tight. Formally, 
\begin{thm}\label{thm:main theorem step 3}
 There exist constants $a>0$ and $b> 1$ such that for all $1\le t\le a n$ and $n\le m\le b n$ we have 
 $$\IB{q}{n}{m}{t}=q^{2m(n-t)}\ .$$   
\end{thm}
In the course of proving the above theorem, we also establish a stability result: 
if the product measure of two cross-$(t-1)$-intersection-free families is large, then each of these families is almost entirely contained in a $t$-umvirate (or in its dual when $n = m$). In addition, we prove that the only extremal families are $t$-umvirates (or dual $t$-umvirates when $n = m$).

Theorem~\ref{thm:main theorem step 3} is an immediate consequence of the following inductively formulated lemma.
We begin with a very crude estimate (in fact, the trivial density bound of $1$, or even worse) and then progressively sharpen it as $n$ and $m$ increase.
\begin{lem}\label{lem:inductive_statment}
There exists a constant $C>2$ such that for all $t,n,m\in \N$ with $1\le t\le n\le m $ we have
    $$\IB{q}{n}{m}{t}\le
     q^{2m(n-t)+\max\{2t(Ct-m),0\}+\max\{2t(2m-4(n-t)),0\}}  \ .$$
\end{lem}

We first show how this easily implies the main theorem and the rest of this section will be devoted to proving Lemma~\ref{lem:inductive_statment}.
\begin{proof}[Proof of Theorem~\ref{thm:main theorem step 3}.]
Fix $C$ from Lemma~\ref{lem:inductive_statment}. Note that for $ Ct\le m\le 2(n-t)$ we have $$\IB{q}{n}{m}{t}\le
     q^{2m(n-t)+\max\{2t(Ct-m),0\}+\max\{2t(2m-4(n-t)),0\}}= q^{2m(n-t)}  \ .$$
   Let $a=1/C$ and $b=2(1-a)$. Since $C> 2$ we have $a< 1/2$ and $b > 1.$
   If $1\le t\le a n$ and $n\le m\le  b n$, then $m\ge n\ge Ct$ and $m\le b n=2(1-a)n=2(n-a n)\le 2(n-t)$. This gives us $\max\{2t(Ct-m),0\} = \max\{2t(2m-4(n-t)),0\} = 0$, which finishes the proof.
\end{proof}

Most of this section is dedicated to proving this main inductive estimate,
Lemma~\ref{lem:inductive_statment}. We prove that lemma by a strong induction on
the codomain dimension $m$, proving the statement simultaneously for all
triples $(t,n,m)$ with $1\le t\le n\le m$. Thus, throughout the induction
step, we may assume Lemma~\ref{lem:inductive_statment} for every parameter triple
$(t',n',m')$ with $1\le t'\le n'\le m'$ and $m'<m$.

\subsection{Density Increment Inside a Dictatorship}\label{subsec:6.1}
Our starting point for the density increment argument is the guarantee, established in Step 2, of a non-trivial increase in density within a dictatorship. What we actually require is the following result.

\begin{cor}[Corollary of Proposition~\ref{prop:main step 2} (with $c=1.5$ and $\beta'=4$)]\label{cor:main step 2 corollary}
Fix $\xi \geq 8$. Then there exists a constant $\alpha = \alpha(\xi) > 0$ such that for all $1 \le t \le \alpha n$ and for all $n \le m \le 4n$, the following holds. 

Let $\mathcal{F},\mathcal{G} \subseteq \mathcal{L}(V,W)$ be cross-$(t-1)$-intersection-free families such that
\begin{equation}
    \label{eq:step 2 counterpositive assumption}
    \card{\mathcal{F}}\card{\mathcal{G}} > q^{-mt} q^{2m(n-t)} \ .
\end{equation}
Then at least one of the following must occur:
\begin{itemize}
    \item[(i)] there is a non-zero vector $v \in V$ and some $w \in W$ for which the column-dictatorship restricted subfamily
    $$\mathcal{F}_{v\mapsto w} = \{A \in \mathcal{F} \mid Av = w\}$$
    is large, in the sense that $\card{\mathcal{F}_{v\mapsto w}} \ge q^{-(m-\xi t)} \card{\mathcal{F}}$; or
    \item[(ii)] there is a non-zero vector $w \in W$ and some $v \in V$ such that the row-dictatorship restricted subfamily
    $$\mathcal{F}_{w\uparrow v} =\{A\in\mathcal{F}\mid A^\top w=v \}$$
    is large, meaning that $\card{\mathcal{F}_{w \uparrow v}} \ge q^{-(n-\xi t)} \card{\mathcal{F}}.$
\end{itemize}
\end{cor}

The next lemma states that for sufficiently large families $\mathcal{F},\mathcal{G}$, if one of them admits a column dictatorship restriction exhibiting a non-trivial density bump, then this restriction contains almost all of the family, and this holds simultaneously for both $\mathcal{F}$ and $\mathcal{G}$.
\begin{lem}\label{lem:density_increment}
Let $\mathcal{F,G}\subseteq \mathcal{L}(V,W)$ be non-empty cross-$(t-1)$-intersection-free families for some $t\in \mathbb{N}$ such that $t \le n-1 $ and $n\le m$. Let $\varepsilon>0$. 
If  
\begin{equation}\label{eq:families size to epsilon relation assumption}
 \card{\mathcal{F}}\card{\mathcal{G}}\ge  \frac{2q^{m+2(n-1)}\IB{q}{n-1}{m-1}{t}}{\varepsilon }\ .  
 \end{equation} 
and for some $0\neq v\in V$ and $w\in W$
\begin{equation}\label{eq:density_bump_assumption}
    \card{\mathcal{F}_{v\mapsto w}} \ge \varepsilon\card{\mathcal{F}}.
\end{equation} 
Then
\begin{equation}\label{eq:huge_density_bump}
    \card{\mathcal{F}_{v\mapsto w}}\geq  (1-\varepsilon)\card{\mathcal{F}},
\end{equation}
and 
\begin{equation}\label{eq:density_bump_for_G}
    \card{\mathcal{G}_{v\mapsto w}}\geq (1-\varepsilon)\card{\mathcal{G}}.
\end{equation}

 \end{lem}

We will use the following lemma to prove Lemma \ref{lem:density_increment}, its proof deferred to Appendix~\ref{apx:missing_proofs}. Recall that $\spn{\cdot}$ denotes the $\mathbb{F}_q$-span of a collection of vectors.

\begin{lem} \label{lem:projection_exact_intersection}
For any $A,B \in \mathcal{L}(V,W)$, fix a nonzero $v\in V$ and denote $w=Av$ and $w'=Bv$. Assume $w\ne w'$. Let $V' < V$ be a subspace such that $V=V'\oplus\spn{v}$ and let $\pi: W \to W/\spn{w - w'}$ be the natural projection. Consider the restricted projection
\begin{equation}
    \label{eq:define restricted projection}
     p : \mathcal{L}(V,W) \to \mathcal{L}(V', W/\spn{w - w'})
\end{equation}
defined as $p(M)x = \pi(Mx)$ for all $x \in V'$. Then 
\[ \dim \ker(A-B) = \dim \ker(p(A) - p(B)). \]
\end{lem}

\begin{proof}[Proof of Lemma~\ref{lem:density_increment}]

For every  $w'\in W$ such that $w'\neq w$, let $V'\le V$ be such that $V=V'\oplus \spn{v}$ and  consider the restricted  projection $p$ as defined in \eqref{eq:define restricted projection} applied to
the families $\mathcal{F}_{v\mapsto w},\mathcal{G}_{v\mapsto w'}$ denoted by $\mathcal{F'}_{w}=p(\mathcal{F}_{v\mapsto w})$ and $\mathcal{G'}_{w'}=p(\mathcal{G}_{v\mapsto w'})$. 
By  Lemma~\ref{lem:projection_exact_intersection} the families $\mathcal{F'}_{ w},\mathcal{G'}_{ w'}\subseteq\mathcal{L}(V', W/\spn{w-w'})$ are also cross-$(t-1)$-intersection-free, thus we obtain 
\begin{equation}\label{eq:bounds_on_cross_restricitons}
     \card{\mathcal{F}_{v\mapsto w}}\card{\mathcal{G}_{v\mapsto w'}}\leq q^{2(n-1)}\card{\mathcal{F'}_{ w}}\card{\mathcal{G'}_{ w'}}\leq q^{2(n-1)}\IB{q}{n-1}{m-1}{t}.
\end{equation}
The first inequality holds since each element in the projected family might have at most $q^{n-1}$ pre-images in the dictatorship family, and the last inequality follows directly from Definition~\ref{def:maximal_cross_intersection_free_families}.

 Recall that $\bigsqcup$ denotes the disjoint union operator, and since
$\mathcal{G}=\bigsqcup_{w'\in W}\mathcal{G}_{v\mapsto w'},$
combining with \eqref{eq:bounds_on_cross_restricitons} for every $w'\neq w$, we conclude
\begin{equation}
    \label{eq:LB_g_w_density}
    |\mathcal{G}_{v\mapsto w}|= |\mathcal{G}|-\sum_{w'\neq w}|\mathcal{G}_{v\mapsto w'}|\geq |\mathcal{G}|-q^m\frac{q^{2(n-1)}\IB{q}{n-1}{m-1}{t}}{\card{\mathcal{F}_{v\mapsto w}}}.
\end{equation}

Let $\rho=\card{\mathcal{F}_{v\mapsto w}}/\card{\mathcal{F}}$. We have
$$\sum_{w'\neq w}\card{\mathcal{F}_{v\mapsto w'}}= \card{\mathcal{F}}-\card{\mathcal{F}_{v\mapsto w}}= (1-\rho)\card{\mathcal{F}}.$$
Let $w^{\star}=\mathrm{argmax}_{w'\neq w}\card{\mathcal{F}_{v\mapsto w'}}$. It follows that, 
\begin{equation}\label{eq:lb_f_w_star_density_2}
    \card{\mathcal{F}_{v\mapsto w^\star}}\geq\frac{(1-\rho)\card{\mathcal{F}}}{q^m}.
\end{equation}
Note that $w^\star\neq w$. 
Proceeding analogously to the argument used for \eqref{eq:bounds_on_cross_restricitons}, but now replacing $w,w'$ with $w^\star,w$ and defining the corresponding projection when applying Lemma~\ref{lem:projection_exact_intersection}, we obtain \begin{equation}
\label{eq:restriction_ih_w_star_small_2}
    \card{\mathcal{F}_{v\mapsto w^\star}}\card{\mathcal{G}_{v\mapsto w}}
    \leq q^{2(n-1)}\IB{q}{n-1}{m-1}{t}.
\end{equation}

On the other hand, combining \eqref{eq:LB_g_w_density} and \eqref{eq:lb_f_w_star_density_2}
\begin{align*}
 \card{\mathcal{F}_{v\mapsto w^\star}}\card{\mathcal{G}_{v\mapsto w}}&\geq \frac{(1-\rho)\card{\mathcal{F}}}{q^m}\left(|\mathcal{G}|-q^m\frac{q^{2(n-1)}\IB{q}{n-1}{m-1}{t}}{\card{\mathcal{F}_{v\mapsto w}}}\right) \\
 & \ge (1-\rho)\left(\frac{\card{\mathcal{F}}\card{\mathcal{G}}}{q^m}-\frac{q^{2(n-1)}\IB{q}{n-1}{m-1}{t}}{\varepsilon}\right) && \text{[By \eqref{eq:density_bump_assumption}]}\\
 &\geq (1-\rho)\left(\frac{2q^{2(n-1)}\IB{q}{n-1}{m-1}{t}}{\varepsilon}-\frac{q^{2(n-1)}\IB{q}{n-1}{m-1}{t}}{\varepsilon}\right)&& \text{[By \eqref{eq:families size to epsilon relation assumption}]}\\
 &=\frac{1-\rho}{\varepsilon} q^{2(n-1)}\IB{q}{n-1}{m-1}{t}.
\end{align*}

Using \eqref{eq:restriction_ih_w_star_small_2} we conclude that $\rho\ge 1-\varepsilon$, thereby establishing \eqref{eq:huge_density_bump}.

To see \eqref{eq:density_bump_for_G} we use the bound from \eqref{eq:LB_g_w_density}: 
\[|\mathcal{G}_{v\mapsto w}|\geq |\mathcal{G}|\left(1-\frac{q^{m+2(n-1)}\IB{q}{n-1}{m-1}{t}}{\card{\mathcal{F}_{v\mapsto w}}|\mathcal{G}|}\right)\geq |\mathcal{G}|(1-\varepsilon),\]
where for the last inequality we used $\card{\mathcal{F}_{v\mapsto w}}\ge \frac{1}{2}|\mathcal{F}|$  (by \eqref{eq:huge_density_bump}), then the assumption \eqref{eq:families size to epsilon relation assumption}.
\end{proof}

We now ensure that the density increase guaranteed in Step 2 consistently occurs in the same direction.
\begin{lem}\label{lem:density_bump_on_smaller_side_imply_too_large}
    Let $\mathcal{F,G}\subseteq \mathcal{L}(V,W)$ be two non-empty cross-$(t-1)$-intersection-free families for some $t\le n-1$.
If $n<m$, then all row restrictions are small. 
More precisely, for every $v\in V$ and $w\in W\setminus\{\bar{0}\}$, we have
\begin{equation*}
  \card{\mathcal{F}_{w\uparrow v}} \leq b\card{\mathcal{F}}\ .
\end{equation*}
Where
$$
b=\frac{q^n}{\card{\mathcal{F}}\card{\mathcal{G}}}\max\Bigl\{q^{2m-2}\IB{q}{n-1}{m-1}{t},\ \IB{q}{n}{m-1}{t}\Bigr\}\ .
$$
\end{lem}

\begin{proof}
  Fix $v\in V$ and $w\in W\setminus\{\bar{0}\}$. Assume, for the sake of contradiction, that $\card{\mathcal{F}_{w\uparrow v}} > b\card{\mathcal{F}}.$
 We can express $\mathcal{G}$ as a disjoint union over all possible restrictions of $w$ in the following way:
$$\mathcal{G}=\bigsqcup_{v'\in V}\mathcal{G}_{w\uparrow v'}.$$ 
 Since $\card{V}=q^n$, it follows that there exists some $v^\star \in V$ for which $\card{\mathcal{G}_{w\uparrow v^\star}}\geq q^{-n}\card{\mathcal{G}}.$
 Therefore,
 \begin{equation} \label{eq:lb_small_side_restriction}
\card{\mathcal{F}_{w\uparrow v}}\card{\mathcal{G}_{w\uparrow v^\star}}> b q^{-n}\card{\mathcal{F}}\card{\mathcal{G}}\geq \max\{q^{2m-2}\IB{q}{n-1}{m-1}{t},\IB{q}{n}{m-1}{t}\}.
 \end{equation}
Let $V'=\spn{v-v^\star}^\perp\le V$ so\footnote{Note that when $v^\star = v$, it follows that $V' = V$, and the proof is already general enough to cover this case.} that $\langle x, v-v^\star\rangle:=x^\top (v-v^\star)=0$ for all $x\in V'$
   and let $W'=\{w\}^\perp\le W$. Fix $w''\in W\setminus W'$ so that $W=W'\oplus \spn{w''}$ and every $y\in W$ has a unique decomposition as $y=w'+cw''$ for some $w'\in W'$ and $c\in \F_q$. Let $\pi:W\to W'$ be the projection to $W'$ defined by $\pi(w'+cw'')=w'$. Consider the    restricted projection map $$p:\mathcal{L}(V,W)\to \mathcal{L}(V',W')$$ 
    defined as $p(A)u=\pi(Au)$ for every $u\in V'$. 
Consider the families $\mathcal{F}'=p(\mathcal{F}_{w\uparrow v}),\ \mathcal{G}'=p(\mathcal{G}_{w\uparrow v^\star})\subseteq  \mathcal{L}(V',W')$.
We claim that  $\mathcal{F}',\mathcal{G}'$ are cross-$(t-1)$-intersection-free. To show that, consider any pair $A\in \mathcal{F}_{w\uparrow v},\ B\in \mathcal{G}_{w\uparrow v^\star}$. We claim that
\begin{equation}\label{eq:kernel_preserved_under_restricted_projection}
    \ker(A-B)=\ker(p(A)-p(B))\ ,
\end{equation}
and therefore the dimension of the intersection of $A$ and $B$ is preserved by the map $p$.

Note that for any $x\in V$ we have
\begin{equation}
\label{eq:orthogonality_equality}
\langle(A-B)x,w\rangle=\langle x,(A-B)^\top w\rangle=\langle x,v-v^\star\rangle\ .
\end{equation}
First, if $x\in V'$, then the right-hand side is $0$, so the left-hand side shows that $(A-B)x\perp w$, hence $(A-B)x\in W'$. This implies
\[
\ker(A-B)\cap V'=\ker(p(A)-p(B))\ .
\]
Conversely, if $x\in\ker(A-B)$, then the left-hand side of \eqref{eq:orthogonality_equality} is $0$, so the right-hand side implies $x\in V'$. Thus $\ker(A-B)\subseteq V'$, which establishes \eqref{eq:kernel_preserved_under_restricted_projection}.
We now count the liftings of $\mathcal F', \mathcal G'$ back to the restricted families $\mathcal F_{w\uparrow v}$ and $\mathcal G_{w\uparrow v^\star}$ inside $\LL(V,W)$.
Consider first the case $v\ne v^\star$, in which $\dim(V')=n-1$.  
Fix a complement \(V=V'\oplus \spn{\tilde v}\).
For a fixed \(A'\in \LL(V',W')\), any lift \(A\in \mathcal F_{w\uparrow v}\)
with \(p(A)=A'\) must satisfy, for every \(x\in V'\),
\[
A x = A'x+c(x) w''
\]
for some $c(x)\in \F_q$.   
We next show that $c(x)$ is uniquely determined by the row
condition \(A^\top  w=v\), and does not depend on the exact \(A\in \mathcal F_{w\uparrow v}\). Indeed, this gives us
\[\langle Ax,w\rangle=\langle x,v\rangle\ .\]
Since \(A' x\in W'=\{w\}^\perp\), this implies
\[
\langle A' x+ c(x) w'',w\rangle=\langle x,v\rangle,
\]
and  since  \(w''\notin W'=\{w\}^\perp\), we have
\(\langle w'',w\rangle\neq 0\). Hence,
\(c(x)=\langle  w'',w\rangle^{-1} \langle x,v\rangle\) 
is uniquely
determined by \(w,w''\) and \(v\). 
The only remaining freedom is the
choice of \(A\tilde v\in W\), subject to the single nontrivial
linear constraint
\[
\langle A\tilde v,w\rangle=\langle \tilde v,v\rangle.
\]
Thus each fiber of \(p\) inside \(\mathcal F_{w\uparrow v}\) has size at most
\(q^{m-1}\). The same argument applies to \(\mathcal G_{w\uparrow v^\star}\).
Therefore
\[
|\mathcal F_{w\uparrow v}|\,|\mathcal G_{w\uparrow v^\star}|
\le q^{2m-2}|\mathcal F'|\,|\mathcal G'|\ ,
\]
 which is at most \(q^{2m-2}\IB{q}{n-1}{m-1}{t}\) by Definition \ref{def:maximal_cross_intersection_free_families} 
 
  contradicting \eqref{eq:lb_small_side_restriction}.
If \(v=v^\star\), then \(V'=V\).  For every \(x\in V\), the row
condition uniquely determines the \(w''\)-component of \(Ax\).
Hence \(p\) is injective on both \(\mathcal F_{w\uparrow v}\) and 
\(\mathcal G_{w\uparrow v}\). By Definition \ref{def:maximal_cross_intersection_free_families}, this implies
\[
|\mathcal F_{w\uparrow v}|\,|\mathcal G_{w\uparrow v}|
=|\mathcal F'|\,|\mathcal G'|
\le\IB{q}{n}{m-1}{t},
\]
which again contradicts \eqref{eq:lb_small_side_restriction}.

\end{proof}

We now combine Lemma~\ref{lem:density_increment}, Lemma~\ref{lem:density_bump_on_smaller_side_imply_too_large}, and Corollary~\ref{cor:main step 2 corollary}. In the case $m>n$, Lemma~\ref{lem:density_bump_on_smaller_side_imply_too_large} rules out density bumps coming from row-type dictatorships. When $m=n$, we may instead pass to the dual families, turning any row-type dictatorship into an equivalent column-type dictatorship. This allows us to restrict the technical analysis to column-type dictatorship density bumps only. 
We will use the following lemma to prove Lemma~\ref{lem:inductive_statment} by induction on $m$. Accordingly, in the proof we may assume that Lemma~\ref{lem:inductive_statment} already holds for $m-1$. We state the specific cases assumed in the lemma.

\begin{cor}\label{cor:large free families almost contained in a dictatorship}
    For each fixed $C>2$ and $k_0 \ge 1$, there exists $\alpha=\alpha(k_0 )>0$ such\footnote{Crucially, the value of $\alpha$ is independent of the specific choice of $C$.} that for all $1\le t\le \alpha n$ and $n\le m\le 4n$, if we assume that Lemma~\ref{lem:inductive_statment} holds for the parameters $(n-1,m-1,t)$ and, in addition, when $n<m$ it also holds for $(n,m-1,t)$ with the constant $C$, the following holds.
    
    Let $\mathcal{F,G}\subseteq \mathcal{L}(V,W)$ be two cross-$(t-1)$-intersection-free families.
If 
\begin{equation}\label{eq:families lower bound density bump corollary}
\card{\mathcal{F}}\card{\mathcal{G}} \geq  q^{-k_0 t} q^{2m(n-t)+\max\{2t(Ct-m),0\}+\max\{2t(2m-4(n-t)),0\}},
\end{equation}
  then if $n<m$ there exist $v\in V\setminus \{\bar{0}\}$ and $w\in W$ such that 
 
\begin{equation}\label{eq:density bump in both families}
    \card{\mathcal{F}_{v\mapsto w}}\geq  (1-\varepsilon)\card{\mathcal{F}}\ ,\ 
\text{ and }\ 
    \card{\mathcal{G}_{v\mapsto w}}\geq (1-\varepsilon)\card{\mathcal{G}}.
\end{equation}
Where $\varepsilon=q^{-m+\xi t}$ with  $\xi= k_0 +9$.

In the case $n = m$ (so that $V\approxeq W$), the bounds in \eqref{eq:density bump in both families} apply either directly to the families $\mathcal{F}$ and $\mathcal{G}$, or instead to their dual families $\mathcal{F}^* = \{A^\top : A \in \mathcal{F}\}$ and $\mathcal{G}^* = \{A^\top : A \in \mathcal{G}\}$.
\end{cor}
\begin{proof}
  Let $\alpha$ be a constant smaller than $\alpha$ from Corollary~\ref{cor:main step 2 corollary}, applied with $\xi=k_0 +9$. In addition, pick $\alpha$ sufficiently small so that $\alpha\le 1/k_0 $. With this choice, it follows that $m\ge n\ge t/\alpha\ge k_0$.
   Hence, the assumption
   \eqref{eq:families lower bound density bump corollary} yields 
   $$ \card{\mathcal{F}}\card{\mathcal{G}} > q^{-mt} q^{2m(n-t)} \ .$$ 
   Applying Corollary~\ref{cor:main step 2 corollary}, we conclude that at least one of the following two cases must occur: 
   \begin{itemize}
    \item[(i)] there exist $v\in V\setminus\{\bar{0}\}$ and $w\in W$ such that $\card{\mathcal{F}_{v\mapsto w}}\ge q^{-(m-\xi t)}\card{\mathcal{F}}$, \\
    or
    \item[(ii)] there exist  $w\in W\setminus\{\bar{0}\}$ and $v\in V$ such that $\card{\mathcal{F}_{w \uparrow v}}\ge q^{-(n-\xi t)}\card{\mathcal{F}}.$
\end{itemize}
If $m>n$, Lemma~\ref{lem:density_bump_on_smaller_side_imply_too_large} implies  $\card{\mathcal{F}_{w \uparrow v}}< b\card{\mathcal{F}}$ with 
$$b=\frac{q^n}{\card{\mathcal{F}}\card{\mathcal{G}} }\max\{q^{2m-2}\IB{q}{n-1}{m-1}{t},\IB{q}{n}{m-1}{t}\}\ .$$  
Since Lemma~\ref{lem:inductive_statment} holds for the parameters $(n-1,m-1,t)$ and $(n,m-1,t)$ we have
\begin{multline}\label{eq:IH on n and m}
 \IB{q}{n-1}{m-1}{t}\le q^{2(m-1)(n-1-t)+\max\{2t(Ct-(m-1)),0\}+\max\{2t(2(m-1)-4(n-1-t)),0\}}\\
 = q^{-2m+2-2n+2t+2m(n-t)+\max\{2t(Ct-m)+2t,0\}+\max\{2t(2m-4(n-t))+4t,0\}}\ ,
 \end{multline} 
and
\begin{multline*}
\IB{q}{n}{m-1}{t}\le q^{2(m-1)(n-t)+\max\{2t(Ct-(m-1)),0\}+\max\{2t(2(m-1)-4(n-t)),0\}} \\
= q^{-2n+2t+2m(n-t)+\max\{2t(Ct-m)+2t,0\}+\max\{2t(2m-4(n-t))-4t,0\}}\ .
\end{multline*}
Combining this with \eqref{eq:families lower bound density bump corollary} we have,  
$$b \le \frac{q^n q^{-2n+2t+2m(n-t)+\max\{2t(Ct-m)+2t,0\}+\max\{2t(2m-4(n-t))+4t,0\}}}{ q^{-k_0 t} q^{2m(n-t)+\max\{2t(Ct-m),0\}+\max\{2t(2m-4(n-t)),0\}}}\le q^{-n+(k_0 +8)t}< q^{-n+\xi t}\ .$$ 
Contradicting Item (ii), it follows that Item (i) must hold. 
In the case $n=m$ and we are in Item (ii), we instead consider the dual families $\mathcal{F}^*, \mathcal{G}^*$, which then fall under Item (i).

Next, we assume that Item (i) holds, that is, that $\card{\mathcal{F}_{v\mapsto w}}\ge \varepsilon\card{\mathcal{F}}$. Combining \eqref{eq:families lower bound density bump corollary} with $k_0 + 9 \le \xi$ guarantees that the condition \eqref{eq:families size to epsilon relation assumption} holds. Therefore, we may apply Lemma~\ref{lem:density_increment} to obtain \eqref{eq:density bump in both families}, completing the proof. Indeed,
\begin{multline*}
   \frac{2q^{m+2(n-1)}\IB{q}{n-1}{m-1}{t}}{\card{\mathcal{F}}\card{\mathcal{G}}}
    \le \frac{2 q^{-m+2t+2m(n-t)+\max\{2t(Ct-m)+2t,0\}+\max\{2t(2m-4(n-t))+4t,0\}}}{q^{-k_0 t} q^{2m(n-t)+\max\{2t(Ct-m),0\}+\max\{2t(2m-4(n-t)),0\}}}\\
    \le q^{-m+(k_0 +9)t}\le \varepsilon\ .\qedhere
\end{multline*}
\end{proof}
\subsection{Finding the Common Umvirate}\label{sec:6.2}
The following lemma shows that if large families are cross-$(t-1)$-intersection-free, then they are almost contained in a (dual) $t$-umvirate family.

\begin{lem}
\label{lem:iterative_argument}
For each fixed $C>2$ and $k \ge 1$, there exists $\alpha=\alpha(k )>0$ (independent of $C$) such that the following holds. 

Let $t,n,m\in\N$ such that $1\le t\le \alpha n$ and $n \leq m\le 4(n-t)$, and let $\mathcal{F}, \mathcal{G} \subseteq \mathcal{L}(V,W)$ be two cross-$(t-1)$-intersection-free families. Assume that Lemma~\ref{lem:inductive_statment} holds for all parameter triples $(n',m',t')$ satisfying $t'\le n'\le m'$, $1\le t'\le t$, and $m'<m$, with the constant $C$.
    If
    \begin{equation}\label{eq:size_lower_bound_1}
        |\mathcal{F}||\mathcal{G}| \geq q^{-k } q^{2m(n-t)+\max\{2t(Ct-m),0\}+\max\{2t(2m-4(n-t)),0\}}\ .
    \end{equation}
    Then, one of the following must hold: 
    \begin{itemize}
        \item [Case (i):] 
        There exists a $t$-umvirate family $\mathcal{U}\subseteq \LL(V,W)$ such that 
        $$\card{\mathcal{F}\cap \mathcal{U}}\geq (1-\gamma)\card{\mathcal{F}}\text{ and } \card{\mathcal{G}\cap \mathcal{U}}\geq (1-\gamma)\card{\mathcal{G}}.$$
        Or,
           \item [Case (ii):] 
        $m=n$ and there exists a dual $t$-umvirate family $\mathcal{U}^*\subseteq \LL(V,W)$ such that 
        $$\card{\mathcal{F}\cap \mathcal{U}^*}\geq (1-\gamma)\card{\mathcal{F}}\text{ and } \card{\mathcal{G}\cap \mathcal{U}^*}\geq (1-\gamma)\card{\mathcal{G}}.$$
    \end{itemize}
     For $\gamma= tq^{\xi t-m}$ with $\xi=k +10.$ 
   
\end{lem}
\noindent Assuming Lemma~\ref{lem:inductive_statment} and Lemma~\ref{lem:iterative_argument} hold, we may deduce Theorem~\ref{thm:main} as follows.
\begin{proof}[Proof of Theorem~\ref{thm:main}]
Assume Lemma~\ref{lem:inductive_statment} holds. Apply Lemma~\ref{lem:iterative_argument} with $k=200$ proves Theorem~\ref{thm:main} for $\alpha\le 1/300$ which is the minimum between $\alpha$ from Lemma~\ref{lem:iterative_argument} and $1/300$, $\beta=2$, and $c_0 = 1/2$.
\end{proof}

\begin{proof}[Proof of Lemma \ref{lem:iterative_argument}]
    To prove the lemma, we will iteratively find a sequence of $t$ dictatorship restrictions $v_1\mapsto w_1, \dots v_t\mapsto w_t$ for linearly independent $v_1,\dots v_t\in V$ and some $w_1, \dots w_t\in W$ and define 
    $$\mathcal{U}=\{A\in\LL(V,W): Av_i=w_i \text{ for all } i\in \{1,\dots t\}\}\ .$$
    We will proceed in $t$ steps iteratively. At each step $\ell \in \{1,\dots, t\}$, using Corollary~\ref{cor:large free families almost contained in a dictatorship} we will find a new restriction $v_\ell\mapsto w_\ell$ that has a large density bump into the families $\mathcal{F}$ and $\mathcal{G}$, and we then define the restricted families as
    \[
        \mathcal{F}^{(\ell)} := \{ A \in \mathcal{F} : A v_i = w_i \quad \forall i=1,\dots,\ell \}\ ,
    \]
    and
    \[
        \mathcal{G}^{(\ell)} := \{ B \in \mathcal{G} : B v_i = w_i \quad \forall i=1,\dots,\ell \}\ .
    \]
    We will show that those families are large (w.r.t. $\mathcal{F}$ and $\mathcal{G}$) and at the end we will have $\mathcal{F}^{(t)}=\mathcal{F}\cap \mathcal{U}$ and $\mathcal{G}^{(t)}=\mathcal{G}\cap \mathcal{U}$.
    
       If $m=n$, Corollary~\ref{cor:large free families almost contained in a dictatorship} might provide these restrictions for the dual families. 
       So we proceed similarly, the only difference is that we now work with the dual families $\mathcal{F}^*=\{A\in \LL(W,V):A^\top\in \mathcal{F}\}$ and $\mathcal{G}^*=\{A\in \LL(W,V):A^\top\in \mathcal{G}\}$ instead. The distinction will occur at $
      \ell=1$. 

    We begin with the base case $\ell=1$.  
Assume that \eqref{eq:size_lower_bound_1} holds.  
By applying Corollary~\ref{cor:large free families almost contained in a dictatorship} 
We have that there exists nonzero $v_1\in V$ and $w_1\in W$ such that
    \[
        \card{\mathcal{F}_{v_1\mapsto w_1}}\geq (1-q^{(k +9) t-m})\card{\mathcal{F}}\ \text{ and }\ 
        \card{\mathcal{G}_{v_1\mapsto w_1}}\geq (1-q^{(k +9) t-m})\card{\mathcal{G}}
        \ .
    \]
    Or, if $n=m$ and those bounds hold for the dual families, we will proceed with  $\mathcal{F}^*,\mathcal{G}^*$ instead  and renaming them to $\mathcal{F},\mathcal{G}$.
 
Set $\varepsilon = q^{\xi t - m}$ (where, recall, $\xi = k  + 10$). Then we obtain
$$\card{\mathcal{F}^{(1)}} \ge (1-\varepsilon)\card{\mathcal{F}} \quad\text{and}\quad \card{\mathcal{G}^{(1)}} \ge (1-\varepsilon)\card{\mathcal{G}}\ .$$
If $t=1$, the proof is finished with $\gamma=\varepsilon$. Otherwise, we proceed iteratively with $\ell\in \{2,\dots , t\}$. Assume we found appropriate restrictions up to $\ell-1$, meaning that
\begin{equation}\label{eq:iterative_argument_assumption}
 \card{\mathcal{F}^{(\ell-1)}}\ge (1-\varepsilon)^{\ell-1}\card{\mathcal{F}}\text{ and } \card{\mathcal{G}^{(\ell-1)}}\ge (1-\varepsilon)^{\ell-1}\card{\mathcal{G}}.   
\end{equation}
 Let $V'\le V$ be such that $V=V'\oplus \spn{v_1,\dots, v_{\ell-1}}$, and denote  $n':=\dim(V')$ so that $n'=n-(\ell-1)$. For a map $A\in \LL(V,W)$, denote $A|_{V'}\in \LL(V',W)$ its restriction to $V'$ defined by $(A|_{V'})x=Ax$ for all $x\in V'$.
Consider the restriction of the families from \eqref{eq:iterative_argument_assumption} to $\LL(V',W)$, defined by 
$$\widetilde{\mathcal{F}}^{(\ell-1)}:=\big\{A|_{V'}:A\in\mathcal{F}^{(\ell-1)}\big\} \text{\quad 
and \quad} \widetilde{\mathcal{G}}^{(\ell-1)}:=\big\{B|_{V'}:B\in\mathcal{G}^{(\ell-1)}\big\}.$$
Note that the intersection size of any pair $\tilde{A}\in \widetilde{\mathcal{F}}^{(\ell-1)},\tilde{B}\in \widetilde{\mathcal{G}}^{(\ell-1)}$ is $\ell-1$ less than their preimages in $\mathcal{F}^{(\ell-1)}$ and $\mathcal{G}^{(\ell-1)}$, thus, for $t'=t-(\ell-1)$ the families $\widetilde{\mathcal{F}}^{(\ell-1)}, \widetilde{\mathcal{G}}^{(\ell-1)}$ are cross-$(t'-1)$-intersection-free. Furthermore, 
\begin{multline*}
    \card{\widetilde{\mathcal{F}}^{(\ell-1)}}\card{\widetilde{\mathcal{G}}^{(\ell-1)}}=\card{\mathcal{F}^{(\ell-1)}}\card{\mathcal{G}^{(\ell-1)}}\ge (1-\varepsilon)^{2(\ell-1)}\card{\mathcal{F}}\card{\mathcal{G}}\\
    \ge (1-2t\varepsilon)q^{-k }\cdot  q^{2m(n-t)+\max\{2t(Ct-m),0\}+\max\{2t(2m-4(n-t)),0\}}\\ 
    \ge q^{-1}q^{-k t'} q^{2m(n'-t')+\max\{2t'(Ct'-m),0\}+\max\{2t'(2m-4(n'-t')),0\}}
\ .
\end{multline*}
The last inequality follows from $n'-t'=n-(\ell-1)-(t-(\ell-1))=n-t$ and $1\le t'<t$. Moreover, for sufficiently small $\alpha\le \frac{1}{\xi+3}$ we have $m\ge n\ge \frac{t}{\alpha}\ge (\xi+3)t$, and it follows that $1-2t\varepsilon\ge \frac{1}{2}\ge q^{-1}$.
Moreover, we have $m \le 4(n-t) = 4(n'-t') \le 4n'$, and
\[
t' = t-(\ell-1) \le \alpha n - (\ell-1) = \alpha n' + (\alpha-1)(\ell-1) \le \alpha n',
\]
so we can apply Corollary~\ref{cor:large free families almost contained in a dictatorship} with $k_0 = k+1$ to the families $\widetilde{\mathcal{F}}^{(\ell-1)}, \widetilde{\mathcal{G}}^{(\ell-1)} \subseteq \LL(V',W)$, using $t', n', m$ as the corresponding parameters (note that $n' < m$). This ensures the existence of a density bump inside a dictatorship.
That is, there exists a non-zero $v_\ell\in V'$ and some $w_\ell\in W$ such that
$$\card{\widetilde{\mathcal{F}}^{(\ell-1)}\big|_{v_\ell\mapsto w_\ell}}\ge (1-\varepsilon_\ell) \card{\widetilde{\mathcal{F}}^{(\ell-1)}}\ ,\text{\quad 
and \quad}
\card{\widetilde{\mathcal{G}}^{(\ell-1)}\big|_{v_\ell\mapsto w_\ell}}\ge (1-\varepsilon_\ell) \card{\widetilde{\mathcal{G}}^{(\ell-1)}}\ .$$
Here $\varepsilon_\ell = q^{\xi t' - m}$, where, as a reminder, $\xi = k  + 10 = k_0  + 9$.
Observe that, since $V'$ was chosen as a complement, $v_\ell$ is linearly independent of $\{v_1,\dots, v_{\ell-1}\}$. In addition, $\mathcal{F}^{(\ell)}$ is isomorphic to $\widetilde{\mathcal{F}}^{(\ell-1)}|_{v_\ell\mapsto w_\ell}$ and $\mathcal{G}^{(\ell)}$ is isomorphic to $\widetilde{\mathcal{G}}^{(\ell-1)}|_{v_\ell\mapsto w_\ell}$. We use $\varepsilon_\ell\le \varepsilon$ and combine with the bounds of \eqref{eq:iterative_argument_assumption} on $\card{\widetilde{\mathcal{F}}^{(\ell-1)}}=\card{\mathcal{F}^{(\ell-1)}}$ and $\card{\widetilde{\mathcal{G}}^{(\ell-1)}}=\card{\mathcal{G}^{(\ell-1)}}$ to conclude that 
$$\card{\mathcal{F}^{(\ell)}}\ge (1-\varepsilon)^\ell\card{\mathcal{F}}\text{ and } \card{\mathcal{G}^{(\ell)}}\ge (1-\varepsilon)^\ell\card{\mathcal{G}}.$$
In particular, 
$$\card{\mathcal{F}^{(t)}}\ge (1-\varepsilon)^t\card{\mathcal{F}}\ge (1-t\varepsilon)\card{\mathcal{F}}\text{ and } \card{\mathcal{G}^{(t)}}\ge (1-t\varepsilon)\card{\mathcal{G}}\ . \qedhere$$
\end{proof}
We need the following lemma to bound the portion of the families outside the $t$-umvirate provided by  Lemma~\ref{lem:iterative_argument}.
\begin{lem}\label{lem:families_almost_contained_in_U_small_intersection_other_umvirates}
   Under the assumptions of Lemma~\ref{lem:iterative_argument}, assume that we are in Case (i) with a $t$-umvirate family $\mathcal{U}$. Let $V'\le V$ be the $t$-dimensional subspace on which $\mathcal{U}$ is determined. Let $\mathcal{U}'\neq \mathcal{U}$ be another $t$-umvirate family determined by $V'$, and let $0\le j\le t-1$ be maximal such that there exists a $j$-umvirate family containing both $\mathcal{U}$ and $\mathcal{U}'$. Set $r=t-j$, so that $1\le r\le t$. Then there exists a constant $\kappa=\kappa(k)>1$, where $k$ is as in Lemma~\ref{lem:iterative_argument}, such that
\begin{equation}
    \label{eq:small intersection with other t umvirate families}
    \card{\mathcal{F}\cap\mathcal{U}'}\le q^{-2r(m-\kappa r)}q^{m(n-t)}\quad \text{and} \quad\card{\mathcal{G}\cap\mathcal{U}'}\le q^{-2r(m-\kappa r)}q^{m(n-t)}\ .
\end{equation}
The same bounds apply in Case (ii) of Lemma~\ref{lem:iterative_argument}, after passing to the dual families.
\end{lem}
\begin{proof}
 We will prove that $\card{\mathcal{F}\cap\mathcal{U}'}\le q^{-2r(m-\kappa r)}q^{m(n-t)}$. By symmetry, the same bound holds for $\card{\mathcal{G}\cap\mathcal{U}'}.$
For Case (ii), an analogous argument applies to the dual families.
 We first observe that $\card{\mathcal{G}\cap \mathcal{U}}$ must be large. Indeed, 
    $$\card{\mathcal{U}}\card{\mathcal{G}\cap\mathcal{U}}\ge \card{\mathcal{F}\cap\mathcal{U}}\card{\mathcal{G}\cap\mathcal{U}}\ge (1-\gamma)^2\card{\mathcal{F}}\card{\mathcal{G}}\ .$$
    Using $(1-\gamma)^2\ge q^{-1}$ for sufficiently small $\alpha$, the bound on $\card{\mathcal{F}}\card{\mathcal{G}}$ from \eqref{eq:size_lower_bound_1} and the fact that $\card{\mathcal{U}}=q^{m(n-t)}$, this implies
    \begin{equation}\label{eq:G_in_U_is_large}
    \card{\mathcal{G}\cap \mathcal{U}}
    \ge q^{-(k +1)+m(n-t)+\max\{2t(Ct-m),0\}+\max\{2t(2m-4(n-t)),0\}}\ .    
    \end{equation}
    Let $V' = \spn{v_1,\dots,v_t}$, and recall that $\mathcal{U}$ is the $t$-umvirate family determined by $v_i \mapsto w_i$ for all $1 \le i \le t$. Let $w_i' \in W$ be vectors such that $\mathcal{U}'$ is the $t$-umvirate defined by $v_i \mapsto w_i'$ for all $1 \le i \le t$. Define $W' := \spn{u_1,\dots,u_t}$ where $u_i := w_i - w_i'$, and observe that $\dim(W') = r$. 
    Fix a subspace $V'' \le V$ that complements $V'$, so that $V = V' \oplus V''$. Likewise, choose $W'' \le W$ with $W = W' \oplus W''$, and observe that $V'' \cong \mathbb{F}^{n-t}$ and $W'' \cong \mathbb{F}^{m-r}$. Now define the families $\mathcal{F}', \mathcal{G}' \subseteq \LL(V'', W'')$ by restricting the domain of the maps to $V''$ and applying the natural projections $W \to W''$ to the elements of $\mathcal{F} \cap \mathcal{U}'$ and $\mathcal{G} \cap \mathcal{U}$, respectively.   
    Since $\mathcal{F}\cap\mathcal{U'}$ and $\mathcal{G}\cap \mathcal{U}$ are cross-$(t-1)$-intersection-free, and for any $A\in \mathcal{U}', B\in\mathcal{U}$ the dimension of their intersection on $V'$ is exactly $t-r$, we can apply an argument analogous to the proof of Lemma~\ref{lem:density_bump_on_smaller_side_imply_too_large} to get that the families $\mathcal{F}',\mathcal{G}'$ are cross-$(r-1)$-intersection-free. 
    By Definition \ref{def:maximal_cross_intersection_free_families},      we have $$\card{\mathcal{F}'}\card{\mathcal{G}'}\le \IB{q}{n-t}{m-r}{r}     \ .$$
    Lifting those families to $\LL(V'',W)$ we have at most $\card{W'}^{n-t}=q^{r(n-t)}$ original maps that project to each element. Therefore, 
    $$\card{\mathcal{F}\cap{\mathcal{U}'}}\le q^{r(n-t)}\card{\mathcal{F}'}\text{ and } \card{\mathcal{G}\cap{\mathcal{U}}}\le q^{r(n-t)}\card{\mathcal{G}'}.$$
                Using the assumption that  Lemma~\ref{lem:inductive_statment} holds for the parameters $(n-t,m-r,r)$, we conclude that 
    \begin{multline*}
        \card{\mathcal{F}\cap\mathcal{U}'}\le q^{r(n-t)}\card{\mathcal{F}'}\le q^{r(n-t)}\frac{\IB{q}{n-t}{m-r}{r}}{\card{\mathcal{G}'}}\le q^{2r(n-t)}\frac{\IB{q}{n-t}{m-r}{r}}{\card{\mathcal{G}\cap\mathcal{U}}} \\
        \le q^{2r(n-t)}\frac{q^{2(m-r)(n-t-r)+\max\{2r(Cr-(m-r)),0\}+\max\{2r(2(m-r)-4(n-t-r)),0\}}}{q^{-(k +1)+m(n-t)+\max\{2t(Ct-m),0\}+\max\{2t(2m-4(n-t)),0\}}}\\
               \le q^{2r(n-t)+2(m-r)(n-t-r)+(k +1)-m(n-t)+6r^2} = 
        q^{m(n-t)-2r(m-4r)+k +1}\ .
    \end{multline*}
    For the last inequality, we observe that 
    \begin{multline*}
        \max\{2r(Cr-(m-r)),0\}-\max\{2t(Ct-m),0\}\le \\2r^2+\max\{2r(Cr-m),0\}-\max\{2t(Ct-m),0\}\le 2r^2\ ,
    \end{multline*}
    \begin{multline*}
        \max\{2r(2(m-r)-4(n-t-r)),0\}-\max\{2t(2m-4(n-t)),0\}\le \\4r^2+\max\{2r(2m-4(n-t)),0\}-\max\{2t(2m-4(n-t)),0\}\le 4r^2\ .
    \end{multline*}
Choosing $\kappa=k +5$ completes the proof.
\end{proof}

\subsection{Proof of Main Bound}\label{subsec:proof of inductive lemma}

This section is devoted to proving the main inductive lemma; its proof relies mainly on Lemmas~\ref{lem:iterative_argument} and~\ref{lem:families_almost_contained_in_U_small_intersection_other_umvirates}.
We will need the following general lemmas which state that any map not contained in a $t$-umvirate family $\mathcal{U}$ must possess a nontrivial subset of $\mathcal{U}$ on which its intersection dimension is exactly $t-1$ and that there are not too many $t$-umvirate families with certain properties. Their standalone proofs are provided later in Appendix~\ref{apx:missing_proofs}.
\begin{lem}\label{lem:counting_t-1_free_with_t_umvirate}
For $n\le m$ and $1\le t\le \alpha n$, let  $\mathcal{U}\subseteq\LL(V,W)$ be a $t$-umvirate family defined by $Av_i=w_i$ for all $A\in\mathcal{U}$ and $i\in\{1,\dots,t\}$. 
Let $B\in \LL(V,W)\setminus \mathcal{U}$ and define
$$\mathcal{U}'=\{A\in \mathcal{U}:\dim\ker(A-B)\neq t-1\}$$
as the set of all maps in $\mathcal{U}$ that satisfy the forbidden intersection of dimension $t-1$ with $B$.
Define $u_i=w_i-Bv_i$, and let $r=\dim\spn{u_1,\dots,u_t}$. 
Then 
$$\card{\mathcal{U}'}
\le (1-\frac{1}{4}\cdot q^{-r(t+(m-n))})\card{\mathcal{U}}.$$ 
\end{lem}
\begin{lem}\label{lem:number of different t umvirate families with t-j intersection with U}
    Let $\mathcal{U}$ be a $t$-umvirate in $\LL(V,W)$ that is fixed on a $t$-dimensional subspace $V'\le V$. 
    For $1\le r\le t$, let $\mathfrak{U}_r$ be the set of all
        $t$-umvirate families $\mathcal{U}'$, also fixed on $V'$, that agree with $\mathcal{U}$ on exactly a $(t-r)$-dimensional subspace of $V'$. Then     $$\card{\mathfrak{U}_r}\le  4q^{r(m+t-r)}\ .$$
\end{lem}

We conclude this section by combining all the above lemmas into a proof for the inductive statement.

\begin{proof}[Proof of Lemma~\ref{lem:inductive_statment}]
   As mentioned before, we prove the lemma by a strong induction on
the codomain dimension $m$, proving the statement simultaneously for all
triples $(t,n,m)$ with $1\le t\le n\le m$. 

Let $\mathcal{F},\mathcal{G}\subseteq \LL(V,W)$ be cross-$(t-1)$-intersection-free families with $1\le t\le n\le m$. Let $\alpha'$ be a sufficiently small constant to be determined later. We will prove the lemma for any constant $C\ge 8/\alpha'$. Note that, in particular, $C>2$ as required by Lemma~\ref{lem:iterative_argument}. 
   Recall that our goal is to show that 
   \begin{equation}
       \label{eq:goal of induction}|\mathcal{F}||\mathcal{G}|\le q^{2m(n-t)+\max\{2t(Ct-m),0\}+\max\{2t(2m-4(n-t)),0\}}\ .
   \end{equation}
Equation~\eqref{eq:goal of induction} holds trivially whenever $m\le \frac{Ct}{2}$ because $\IB{q}{n}{m}{t}\le |\LL(V,W)|^2=q^{2nm}$ and in this case $\max\{2t(Ct-m),0\}\ge 2mt$. It also 
holds trivially if $m\ge 4(n-t)$ since then $\max\{2t(2m-4(n-t)),0\}\ge 2mt$. These two cases cover the base case and we will proceed by induction narrowing the range of the parameters. For the induction step, we may assume for the rest of the proof that $\frac{Ct}{2}<m<4(n-t)$ and it follows that, in this range, $n>\frac{Ct}{8}\ge \frac{t}{\alpha'}$. 
If 
$$
 |\mathcal{F}||\mathcal{G}| < q^{-1} q^{2m(n-t)+\max\{2t(Ct-m),0\}+\max\{2t(2m-4(n-t)),0\}}\ ,
$$
then we are done. Otherwise, we can apply Lemma~\ref{lem:iterative_argument} with $k=1$. Let
$\alpha=\alpha(1)$ be the resulting constant from Lemma~\ref{lem:iterative_argument} and pick  $\alpha'\le \alpha$. Recall that $\alpha$ is independent of $C$, therefore
this choice is non-circular. Without loss of generality assume that we are in Case (i) of Lemma~\ref{lem:iterative_argument}, otherwise move to the dual families and proceed. Let $\mathcal{U}$ be the $t$-umvirate family that almost contains both $\mathcal{F}$ and $\mathcal{G}$. If $\mathcal{F},\mathcal{G}$ are contained in $\mathcal{U}$ then we are done, so without loss of generality assume $\mathcal{F}\not\subseteq\mathcal{U}$. 
let $V'\le V$ be the $t$-dimensional subspace on which $\mathcal{U}$ is fixed. Define
$$r_0:=\min\{\rank((A-B)|_{V'}): A\in \mathcal{F}\setminus\mathcal{U},B\in\mathcal{{U}}\}\ .$$
Observe that $1\le r_0\le t$. We now use the definition of $\mathfrak{U}_r$ from Lemma~\ref{lem:number of different t umvirate families with t-j intersection with U} and the corresponding bound obtained. Let $\kappa$ be the constant from Lemma~\ref{lem:families_almost_contained_in_U_small_intersection_other_umvirates}, which shows that for every $t$-umvirate $\mathcal{U}'\in \mathfrak{U}_r$ we have $\card{\mathcal{F}\cap\mathcal{U}'}\le q^{m(n-t)-2r(m-\kappa r)}$. Then,
    \begin{multline*}
    \card{\mathcal{F}\setminus\mathcal{U}}\le \sum_{r=r_0}^t\sum_{\mathcal{U}'\in\mathfrak{U}_r} \card{\mathcal{F}\cap\mathcal{U}'}\le \sum_{r=r_0}^t\sum_{\mathcal{U}'\in\mathfrak{U}_r}q^{-2r(m-\kappa r)}\card{\mathcal{U}}\le \sum_{r=r_0}^t\card{\mathfrak{U}_r}q^{-2r(m-\kappa r)}\card{\mathcal{U}} \\
        \le\sum_{r=r_0}^t 4q^{r(m+t-r)}q^{-2r(m-\kappa r)}\card{\mathcal{U}}\le \sum_{r=r_0}^t q^{-r(m-(2\kappa+2)t)}\card{\mathcal{U}}\le 2 q^{-r_0(m-(2\kappa+2)t)}\card{\mathcal{U}}\ .
    \end{multline*}
The second to last inequality is by $r\le t$, and the last inequality uses $m>(2\kappa+2)t$. This condition can be guaranteed by choosing $\alpha'$ sufficiently small.
Let $\kappa' =2\kappa+3$ we can bound further
\begin{equation}
    \label{eq:bound on total size outside the umvirate U}
    \card{\mathcal{F}\setminus\mathcal{U}}\le q^{-r_0(m-\kappa' t)}\card{\mathcal{U}}\ .
\end{equation}
 Similarly, if $\mathcal{G}\not\subseteq\mathcal{U}$, we have the following bound: 
 $$\card{\mathcal{G}\setminus\mathcal{U}}\le  q^{-r_0'(m-\kappa' t)}\card{\mathcal{U}}\ .$$
 Where $$r_0':=\min\{\rank((A-B)|_{V'}): A\in \mathcal{G}\setminus\mathcal{U},B\in\mathcal{{U}}\}\ .$$
We next show that 
$$\card{\mathcal{F}}\card{\mathcal{G}}\le (1-2^{-3}\cdot q^{-t(t+(m-n))})\card{\mathcal{U}}^2\ .$$ 
Let  
$$r':=\min\{\rank((A-B)|_{V'}): A\in (\mathcal{F}\cup\mathcal{G})\setminus\mathcal{U},B\in\mathcal{{U}}\}\ .$$
    Without loss of generality (since $\mathcal{F}$ and $\mathcal{G}$ play symmetric roles), assume that the minimal value $r'$ is achieved by some element $A \in \mathcal{F}$, and thus $r' = r_0.$ By  \eqref{eq:bound on total size outside the umvirate U} we have $$\card{\mathcal{F}\setminus \mathcal{U}}\le q^{-r'(m-\kappa' t)}\card{\mathcal{U}}\ .$$ 
    Now, if $\mathcal{G}\subseteq \mathcal{U}$ then $\card{\mathcal{G}\setminus \mathcal{U}}=0$.
    Otherwise, 
    $\card{\mathcal{G}\setminus \mathcal{U}}\le q^{-r_0'(m-\kappa' t)}\card{\mathcal{U}}$ with $r_0'\ge r'$. In both cases, we can bound
    $$\card{\mathcal{G}\setminus \mathcal{U}}\le q^{-r'(m-\kappa' t)}\card{\mathcal{U}}\ .$$
   By Lemma \ref{lem:counting_t-1_free_with_t_umvirate}, taking $A\in \mathcal{F}\setminus\mathcal{U}$ that realizes the minimal rank above and recalling that $\rank((A-B)|_{V'})=r'$ for every $B\in\mathcal{U}$, it follows that
    $$\card{\mathcal{G}\cap \mathcal{U}}\le \left(1-\frac{1}{4}q^{-r'(t+m-n)}\right)\card{\mathcal{U}}\ .$$
    We will use the trivial bound $\card{\mathcal{F\cap U}}\le\card{\mathcal{U}}$ and conclude the following: 

    \begin{align*}
        \card{\mathcal{F}}\card{\mathcal{G}}&=(\card{\mathcal{F\cap U}}+\card{\mathcal{F}\setminus\mathcal{U}})(\card{\mathcal{G\cap U}}+\card{\mathcal{G}\setminus\mathcal{U}})\\
        &\le \left(1+q^{-r'(m-\kappa' t)}\right)\left(1-\frac{1}{4} q^{-r'(t+(m-n))}+q^{-r'(m-\kappa' t)}\right)\card{\mathcal{U}}^2\\
        &\le\left(1-\frac{1}{4} q^{-r'(t+(m-n))}+2 q^{-r'(m-\kappa' t)}+ q^{-2r'(m-\kappa' t)}\right)\card{\mathcal{U}}^2\\
        &\le \left(1-\frac{1}{8} q^{-r'(t+(m-n))}\right)\card{\mathcal{U}}^2\\
        &\le \left(1-\frac{1}{8} q^{-t(t+(m-n))}\right)\card{\mathcal{U}}^2\ .
    \end{align*}
    The second to last inequality is by $2 q^{-r'(m-\kappa' t)}+ q^{-2r'(m-\kappa' t)}\le 3 q^{-r'(m-\kappa' t)}\le \frac{1}{8} q^{-r'(t+(m-n))}$ for all $n\ge (\kappa' +6)t$, which holds by taking $\alpha'$ to be small enough, and the last is by $r'\le t$.
\end{proof}

\subsection{Forbidden intersection in $\SL(n,q)$ and $\GL(n,q)$}
The goal of this section is to prove Theorem~\ref{thm:main2}.
We actually present here the proof for the analogous theorem in $\SL(n,q)$, which needs slightly stronger bounds, and deduce Theorem~\ref{thm:main2} by a simple argument (or via the same proof \emph{mutatis mutandis}).

We shall use Lemma~\ref{lem:iterative_argument} and the two lemmas below to conclude it is almost contained in a $t$-umvirate.

\begin{lem}[The \(\SL(n,q)\) version of Lemma~\ref{lem:counting_t-1_free_with_t_umvirate}]\label{lem:counting_t-1_free_with_t_umvirate-2}

There is an absolute constant \(c_0>0\) such that the following holds.
Assume \(n\ge 3t\). Let  $U\subseteq \SL(n,q)$ be a non-empty $t$-umvirate family defined by $Av_i=w_i$ for all $A\in U$ and $i\in\{1,\dots,t\}$. 
Let $B\in \SL(n,q)\setminus U$ and define
$$U'=\{A\in U:\dim\ker(A-B)\neq t-1\}$$
as the set of all maps in $U$ that satisfy the forbidden intersection of dimension $t-1$ with $B$.
Define $u_i=w_i-Bv_i$, and let $r=\dim\spn{u_1,\dots,u_t}$. 
Then 
$$\card{U'}
\le (1-c_0\cdot q^{-rt})\card{U}.$$ 

\end{lem}

The following lemma follows from Lemma~\ref{lem:number of different t umvirate families with t-j intersection with U}.
\begin{lem}[The \(\SL(n,q)\) version of Lemma~\ref{lem:number of different t umvirate families with t-j intersection with U}]\label{lem:number of different t umvirate families with t-j intersection with U-3}
Let $U$ be an $\SL(n,q)$ $t$-umvirate fixed on a $t$-dimensional subspace $V'\le V$.  
For $1\le r\le t$, let $\mathcal U_r$ be the set of $\SL(n,q)$ $t$-umvirates $U'$, also fixed on $V'$, that agree with $U$ on exactly a $(t-r)$-dimensional subspace of $V'$. Then
\[
|\mathcal U_r|\le 4q^{r(n+t-r)}.
\]
\end{lem}
\begin{thm}[Decoupled version of Theorem~2 for \(\mathrm{SL}(n,q)\)]\label{thm:main:sln}
There exists an absolute constant \(C_{\mathrm{SL}}>0\) such that the following holds. 
If \(n\ge C_{\mathrm{SL}}t\) and
\(\mathcal F,\mathcal G\subseteq \mathrm{SL}(n,q)\) are
cross-\((t-1)\)-intersection-free, then
\[
|\mathcal F|\,|\mathcal G|\le \frac{1}{(q-1)^2}\prod_{i=1}^{n-t}\bigl(q^n-q^{i+t-1}\bigr)^2.
\]
Moreover, equality holds only if \(\mathcal F=\mathcal G\), and $\mathcal F$ is an
\(\mathrm{SL}(n,q)\) \(t\)-umvirate or a dual \(\mathrm{SL}(n,q)\) \(t\)-umvirate.
\end{thm}

\begin{proof}
Assume \(\mathcal F,\mathcal G\subseteq \mathrm{SL}(n,q)\)
is a cross-\((t-1)\)-intersecting free pair
whose product
\(|\mathcal F|\,|\mathcal G|\) is maximum.

Notice that a (dual) \(\mathrm{SL}(n,q)\) \(t\)-umvirate has size
$M_{\mathrm{SL}}
 :=
 \frac{1}{q-1}\prod_{i=1}^{n-t}\bigl(q^n-q^{i+t-1}\bigr)
 \ge \frac14 q^{n(n-t)-1} .$

Since such a family is itself cross-\((t-1)\)-intersection-free, maximality gives
$
|\mathcal F|\,|\mathcal G|\ge M_{\mathrm{SL}}^2.$

Let \(C_{\mathrm{SL}}\) so that \(n\ge C_{\mathrm{SL}}t\) gives that all applications of Lemmas~\ref{lem:iterative_argument}, \ref{lem:families_almost_contained_in_U_small_intersection_other_umvirates}, and \ref{lem:counting_t-1_free_with_t_umvirate-2} below are valid.

We apply Lemma~\ref{lem:iterative_argument} with \(m=n\) and \(k=6\). Indeed, the preceding lower bound gives
\[
M_{\mathrm{SL}}^2
 \ge \frac1{16}q^{2n(n-t)-2}
 \ge q^{-6}q^{2n(n-t)}
\]
for all \(q\ge2\). Since \(n\ge C_{\mathrm{SL}}t\), the error terms in the lower bound~\eqref{eq:size_lower_bound_1} of
Lemma~\ref{lem:iterative_argument} vanish when \(m=n\). Therefore Lemma~\ref{lem:iterative_argument} gives a \(t\)-umvirate
\(\widetilde U\subseteq L(V,V)\), or a dual \(t\)-umvirate
\(\widetilde U^*\subseteq L(V,V)\), such that
\[
|\mathcal F\cap \widetilde U|\ge (1-\gamma)|\mathcal F|,
\qquad
|\mathcal G\cap \widetilde U|\ge (1-\gamma)|\mathcal G|,
\]
or the analogous conclusion with \(\widetilde U^*\), where
\[
\gamma:=tq^{\xi t-n},\qquad \xi:=k+10.
\]
After increasing \(C_{\mathrm{SL}}\), we may assume \(\gamma<1/2\). After replacing \((\mathcal F,\mathcal G)\) by
\((\mathcal F^\top,\mathcal G^\top)\) if necessary, we may assume that \(\widetilde U\) is an ordinary
\(\mathrm{SL}(n,q)\) \(t\)-umvirate.

Write $U:=\widetilde U\cap \mathrm{SL}(n,q)$, and 

\[
U=\{A\in \mathrm{SL}(n,q): Av_i=w_i\text{ for }1\le i\le t\},
\]
where \(v_1,\ldots,v_t\in V\) and \(w_1,\ldots,w_t\in V\) are linearly independent. Let
$V_0:=\langle v_1,\ldots,v_t\rangle.$

Since \(|\mathcal F\cap U|\le |U|=M_{\mathrm{SL}}\) and
\(|\mathcal G\cap U|\le |U|=M_{\mathrm{SL}}\), we have
\[
|\mathcal F\cap U|\,|\mathcal G\cap U|
 \ge (1-\gamma)^2 |\mathcal F|\,|\mathcal G|
 \ge (1-\gamma)^2 M_{\mathrm{SL}}^2.
\]
Consequently,
\[
|\mathcal F\cap U|\ge (1-\gamma)^2M_{\mathrm{SL}},
\qquad
|\mathcal G\cap U|\ge (1-\gamma)^2M_{\mathrm{SL}}.
\]

Assume that at least
one of \(\mathcal F,\mathcal G\) is not contained in \(U\), otherwise we are done.

For \(1\le r\le t\), let \(\mathcal U_r\) be the collection of all
\(\mathrm{SL}(n,q)\) \(t\)-umvirates fixed on \(V_0\) that agree with \(U\) on exactly a
\((t-r)\)-dimensional subspace of \(V_0\). By Lemma~\ref{lem:number of different t umvirate families with t-j intersection with U-3},
$|\mathcal U_r|\le 4q^{r(n+t-r)}.$

Let \(\kappa>1\) be the constant supplied by Lemma~\ref{lem:families_almost_contained_in_U_small_intersection_other_umvirates} for the fixed value
\(k=6\). Suppose first that \(\mathcal F\setminus U\neq\varnothing\), and define
\[
r_{\mathcal F}
 :=
 \min\{\operatorname{rank}((A-B)|_{V_0}): A\in\mathcal F\setminus U,\ B\in U\}.
\]
Since all elements of \(U\) agree on \(V_0\), this rank is independent of the particular
choice of \(B\in U\). Clearly \(1\le r_{\mathcal F}\le t\).

Divide \( \mathcal F\setminus U\) according to the
\(\mathrm{SL}(n,q)\) \(t\)-umvirates fixed on \(V_0\) to get:
\[
\mathcal F\setminus U
 \subseteq
 \bigcup_{r=r_{\mathcal F}}^t
 \bigcup_{U'\in\mathcal U_r}
 (\mathcal F\cap U').
\]
Applying Lemma~\ref{lem:families_almost_contained_in_U_small_intersection_other_umvirates} with \(m=n\), we get
$|\mathcal F\cap U'|
 \le q^{n(n-t)}q^{-2r(n-\kappa r)}.$
 Therefore
\[
\begin{aligned}
|\mathcal F\setminus U|
&\le
 \sum_{r=r_{\mathcal F}}^t
 \sum_{U'\in\mathcal U_r}
 |\mathcal F\cap U'|  \le
 4q^{n(n-t)}
 \sum_{r=r_{\mathcal F}}^t
 q^{r(n+t-r)-2r(n-\kappa r)} \\
&=
 4q^{n(n-t)}
 \sum_{r=r_{\mathcal F}}^t
 q^{-rn+rt+(2\kappa-1)r^2}
 \le 8q^{n(n-t)}q^{-r_{\mathcal F}(n-2\kappa t)}
 \le C_1 q^{-r_{\mathcal F}(n-Kt)}M_{\mathrm{SL}} .
\end{aligned}
\]
where \(C_1>0\) and \(K>0\) are absolute constants  depending only on \(\kappa\), and in the second to last inequality we assumed \(n\ge (2\kappa+1)t\).

The same argument gives the following: if \(\mathcal G\setminus U\neq\varnothing\) and
$r_{\mathcal G}$ defined respectively, then
\[
|\mathcal G\setminus U|
 \le C_1 q^{-r_{\mathcal G}(n-Kt)}M_{\mathrm{SL}}.
\]

Define $r' = \min(r_{\mathcal F},r_{\mathcal G})$.
By symmetry between \(\mathcal F\) and \(\mathcal G\), assume that this minimum is attained
by some \(B_0\in\mathcal F\setminus U\). Then \(r'=r_{\mathcal F}\). The preceding bounds
imply
\[
|\mathcal F\setminus U|
 \le \varepsilon M_{\mathrm{SL}},
\qquad
|\mathcal G\setminus U|
 \le \varepsilon M_{\mathrm{SL}},
\]
where $\varepsilon:=C_1q^{-r'(n-Kt)}.$

Notice that every element of
\(\mathcal G\cap U\) must avoid having intersection dimension exactly \(t-1\) with \(B_0\). Apply Lemma~\ref{lem:counting_t-1_free_with_t_umvirate-2} to \(\mathcal G\cap U\) and
\(B_0\) to get:
\[
|\mathcal G\cap U|
 \le (1-c_0q^{-r't})M_{\mathrm{SL}}.
\]
Set
$\delta:=c_0q^{-r't}.$
Using the trivial bound \(|\mathcal F\cap U|\le M_{\mathrm{SL}}\), we obtain
\[
|\mathcal F|\,|\mathcal G|
 \le
 (1+\varepsilon)(1-\delta+\varepsilon)M_{\mathrm{SL}}^2.
\]
Notice that by increasing \(C_{\mathrm{SL}}\) if needed, we get that $\varepsilon\le \frac{\delta}{8}$ for all \(n\ge C_{\mathrm{SL}}t\). Then
\[
(1+\varepsilon)(1-\delta+\varepsilon)
 \le 1-\delta+2\varepsilon+\varepsilon^2
 <1.
\]
Thus
$|\mathcal F|\,|\mathcal G|<M_{\mathrm{SL}}^2$,
contradicting the maximality of \(|\mathcal F|\cdot|\mathcal G|\).

Therefore \(\mathcal F,\mathcal G\subseteq U\), and hence
$|\mathcal F|\,|\mathcal G|\le |U|^2=M_{\mathrm{SL}}^2.$
Equality forces \(|\mathcal F|=|\mathcal G|=|U|\), so
\(\mathcal F=\mathcal G=U\). In the dual case, transposing back gives
\(\mathcal F=\mathcal G=U^*\). This proves the theorem.
\end{proof}

We now deduce the \(\GL(n,q)\) version from the \(\SL(n,q)\) version by decomposing
\(\GL(n,q)\) into determinant fibers.

\begin{proof}[Proof of Theorem~\ref{thm:main2}]
For \(\delta\in\mathbb F_q^\times\), write
$\GL_\delta(n,q)=\{A\in\GL(n,q):\det A=\delta\}.$

Each \(\GL_\delta(n,q)\) is a left coset of \(\SL(n,q)\). Fixing \(M_\delta\in
\GL_\delta(n,q)\), the map \(A\mapsto M_\delta^{-1}A\) is a bijection
$\GL_\delta(n,q)\to \SL(n,q)$
and preserves \(\dim\ker(A-B)\) for all pairs \(A,B\).

Let \(\mathcal F\subseteq\GL(n,q)\) be \((t-1)\)-intersection-free, and put
$\mathcal F_\delta=\mathcal F\cap \GL_\delta(n,q).$

Then \(M_\delta^{-1}\mathcal F_\delta\subseteq\SL(n,q)\) is also
\((t-1)\)-intersection-free. Hence, by Theorem~\ref{thm:main:sln},
$|\mathcal F_\delta|\le M_{\SL}$
for every \(\delta\in\mathbb F_q^\times\), where \(M_{\SL}\) is the size of an
\(\SL(n,q)\) \(t\)-umvirate. Summing over \(\delta\in\mathbb F_q^\times\) gives
\[
|\mathcal F|\le (q-1)M_{\SL} = \prod_{i=1}^{n-t} (q^n-q^{i+t-1}).
\]
which is the size of a \(\GL(n,q)\) \(t\)-umvirate. This proves the desired upper
bound in \(\GL(n,q)\).

It remains to discuss equality. If equality holds, then equality holds in every $\mathcal F_{\delta}$, meaning for each \(\delta\in\mathbb F_q^\times\),
$\mathcal F\cap\GL_\delta(n,q)$ is either the \(\delta\)-slice of a full \(t\)-umvirate in \(\GL(n,q)\), or the
\(\delta\)-slice of a full dual \(t\)-umvirate in \(\GL(n,q)\).

We claim that all slices come from the same full extremal (dual) umvirate. Fix
\(\delta_0\). There exists a full
extremal family \(U^{\mathrm{full}}\subseteq \GL(n,q)\), either a \(t\)-umvirate or a dual
\(t\)-umvirate, such that
$\mathcal F_{\delta_0}=U^{\mathrm{full}}\cap\GL_{\delta_0}(n,q).
$

 If for some \(\delta\) we had \(\mathcal F_\delta\not\subseteq U^{\mathrm{full}}\), then by
Lemma~\ref{lem:outside-point-fixed-det-slice} there exist
\(A\in\mathcal F_{\delta_0}\) and \(B\in\mathcal F_\delta\) with
\(\dim\ker(A-B)=t-1\), contradicting the forbidden-intersection condition. Hence
\(\mathcal F_\delta\subseteq U^{\mathrm{full}}\) for every \(\delta\). Therefore \(\mathcal F\subseteq U^{\mathrm{full}}\).
\end{proof}
We defer the proof of Lemma~\ref{lem:outside-point-fixed-det-slice} to the Appendix.
\begin{lem}\label{lem:outside-point-fixed-det-slice}
Assume \(n\ge 3t\). Let
$U_\eta=\{A\in \GL(n,q):\det A=\eta,\ Av_i=w_i \text{ for }1\le i\le t\}
$
be a determinant slice of a \(t\)-umvirate, and let
$U=\{A\in \GL(n,q): Av_i=w_i \text{ for }1\le i\le t\}$
be the corresponding full \(t\)-umvirate.

Then for every \(B\in \GL(n,q)\setminus U\), there exists
\(A\in U_\eta\) such that
$\dim\ker(A-B)=t-1.$
\end{lem}

\section{Frankl-R\"odl-Type Constructions}\label{sec:FR} We conclude the paper with several constructions of large $(t-1)$-intersection-free families of $\mathbb{F}_q^{n\times n}$ in the $t = \Omega(n)$ regime that are larger than the $t$-umvirates and also larger than the following variant of the Frankl families $F_{n,k,t,r}$ for sets discussed in Section~\ref{sec:new_intro}.

Let $v = v_1,\ldots,v_{t+2r}$ be a list of linearly independent vectors of $\mathbb{F}_q^n$, and let $w = w_1,\ldots,w_{t+2r}$ be a list of vectors of $\mathbb{F}_q^n$. We define the \emph{Frankl families} of $\mathbb{F}_{q}^{n \times n}$ to be
\[
\mathfrak{F}(v,w) := \left \{ A \in \mathbb{F}_q^{n \times n} : Av_i = w_i \text{ for at least $t+r$ of the ordered pairs $(v_i,w_i) \in \{(v_j,w_j)\}_{j=1}^{t+2r}$}  \right \}.
\]
It is easy to see that Frankl families are $(t-1)$-intersection-free and largest when $r=0$. 

 By intersecting the constructed families below with $\GL(n,q)$, we also get large $(t-1)$-intersection-free families of $\GL(n,q)$. We keep the size analysis inside $\mathbb{F}_q^{n\times n}$ for simplicity.

The first family that we construct will be a $(t-1)$-intersection-free family that is larger than a $t$-umvirate (or any Frankl family) for $t \geq n/2+1$. 

\subsubsection*{Polynomial Construction}
Our strategy is to force all intersection dimensions to be strictly less than $t$. We identify the vector space $\mathbb{F}_q^n$ with the field extension $\mathbb{F}_{q^n}$. We let $v_\theta \in \mathbb{F}_q^n$ denote the vector corresponding to the field element $\theta \in \mathbb{F}_{q^n}$ under this isomorphism.
Define $\mathcal{F}_{\text{alg}}$ as the set of linearized polynomials with $q$-degree less than $t-1$, that is,
    \[ \mathcal{F}_{\text{alg}} := \left\{ \sum_{i=0}^{t-2} c_i x^{q^i} \;\middle|\; c_i \in \mathbb{F}_{q^n} \right\} \subseteq \mathbb{F}_{q^n}[x]. \]
Note that $p(x) \in \mathbb{F}_{q^n}[x]$ is a linear transformation, i.e., $\alpha \in \mathbb{F}_q, a,b \in \mathbb{F}_{q^n}$, $p(\alpha a+b) = \alpha p(a)+p(b)$. 
Thus, we may identify each $p(x) \in \mathcal{F}_{\text{alg}}$ with a matrix $A_p \in \mathbb{F}_q^{n \times n}$ acting linearly on $\mathbb{F}_q^n$, so that $p(\theta) = A_p v_\theta$. Moreover, if $p(x),p'(x) \in \mathcal{F}_{\text{alg}}$ are two distinct polynomials, then $A_p \neq A_{p'}$. In particular, the set $\mathcal{F}_{\text{alg}}$ corresponds to a subspace $M_{\text{alg}} \subseteq \mathbb{F}_q^{n \times n}$ such that $| \mathcal{F}_{\text{alg}}| = |M_{\text{alg}}|$.

The roots of any $p(x) \in \mathcal{F}_{\text{alg}}$ generate the subspace $\{v \in \mathbb{F}_q^n : A_pv = 0\} \leq \mathbb{F}_q^n$, which has dimension at most $t-2$; therefore, we have 
\[
\dim \ker(A-B) \le t-2
\]
for any two distinct $A,B \in M_{\text{alg}}$. Furthermore, we have $|M_{\text{alg}}| = (q^n)^{t-1} = q^{n(t-1)}$, which is larger than the $t$-umvirate of size $q^{n^2-nt}$, provided that $t\ge n/2 +1$.

\subsubsection*{Parity Construction}
We now give a simple construction that forces the parity of the intersection of any two matrices to be divisible by 2, which will show the family must be $(t-1)$-intersection-free for all even $t$.

Let $n=2a$. Let $\mathcal{F}$ be the set of $\mathbb{F}_{q^2}$-linear transformations on $\mathbb{F}_{q^2}^a$. Note that such transformations can be written as the set of all $a \times a$ matrices over the vector space $\mathbb{F}_{q^2}^a$, or as a subset of $n \times n$ matrices over $\mathbb{F}_q^{n}$. For any two linear transformations $A,B \in \mathcal{F}$ represented as $a \times a$ matrices over $\mathbb{F}_{q^2}$, we have $\ker(A-B) \cong \mathbb{F}_{q^2}^{a'} \leq \mathbb{F}_{q^2}^a$. Now if we let $A',B'$ be representations of $A,B$ as $n \times n$ matrices with entries in $\mathbb{F}_q$, then we have $\ker (A'-B') \cong  \mathbb{F}_{q}^{2a'} \leq \mathbb{F}_{q}^{2a} = \mathbb{F}_{q}^{n}$, which shows that any two linear transformations of $\mathcal{F}$ even have intersection as $n \times n$ matrices. Therefore, if $t$ is even, then $\mathcal{F} \subseteq \mathbb{F}_q^{n \times n}$ is $(t-1)$-intersection-free. Moreover, we have $|\mathcal{F}| = (q^2)^{a \cdot a} = q^{2a^2}$, which is larger than the $t$-umvirate of size $q^{(2a)^2-2at}$ when $t \geq (n+1)/2$.

\subsubsection*{Probabilistic Construction (Alteration)} We now construct a large family that is $(t-1)$-intersection-free and larger than any $t$-umvirate or Frankl family when $t \approx ((\sqrt{5}-1)/2)n$, for example.

Let $a,k > 0$ be parameters to be chosen later. Draw a family $\mathcal{A} \subseteq \mathbb{F}_q^{n \times n}$ of size $aq^k$ uniformly at random. 
Using the standard Gaussian coefficient estimates in Section~\ref{sec:step1} and basic counting, the number of rank $n-t+1$ matrices of $\mathbb{F}_q^{n \times n}$ is $v_{n-t+1} = \Theta(q^{(n+t-1)(n-t+1)})$, so the probability that the difference of any two matrices of $\mathcal{A}$ has kernel dimension exactly $t-1$ is $c_q/q^{(t-1)^2}$ for some constant $c_q$ depending only on $q$. Now for each pair of matrices of $\mathcal{A}$ that has kernel dimension exactly $t-1$, we remove one of the matrices from $\mathcal{A}$. Let $\mathcal{A}' \subseteq \mathcal{A}$ be the resulting family. Note that $\mathcal{A}'$ is $(t-1)$-intersection-free, and in expectation, we remove at most $\binom{a q^k}{2} c_q q^{-(t-1)^2}$ matrices from $\mathcal{A}$ to obtain $\mathcal{A}'$. Thus, there exists a choice of $\mathcal{A}$ requiring at most this many removals.

Now set $k = (t-1)^2$ and take $a>0$ to be sufficiently small. Then there exists a choice of $\mathcal{A}$ where we remove at most
\[
\frac{a q^{(t-1)^2}(a q^{(t-1)^2}-1)}{2} \cdot \frac{c_q}{ q^{(t-1)^2}} \leq c_q a^2q^{(t-1)^2} 
\]
matrices from $\mathcal{A}$ to make it $(t-1)$-intersection-free. So we have that $\mathcal{A}'$ is a family of size $a_q \cdot q^{(t-1)^2}$ for some constant $a_q$ depending only on $a,q$, which is larger than a $t$-umvirate or Frankl family when, for example, $t \approx ((\sqrt{5}-1)/2)n$, and $n$ is sufficiently large.

\subsection*{Acknowledgements}

We would like to thank Noam Lifshitz, Nathan Keller, and Guy Kindler for their comments and suggestions, which substantially improved the paper. We also thank the Simons Institute for the Theory of Computing and the organizers of the \emph{Analysis and TCS: New Frontiers} program, where this work was conceived. We used AI-based writing tools to assist with proofreading, correcting grammatical and typographical mistakes, rephrasing text, and proposing simplifications for specific parts of the exposition.

\bibliographystyle{alpha}
\bibliography{master}

\appendix

\section{Comparison of~\cite{EllisKL24} and~\cite{ernst2023intersection}}\label{sec:app}

The techniques of Ellis et al.~are analytic whereas Ernst et al.~follow an algebraic approach. From an algebraic point of view, it is convenient to think of the finite general linear group as a $q$-analogue of the symmetric group, as many of the representation-theoretical properties of $S_n$ have natural analogues in $\GL(n,q)$. This analogy was indeed instrumental in~\cite{ernst2023intersection} for deriving the $q$-analogue of the Erd\H{o}s--Ko--Rado theorem for $t$-intersecting permutations~\cite{EllisFP11}. However, from an analytic perspective, the finite general linear groups and symmetric groups are remarkably different. For instance, by Stirling's approximation
\[
    n! \sim \sqrt{2\pi n}\left( \frac{n}{e} \right)^n \ll 2^{n^2},
\]
one observes that the embedding of $S_n \subseteq \mathbb{Z}_2^{n \times n}$ as permutation matrices or the embedding of $S_n \subseteq [n]^n$ as words in a product space incurs an exponential loss of measure, whereas it is a well-known fact that a matrix drawn uniformly at random from $\mathbb{F}_q^{n \times n}$ will be invertible with probability at least $1/4$ for any prime power $q$.
The latter bodes well for analytic approaches, and we believe that other extremal combinatorial problems for matrix groups are perhaps more amenable to analytical methods, despite their algebraic appearances.   

It is plausible that algebraic and representation-theoretic methods can be used to prove $t$-intersecting results for other matrix groups, but it seems this would require ad-hoc algebraic arguments for each group, and the particulars of their representation theories can be rather baroque. On the other hand, the analytical approaches not only give more general results, but they also seem more resilient to slight changes in the structure of the domain (e.g., $\GL(n,q)$ versus $\SL(n,q)$). Indeed, the current state-of-the-art algebraic techniques appear to be incapable of proving $(t-1)$-intersection-free results for arbitrary $t \in \mathbb{N}$. This is because non-trivial edge-weightings of the `disjointness graph' that models the $(t-1)$-intersection-free property (where two vertices are non-adjacent if their intersection is $(t-1)$-intersection-free) do not respect the `natural' symmetries of the graph's automorphism group, making it too difficult to control the eigenvalues of the associated weighted adjacency matrix. The latter is needed in order to bound the graph's independence number via the Delsarte--Hoffman bound. 

In principle, the algebraic approach is capable of proving exact (i.e., non-asymptotic) $t$-intersecting results for matrix groups, but it seems this will require significantly new insights.

\section{Missing Proofs}\label{apx:missing_proofs}

\begin{proof}[Proof of Lemma~\ref{lem:projection_exact_intersection}]
Let $T = \ker(A-B)$ and $T' = \ker(p(A)-p(B))$. Given the decomposition $V = V' \oplus \spn{v}$, every vector $x \in V$ can be uniquely written as $x = u + cv$ for some $u \in V'$ and $c \in \mathbb{F}_q$. We consider the projection map $\phi: V \to V'$ defined by $\phi(u + cv) = u$ and show that its restriction to $T$ is a linear isomorphism onto $T'$.

First, we verify that $\phi(T) \subseteq T'$. For any $t \in T$, we write $t = t' + cv$ where $t' = \phi(t) \in V'$. Since $(A-B)t = \bar{0}$, linearity implies $(A-B)t' = -c(A-B)v = -c(w-w')$. Applying the projection $\pi$ to both sides, we see that $(p(A)-p(B))t' = \pi(-c(w-w')) = \bar{0}$ because $w-w'$ is in the kernel of $\pi$. This confirms that $t' \in T'$.

To see that this mapping is injective, suppose $\phi(t) = \bar{0}$ for some $t \in T$. This implies $t \in \spn{v}$. However, the assumption $w \neq w'$ implies $(A-B)v \neq \bar{0}$, so $v \notin T$. Consequently, $T \cap \spn{v} = \{\bar{0}\}$, and the restriction $\phi|_T$ is injective.

Finally, we establish surjectivity by showing that every $u \in T'$ has a pre-image in $T$. By the definition of $T'$, the vector $(A-B)u$ must lie in the kernel of $\pi$, meaning $(A-B)u = k(w-w')$ for some scalar $k \in \mathbb{F}_q$. Substituting $(A-B)v$ for $w-w'$, we find that $(A-B)(u - kv) = \bar{0}$, which identifies $t = u - kv$ as an element of $T$. Since $\phi(t) = u$ by construction, $\phi|_T$ is surjective. Having shown that $\phi|_T: T \to T'$ is an isomorphism, we conclude that $\dim T = \dim T'$.
\end{proof}

\begin{proof}[Proof of Lemma~\ref{lem:counting_t-1_free_with_t_umvirate}]
   Let $v_1,\dots, v_t, v_{t+1},\dots, v_n$ be a basis of $V$ formed by extending the sequence $v_1,\dots, v_t$.
Then any linear map in $\LL(V,W)$ is determined by the images of the vectors $v_i$. To begin, observe that $\card{\mathcal{U}}=\card{W}^{n-t}=q^{m(n-t)}$.

 We now proceed to bound the number of elements in $\bar{\mathcal{U}}'=\mathcal{U}\setminus\mathcal{U}'$. For any fixed $A\in \mathcal{U}$, define $D=A-B$. 
 The action of $D$ on $v_1, \dots, v_t$ is completely determined by the following relations
$$Dv_i = Av_i - Bv_i = w_i - Bv_i = u_i,\quad i=1,\dots,t.$$
In contrast, since $A$ varies over all mappings in $\mathcal{U}$, the
action of $D$ on $v_{t+1},\dots , v_n$ 
can be chosen arbitrarily.

Let $U=\spn{u_1,\dots,u_t}\le W$  and $r=\dim(U)$. Since $B\notin\mathcal{U}$, we have  $r\ge 1$, so that $1\le r\le t$. 
The condition $\dim\ker(D)=t-1$ is equivalent to $\rank(D)=n-t+1$. Since $\dim(U)=r$, we must complete the definition of $D$ by extending $U$ to a subspace of $W$ whose dimension is $n-t+1$. Put differently, our goal is to increase the rank of $D$ by exactly $n-t+1-r$ by exploiting the freedom to map $v_{t+1},\dots,v_n$ arbitrarily into $W$.
Let $k=n-t$, and let $z_1,\dots, z_k\in W$ denote the images under $D$ of $v_{t+1}, \dots ,v_n$, that is,  $z_j=Dv_{t+j}$. 
Now, 
$$\rank(D)=\dim(\spn{u_1,\dots, u_t,z_1,\dots,z_k})=\dim(U + \spn{z_1,\dots,z_k})=r+\dim(\spn{z_1,\dots,z_k}/U).$$
The number of elements in $\bar{\mathcal{U}}'$ is exactly the number of $k$-tuples $z_1,\dots,z_k\in W$ for which
$$\dim(\spn{z_1,\dots,z_k}/U)=n-t+1-r.$$
To choose a target image subspace of dimension $n-t+1-r$ in $W/U$ there are 
$${m-r \brack n-t+1-r}_q\ge q^{(n-t+1-r)(m-n+t-1)}$$ 
possible choices. 

Next, mapping the $k$ vectors $v_{t+1},\dots,v_{n}$ surjectively onto the chosen subspace is equivalent to having a full-rank matrix of size $(n-t+1-r)\times (n-t)$. There are 
$$\prod_{j=0}^{n-t-r}(q^{n-t}-q^j)=q^{(n-t)(n-t-r+1)}\prod_{j=0}^{n-t-r}(1-q^{j-n+t})\ge c\cdot q^{(n-t)(n-t-r+1)}$$
many options for this (using the fact that $\prod_{j=0}^{n-t-r}(1-q^{j-n+t})>\prod_{k=1}^\infty (1-q^{-k}) = c(q)$ and $c(q)$ is minimized for $q=2$ with $c(2)\ge 0.288$). Finally, lifting the mapped vectors from $W/U$ back to $W$, we have $q^{\dim(U)}=q^r$ options for each and $q^{r(n-t)}$ in total.
Putting everything together and observing that the dependence on $r$ in the surjection and lifting steps cancels out, we obtain
$$\card{\bar{\mathcal{U}}'}\ge  q^{(n-t+1-r)(m-n+t-1)}\cdot c\cdot q^{(n-t)(n-t+1)}=c\cdot q^{m(n-t)-(r-1)(m-n+t-1)}.$$ 
\end{proof}
\begin{proof}[Proof of Lemma~\ref{lem:number of different t umvirate families with t-j intersection with U}]
    Suppose $\mathcal{U}$ is defined by the mapping $v_i \mapsto w_i$ for $i \in \{1,\dots, t\}$. Any other $t$-umvirate $\mathcal{U}'$ defined on $V'$ is uniquely determined by some mapping $v_i \mapsto w_i'$. 
    
    Consider the difference map $D \in \LL(V',W)$ defined by $D(v_i) = w_i - w_i'$. The umvirates $\mathcal{U}$ and $\mathcal{U}'$ agree exactly on a $(t-r)$-dimensional subspace of $V'$ if and only if the kernel of $D$ has dimension $t-r$. By the rank-nullity theorem, this is equivalent to $D$ having rank exactly $r$. Since shifting by the fixed target vectors $w_i$ is a bijection, counting the number of such umvirates $\mathcal{U}'$ is precisely equivalent to counting the number of rank-$r$ linear maps from $V'$ to $W$.
    
    To construct a rank-$r$ map $D \in \LL(V',W)$, we first choose an $r$-dimensional target subspace $W' \le W$, and then choose a surjective linear map from $V'$ to $W'$. 
    The number of ways to choose $W'$ is given by the Gaussian binomial coefficient:
    $$ \begin{bmatrix} m \\ r \end{bmatrix}_q = \prod_{i=0}^{r-1}\frac{q^m-q^i}{q^r-q^i}\ . $$
    The number of surjective linear maps from a $t$-dimensional space to an $r$-dimensional space is:
    $$ \prod_{i=0}^{r-1}(q^t-q^i)\ . $$
    Multiplying these together, the number of such $t$-umvirates is:
    $$ \prod_{i=0}^{r-1}\frac{(q^m-q^i)(q^t-q^i)}{q^r-q^i}= q^{r(m+t-r)}\prod_{i=0}^{r-1}\frac{(1-q^{i-m})(1-q^{i-t})}{1-q^{i-r`}} \le 4q^{r(m+t-r)}\ . $$
    The last inequality holds since $(1-q^{i-m})(1-q^{i-t}) \le 1,$ and note that for any $q\ge 2$ we can bound  $\prod_{i=0}^{r-1}(1-q^{i-r})=\prod_{i=1}^{r}(1-q^{-i})\ge \prod_{i=1}^{\infty}(1-q^{-i})\ge 1/4.$
    
\end{proof}

The proof of Lemma~\ref{lem:counting_t-1_free_with_t_umvirate-2} is longer than the proof
of Lemma~\ref{lem:counting_t-1_free_with_t_umvirate} because the same geometric argument
has to be carried out inside \(\SL(n,q)\). In the full linear space \(L(V,W)\), once the values
of a map on \(v_1,\ldots,v_t\) are fixed, the remaining values can be chosen independently.
In \(\SL(n,q)\), this is no longer true: the remaining choices must satisfy a determinant
constraint. For this we need the following new lemma:
\begin{lem}[Derangements in a fixed determinant fiber]\label{lem:fixed-det-derangements}
Let \(n\ge 2\), and let
\(\delta\in \mathbb F_q^\times\). Then
\[
\bigl|\{Q\in \GL(n,q): \det Q=\delta,\ Q-I \text{ is invertible}\}\bigr|
\ge
\frac{2}{7}\cdot
\bigl|\{Q\in \GL(n,q): \det Q=\delta\}\bigr|.
\]

\end{lem}

\begin{proof}
Let
\[
G_n^\delta=\{Q\in \GL(n,q):\det Q=\delta\} \quad \text{and} \quad D_n^\delta=\{Q\in G_n^\delta:\ker(Q-I)=0\}.
\]
For any subspace \(U\le \mathbb F_q^n\), let \(A_U^\delta\) be the set of matrices in \(G_n^\delta\) that fix \(U\) pointwise. By M\"obius inversion on the subspace lattice,
\[
|D_n^\delta|=\sum_{U\le \mathbb F_q^n} \mu(0,U)\cdot |A_U^\delta|, \qquad \text{where } \mu(0,U)=(-1)^{\dim U} q^{\binom{\dim U}{2}}.
\]

We evaluate this sum by grouping the terms according to the dimension \(d = \dim U\). When \(d < n\), picking a complement to \(U\) allows us to write any matrix in \(A_U^\delta\) in the block form
\[
\begin{pmatrix}
I_d & C\\
0 & B
\end{pmatrix},
\]
where \(C\) is arbitrary and \(B\in \GL^{\delta}_{n-d}.\) This gives \(|A_U^\delta| = q^{d(n-d)}|\GL(n-d,q)|/(q-1)\). 

To find the fractional contribution of all \(d\)-dimensional subspaces to \(|D_n^\delta|/|G_n^\delta|\), we multiply the term $\mu(0,U)\cdot |A_U^\delta|$ by the Gaussian binomial coefficient
\[
\begin{bmatrix} n \\ d \end{bmatrix}_q = \frac{|\GL(n,q)|}{q^{d(n-d)} |\GL(d,q)| |\GL(n-d,q)|}
\]
which corresponds to the number of subspaces $U$ of dimension $d$, and divide by \(|G_n^\delta| = |\GL(n,q)| / (q-1)\). A straightforward cancellation shows that the normalized contribution for a fixed \(d < n\) is exactly
\[
(-1)^d \frac{q^{\binom d2}}{|\GL(d,q)|}.
\]

The only remaining case is \(d=n\). The identity matrix is the only one fixing all of \(\mathbb F_q^n\), and its determinant is \(1\). Thus, the \(d=n\) term appears if and only if \(\delta=1\). 

Setting \(\alpha_0=1\) and
\[
\alpha_d = \frac{q^{\binom d2}}{|\GL(d,q)|} = \prod_{i=1}^d \frac{1}{q^i-1}
\]
for \(d \ge 1\), summing over all dimensions yields an exact formula for the proportion of derangements:
\[
\frac{|D_n^\delta|}{|G_n^\delta|}
=
\sum_{d=0}^{n-1} (-1)^d\alpha_d
+
\mathbf 1_{\delta=1} (-1)^n (q-1)\alpha_n.
\]
Observe that \(\alpha_d / \alpha_{d-1} = 1/(q^d-1)\), meaning the sequence \((\alpha_d)\) is strictly decreasing for \(d \ge 1\). We now bound this sum from below.

If \(q=2\), we must have \(\delta=1\), and the sum becomes \(\sum_{d=0}^{n} (-1)^d\alpha_d\). Since \(\alpha_0=\alpha_1=1\), the first two terms cancel. For \(n=2\), the sum evaluates exactly to \(\alpha_2 = 1/3 \ge 2/7\). For \(n \ge 3\), the alternating sum of decreasing terms is bounded below by its first two non-vanishing terms:
\[
\alpha_2 - \alpha_3 = \frac{1}{3} - \frac{1}{21} = \frac{2}{7}.
\]

Next, assume \(q \ge 3\). If \(\delta \ne 1\), the sum truncates at \(n-1\). As an alternating sum of decreasing positive terms, it is bounded below by
$1 - \alpha_1 = 1 - \frac{1}{q-1} \ge \frac{1}{2}.$

If \(\delta = 1\) and \(n\) is even, the extra term \((q-1)\alpha_n\) is positive, meaning the same \(1/2\) lower bound applies. Finally, if \(\delta = 1\) and \(n\) is odd, we must have \(n \ge 3\), which implies \(\alpha_n \le \alpha_3\). We can lower-bound the proportion by truncating the sum:
\[
\frac{|D_n^1|}{|G_n^1|} \ge 1 - \alpha_1 + \alpha_2 - \alpha_3 - (q-1)\alpha_n \ge 1 - \alpha_1 + \alpha_2 - q\alpha_3.
\]
Since \(\alpha_3 = \alpha_2 / (q^3-1)\), we see that
$\alpha_2 - q\alpha_3 = \alpha_2 \left(1 - \frac{q}{q^3-1}\right) > 0.$

This ensures the proportion remains strictly greater than \(1 - \alpha_1 \ge 1/2\). So in all cases, the fraction of derangements is bounded below by \(2/7\).
\end{proof}
Instead of proving Lemma~\ref{lem:counting_t-1_free_with_t_umvirate-2} directly, we prove
the following slightly stronger fixed-determinant version. The original lemma is the special
case \(\eta=1\), and the stronger form will be used in
Lemma~\ref{lem:outside-point-fixed-det-slice}.

\begin{lem}
\label{lem:counting_t-1_free_with_t_umvirate-fixed-det}
There is an absolute constant \(c_0>0\) such that the following holds.
Assume \(n\ge 3t\). Let \(\eta\in\mathbb F_q^\times\), and let
\[
U_\eta=\{A\in \GL(n,q): \det A=\eta,\ Av_i=w_i \text{ for }1\le i\le t\}
\]
be non-empty. Let
\[
U=\{A\in \GL(n,q): Av_i=w_i \text{ for }1\le i\le t\}
\]
be the corresponding full \(t\)-umvirate. Let \(B\in\GL(n,q)\setminus U\), and define
\[
U_\eta'=\{A\in U_\eta:\dim\ker(A-B)\neq t-1\}.
\]
Put
\[
u_i=w_i-Bv_i,\qquad r=\dim\langle u_1,\ldots,u_t\rangle.
\]
Then
\[
|U_\eta'|\le (1-c_0q^{-rt})|U_\eta|.
\]
\end{lem}

\begin{proof}
Throughout, \(c>0\) denotes an absolute constant that may change from line to line.

We may assume without loss of generality that \(B=I\). Indeed, the map \(A \mapsto B^{-1}A\) is a bijection from \(U_\eta\) to the determinant slice
\[
U_{\eta/\det B,0} = \{C\in \GL(n,q) : \det C=\eta/\det B,\ Cv_i=B^{-1}w_i \text{ for }1\le i\le t\}.
\]
The relevant dimensions are invariant under this change, since 
\[\dim\ker(A-B)=\dim\ker(B^{-1}A-I)\]
and 
\[
\dim\langle B^{-1}w_1-v_1,\ldots,B^{-1}w_t-v_t\rangle = \dim\langle w_1-Bv_1,\ldots,w_t-Bv_t\rangle=r.
\]
Thus, after substituting \(B^{-1}w_i\) for \(w_i\) and \(\eta/\det B\) for \(\eta\), we reduce to \(B=I\) and \(u_i=w_i-v_i\).

Define the subspaces
\[
T=\langle v_1,\ldots,v_t\rangle,\qquad
W_0=\langle w_1,\ldots,w_t\rangle, \qquad
R=\langle u_1,\ldots,u_t\rangle.
\]
By assumption \(I\notin U\), so \(r = \dim R \ge 1\). Let \(X\) be a common complement to \(T\) and \(W_0\), giving \(V=T\oplus X=W_0\oplus X\). Any \(A\in U_\eta\) takes the block form
\[
A(t+x)=A(t)+Lx+Qx \qquad(t\in T,\ x\in X),
\]
for uniquely determined \(L\in\operatorname{Hom}(X,W_0)\) and \(Q\in \GL(X)\). The condition \(A\in U_\eta\) is equivalent to \(\det Q=\delta\), where \(\delta\in\mathbb F_q^\times\) is a fixed value depending on \(\eta\), the assignment \(v_i\mapsto w_i\), and our choice of bases. Setting \(k=\dim X=n-t\), this yields
\[
|U_\eta|=q^{tk}\bigl|\{Q\in \GL(X):\det Q=\delta\}\bigr|.
\]

To analyze the dimension of \(\ker(A-I)\), write \(D_A=A-I\). Observe that \(D_Av_i=w_i-v_i=u_i\), which implies \(D_A(T)=R\). For \(x\in X\), we have \(D_Ax=Lx+(Q-I)x\). Since \(D_A(T)=R\), it is natural to measure the rank of \(D_A\) restricted to \(X\) modulo \(R\). Define \(\overline D_{A,X}:X\to V/R\) by
\[
\overline D_{A,X}(x)=Lx+(Q-I)x+R.
\]
Then \(\operatorname{rank}(D_A)=r+\operatorname{rank}(\overline D_{A,X})\). Since \(k=n-t\), the condition \(\dim\ker(A-I)=t-1\) is equivalent to \(\operatorname{rank}(D_A)=k+1\), which occurs if and only if \(\dim\ker\overline D_{A,X}=r-1\).

Let \(\pi_X:V \to X\) denote the projection onto \(X\) with kernel \(W_0\). Set \(P=\pi_X(R)\) and \(a=\dim(R\cap W_0)\). Since the kernel of \(\pi_X|_R\) is \(R\cap W_0\), we have \(\dim P=r-a\).

If \(x\in\ker\overline D_{A,X}\), then \(Lx+(Q-I)x\in R\). Projecting onto \(X\) gives \((Q-I)x\in P\). Thus, for a fixed \(Q\), any vector in the kernel of \(\overline D_{A,X}\) must lie in the subspace
\[
H=(Q-I)^{-1}P = \{x\in X:(Q-I)x\in P\}.
\]
Since \(n\ge 3t\), all relevant spaces have dimension at least \(2\), so we may safely apply Lemma~\ref{lem:fixed-det-derangements}. We consider two cases depending on \(a\).

\medskip
\noindent\textbf{Case 1: \(a=0\).}
Here \(R\cap W_0=\{0\}\) and \(\dim P=r\). By Lemma~\ref{lem:fixed-det-derangements}, a positive proportion of matrices \(Q\in \GL(X)\) with \(\det Q=\delta\) are derangements, meaning \(Q-I\) is invertible. Fix such a \(Q\). The candidate space \(H=(Q-I)^{-1}P\) has dimension \(r\). The map \(L\) must be chosen to reduce the kernel dimension by one.

Since \(\pi_X|_R:R\to P\) is an isomorphism, each \(x\in H\) lifts to a unique \(\rho(x)\in R\) satisfying \(\pi_X(\rho(x))=(Q-I)x\). Writing \(\rho(x)=\lambda(x)+(Q-I)x\) with \(\lambda(x)\in W_0\), we obtain a linear map \(\lambda:H\to W_0\). For \(x\in H\), \(x\in\ker\overline D_{A,X}\) if and only if \(Lx=\lambda(x)\). Because \(\ker\overline D_{A,X}\subseteq H\), it follows that
\[
\ker\overline D_{A,X}=\ker(L|_H-\lambda).
\]
We therefore need \(L|_H-\lambda:H\to W_0\) to have rank \(1\). Because \(\dim H=r\) and \(\dim W_0=t\), the number of rank-one maps is
\[
\frac{(q^r-1)(q^t-1)}{q-1}\ge cq^{r+t-1}.
\]
As \(|\operatorname{Hom}(H,W_0)|=q^{rt}\), a proportion of at least \(cq^{-(r-1)(t-1)}\) choices for \(L|_H\) result in a rank of \(1\). The restriction map \(\operatorname{Hom}(X,W_0)\to \operatorname{Hom}(H,W_0)\) has equal-size fibers, so the same proportion holds for the choices of \(L\). Multiplying by the proportion of admissible \(Q\)'s given by Lemma~\ref{lem:fixed-det-derangements}, we find at least \(c q^{-(r-1)(t-1)}|U_\eta|\) matrices \(A\in U_\eta\) with \(\dim\ker(A-I)=t-1\).

\medskip
\noindent\textbf{Case 2: \(a\ge 1\).}
Set \(s=a-1\) and \(p=\dim P=r-a\). Then \(s+p=r-1\). Here, we bypass the need for \(Q-I_X\) to be invertible by directly choosing \(Q\) so that the candidate space \((Q-I_X)^{-1}P\) has the desired dimension \(r-1\).

First, let \(K\le X\) be an \(s\)-dimensional subspace with \(K\cap P=\{0\}\). Fixing a complement \(Y\) to \(P\) in \(X\), any graph \(K=\{y+\phi(y):y\in K_0\}\) of a linear map \(\phi:K_0\to P\) from an \(s\)-dimensional subspace \(K_0\le Y\) is disjoint from \(P\). The number of such choices for \(K\) is bounded below by
\[
q^{sp}\begin{bmatrix} k-p\\ s \end{bmatrix}_q \ge c q^{sp}q^{s(k-p-s)} = c q^{s(k-s)}.
\]
Fix \(K\), and let \(C\) be a complement of \(K\) in \(X\) containing \(P\), so that \(X=K\oplus C\). Since \(\dim C=k-s\ge n-2t+1\ge 2\), Lemma~\ref{lem:fixed-det-derangements} yields a positive proportion of matrices \(Q_C\in \GL(C)\) satisfying \(\det Q_C=\delta\) with \(Q_C-I_C\) invertible. 

Let \(C_P=(Q_C-I_C)^{-1}P\), which is a \(p\)-dimensional subspace. Choose a linear map \(M:C\to K\) that vanishes on \(C_P\), and define \(Q\in\GL(X)\) by
\[
Q(y+z)=y+Mz+Q_Cz \qquad(y\in K,\ z\in C).
\]
Then \(Q\in\GL(X)\), \(\det Q=\det Q_C=\delta\), and \((Q-I_X)(y+z)=Mz+(Q_C-I_C)z\). Because \(Q_C-I_C\) is invertible and \(X=K\oplus C\), the kernel of \(Q-I_X\) is exactly \(K\). 

A vector \(y+z\) belongs to \((Q-I_X)^{-1}P\) if and only if \(Mz+(Q_C-I_C)z\in P\le C\). This requires the \(K\)-component \(Mz\) to vanish, meaning \(z\in C_P\), and gives
\[
(Q-I_X)^{-1}P=K\oplus C_P.
\]
The dimension of this preimage is therefore \(s+p=r-1\). Multiplying the number of choices for \(K\), \(Q_C\), and \(M\), and applying the standard bounds \(cq^{d^2-1}\le |\{A\in \GL_d(q):\det A=\delta\}| \le q^{d^2-1}\), the number of valid \(Q\)'s is at least
\[
cq^{-(a-1)(r-1)} \bigl|\{Q\in \GL(X):\det Q=\delta\}\bigr|.
\]

Set \(H=(Q-I_X)^{-1}P\), which has dimension \(r-1\). Choose a linear section \(\sigma:P\to R\) for the projection \(\pi_X|_R\), meaning $\pi_X|_R\circ \sigma = \Id$. For \(x\in H\), we uniquely write \(\sigma((Q-I_X)x)=\lambda(x)+(Q-I_X)x\) for some \(\lambda(x)\in W_0\), defining a linear map \(\lambda:H\to W_0\). The condition \(Lx+(Q-I_X)x\in R\) for all \(x\in H\) is then equivalent to \(L|_H-\lambda\in \operatorname{Hom}(H,R\cap W_0)\). 
The number of valid restrictions \(L|_H\) is thus \(|\operatorname{Hom}(H,R\cap W_0)| = q^{a(r-1)}\). The values of \(L\) on a complement to \(H\) in \(X\) are unconstrained, leaving exactly
\(
q^{tk-(t-a)(r-1)}
\)
choices for \(L\). 

For these choices, every \(x\in H\) lies in \(\ker\overline D_{A,X}\). Conversely, if \(x\notin H\), then \((Q-I_X)x\notin P=\pi_X(R)\), meaning \(Lx+(Q-I_X)x\notin R\) for any choice of \(Lx\). Hence \(\ker\overline D_{A,X}=H\), giving \(\dim\ker\overline D_{A,X}=r-1\). Combining the counts for \(Q\) and \(L\) yields at least
\[
cq^{-(a-1)(r-1)}q^{-(t-a)(r-1)}|U_\eta| = cq^{-(r-1)(t-1)}|U_\eta|
\]
matrices \(A\in U_\eta\) with \(\dim\ker(A-I)=t-1\).

\medskip
In both cases, we find that
\[
\bigl|\{A\in U_\eta:\dim\ker(A-I)=t-1\}\bigr| \ge c q^{-(r-1)(t-1)}|U_\eta|.
\]
Since \((r-1)(t-1)\le rt\), the fraction is bounded below by \(c_0q^{-rt}\) for some absolute constant \(c_0>0\). Recalling our initial reduction \(B=I\), the set of matrices in question is exactly \(U_\eta\setminus U_\eta'\). We conclude that
\[
|U_\eta'|\le (1-c_0q^{-rt})|U_\eta|,
\]
completing the proof.
\end{proof}
\begin{proof}[Proof of Lemma~\ref{lem:outside-point-fixed-det-slice}]
Apply Lemma~\ref{lem:counting_t-1_free_with_t_umvirate-fixed-det} to the determinant slice
\(U_\eta\) and the point \(B\notin U\). It gives
\[
\bigl|\{A\in U_\eta:\dim\ker(A-B)=t-1\}\bigr|
\ge
c_0q^{-rt}|U_\eta|>0.
\]
Hence there exists \(A\in U_\eta\) such that
\(
\dim\ker(A-B)=t-1,
\)
as required.
\end{proof}

\end{document}